\documentclass[a4paper,12pt]{article}

\usepackage{t1enc}      
\usepackage[T1]{fontenc}          
\usepackage[utf8]{inputenc} 
\usepackage{tabularx}     
\usepackage{latexsym} 
\usepackage{longtable}
\usepackage{ltablex}  % to combine longtable(pagebreak) with tabularx(row-break in table)
\usepackage{booktabs}  % tables
\usepackage{array}  % array and tabular environment
\usepackage{float}   % figures and tables
\usepackage{multirow}  % tables
\usepackage{makecell}  % tabular column heads

\usepackage[flushleft]{threeparttable}
\usepackage{dcolumn}
\newcolumntype{.}{D{.}{.}{-1}}
\newcolumntype{d}[1]{D{.}{\cdot}{#1}}
\usepackage{psfrag}
\usepackage{epstopdf}
\usepackage{graphics}     
\usepackage{graphicx}
\usepackage{pgfplots}
\usepackage{textcomp}  % text smbols like bullets, copyright etc
\usepackage{pdfpages}  % including pdfs
\usepackage{tikz,fullpage}
\usepackage{latexsym}
\usetikzlibrary{arrows,petri,topaths,backgrounds,snakes,patterns,positioning}
\usetikzlibrary{pgfplots.groupplots}
\usepackage{fontawesome}
\definecolor{mygray}{rgb}{0.9,0.9,0.9}

\usepackage{fix-cm}   
\usepackage{geometry}
\usepackage{fullpage}  % adjust margins
\usepackage[titles]{tocloft}  % table of contents, figures
\usepackage[nottoc]{tocbibind}
\usepackage[raggedright]{titlesec}  % section titles
\usepackage[titletoc,title]{appendix}
\usepackage{rotfloat}  % rotate floats
\usepackage[justification=centering, labelfont=bf]{caption}  % captions for floats
\usepackage{enumitem}
\newlist{enumerate_noindent_newline}{enumerate*}{1}
\setlist[enumerate_noindent_newline]{label=(\roman*),itemjoin=\\}
\newlist{enumerate_noindent_nonewline}{enumerate*}{1}
\setlist[enumerate_noindent_nonewline]{label=(\roman*),itemjoin=\,}
\usepackage{relsize}	   
\usepackage{anysize}
\usepackage{fancyhdr} 
\fancypagestyle{myplain}
{
  \fancyhf{}

  \fancyfoot[C]{\thepage}
}
\titleformat{\section}
  {\normalfont\Large\bfseries}{\thesection.}{1em}{}
  \titleformat{\subsection}
  {\normalfont\large\bfseries}{\thesubsection.}{1em}{} 
 
\usepackage[hang, flushmargin]{footmisc}

\usepackage{pdflscape}

\usepackage{authblk}
\usepackage[plainpages=false,breaklinks=true, citecolor=blue, colorlinks=true, linkcolor=magenta, urlcolor=blue,breaklinks]{hyperref} 
\usepackage{breakcites}
\usepackage[round]{natbib}
\usepackage{amsmath}     
\usepackage{amssymb}
\usepackage{bm}
\usepackage{bbm}
\usepackage{mathrsfs}
\usepackage{mathtools}
\usepackage{rotating}     
\usepackage{amsthm}
\usepackage{etoolbox}
\usepackage{thmtools}
\newtheorem{assumption}{Assumption}
\newtheorem{theorem}{Theorem}
\newtheorem{lemma}{Lemma}
\newtheorem{proposition}{Proposition}
\newtheorem{corollary}{Corollary}

\newtheorem{example}{Example}

\newtheorem{remark}{Remark}

\usepackage{cleveref}
\crefname{assumption}{Assumption}{Assumptions}
\Crefname{assumption}{Assumption}{Assumptions}
\crefname{example}{Example}{Examples}
\Crefname{example}{Example}{Examples}
\crefname{condition}{Condition}{Conditions}
\Crefname{condition}{Condition}{Conditions}
\crefname{theorem}{Theorem}{Theorems}
\Crefname{theorem}{Theorem}{Theorems}
\crefname{proposition}{Proposition}{Propositions}
\Crefname{proposition}{Proposition}{Propositions}
\crefname{corollary}{Corollary}{Corollaries}
\Crefname{corollary}{Corollary}{Corollaries}
\DeclareMathOperator*{\Var}{\mathbb{V}}
\newcommand{\E}{\mathbb{E}}
\newcommand{\bbP}{\mathbb{P}}

\newcommand{\prob}[1]{\mathbb{P}\left(#1\right)}
\newcommand{\probc}[2]{\prob{\left. #1\,\right\vert\, #2}}
\newcommand{\probcs}[3]{\mathbb{P}_{#1}\left(\left. #2\,\right\vert\, #3\right)}
\newcommand{\Eb}[1]{\E\left[#1\right]}
\newcommand{\Ebc}[2]{\Eb{\left. #1\,\right\vert\, #2}}
\newcommand{\Ebcs}[3]{{\E}_{#1}\left[\left. #2\,\right\vert\, #3\right]}
\newcommand{\Vb}[1]{\Var\left[#1\right]}
\newcommand{\Vbc}[2]{\Var\left[\left. #1\,\right\vert\, #2\right]}
\newcommand{\os}[2]{\overset{#1}{#2}}
\makeatletter
\newcommand{\mypm}{\mathbin{\mathpalette\@mypm\relax}}
\newcommand{\@mypm}[2]{\ooalign{%
  \raisebox{.1\height}{$#1+$}\cr
  \smash{\raisebox{-.6\height}{$#1-$}}\cr}}
\makeatother
\makeatletter
\newcommand{\mymp}{\mathbin{\mathpalette\@mymp\relax}}
\newcommand{\@mymp}[2]{\ooalign{%
  \raisebox{.5\height}{$#1-$}\cr
  \smash{\raisebox{-.2\height}{$#1+$}}\cr}}
\makeatother
\makeatletter
\newcommand{\pushright}[1]{\ifmeasuring@#1\else\omit\hfill$\displaystyle#1$\fi\ignorespaces}
\newcommand{\pushleft}[1]{\ifmeasuring@#1\else\omit$\displaystyle#1$\hfill\fi\ignorespaces}
\makeatother
\newcommand{\tld}[1]{\tilde{#1}}
\newcommand{\fr}[2]{\frac{#1}{#2}}
\newcommand{\ba}[1]{\bar{#1}}
\newcommand{\norm}[2]{\lVert{#2}\rVert_{#1}}
\newcommand{\normu}[1]{\lVert{#1}\rVert}
\newcommand{\supnorm}[1]{\left\lVert{#1}\right\rVert_{\infty}}
\newcommand{\deriv}{\mathrm{d}}
\newcommand{\Deriv}{\mathrm{D}}
\newcommand{\bigO}[1]{O\left(#1\right)}
\newcommand{\bigOP}[1]{O_P\left(#1\right)}
\newcommand{\smallO}[1]{o\left(#1\right)}
\newcommand{\smallOP}[1]{o_P\left(#1\right)}
\newcommand{\indep}{\rotatebox[origin=c]{90}{$\models$}}
\renewcommand{\indep}{\!\perp\!\!\!\perp}

\newcommand{\rsquig}{\rightsquigarrow}
\newcommand{\convprob}{\overset{P}{\longrightarrow}}
\newcommand{\convas}{\overset{a.s.}{\longrightarrow}}
\newcommand{\ceq}{\coloneqq}  % requires \usepackege{mathtools}
\newcommand{\eqc}{\eqqcolon} 
\newcommand{\indic}[1]{\mathbbm{1}_{#1}}
\newcommand{\indichat}[1]{\hat{\mathbbm{1}}_{#1}}
\newcommand{\abs}[1]{\bigg\vert #1 \bigg\vert}
\newcommand{\filt}{\mathcal{F}}
\newcommand{\set}[1]{\left\{#1\right\}}
\newcommand{\sumn}[1]{\sum_{#1\in[n]}}
\newcommand{\nsumn}[1]{\fr{1}{n}\sum_{#1\in[n]}}
\newcommand{\gproc}{\mathbb{G}}
\newcommand{\real}{\mathbb{R}}

\newcommand{\dummy}{\{0,1\}}
\newcommand{\bin}{\mathbb{B}}

\newcommand{\taut}{\tau_{\mathrm{t}}}
\newcommand{\calJ}{\mathcal{J}}
\newcommand{\calJC}{\calJ}

\newcommand{\hatth}{\hat\theta}
\newcommand{\thatpi}{\hat\tau_\pi}
\newcommand{\tthatpi}{\hat{\tau}_{\mathrm{t},\pi}}
\newcommand{\thatpihat}{\hat\tau_{\hat\pi}}
\newcommand{\tthatpihat}{\hat\tau_{\mathrm{t},\hat\pi}}
\newcommand{\hdotk}{\dot{h}_{k}}
\newcommand{\qdotk}{\dot{q}_{\mu,k}}
\newcommand{\hdotkhat}{\hat{\dot{h}}_k}
\newcommand{\qdotkhat}{\hat{\dot{q}}_{\mu,k}}

\newcommand{\eCDO}{E^\textsf{CDO}}
\newcommand{\uphatth}{\underline p_{\hatth}}
\newcommand{\bphatth}{\bar p_{\hatth}}
\newcommand{\uthatth}{\underline t_{\hatth}}
\newcommand{\bthatth}{\bar t_{\hatth}}
\newcommand{\maxminDelta}{\overline{\underline{\Delta}}_n}
\newcommand{\maxminDeltahat}{\widehat{\overline{\underline{\Delta}}}_n}
\newcommand{\minDelta}{\underline{\Delta}}
\newcommand{\Xdensity}{\Psi}
\newcommand{\ginv}{g^{-1}}
\newcommand{\Xsupp}{\mathcal{X}}

\newcommand{\Pisupp}{[\underline p, \bar p]}
\newcommand{\ntheta}{\textup{Nb}(\theta_0,\epsilon)}
\newcommand{\caliperDelta}{\maxminDelta\vee \fr{\log N_0}{N_0+1} \vee \fr{\log N_1}{N_1+1}}
\newcommand{\caliperDeltahat}{\maxminDeltahat\vee \fr{\log N_0}{N_0+1} \vee \fr{\log N_1}{N_1+1}}
\newcommand{\caliperlognn}{s\fr{\log n}{n}}
\newcommand{\Dvector}{D^{(n)}}
\newcommand{\Pvector}{PS^{(n)}}
\newcommand{\ntruncate}{\hat N}

\begin{document}

\title{Asymptotics of Caliper Matching Estimators for Average Treatment Effects}
\author[1]{M\'at\'e Kormos}
\author[2]{St\'ephanie van der Pas}
\author[1]{Aad van der Vaart}

\affil[1]{\footnotesize{Delft Institute of Applied Mathematics, Delft University of Technology, Mekelweg 4, Delft, The Netherlands}}
\affil[2]{\footnotesize{Epidemiology and Data Science, Amsterdam UMC location Vrije Universiteit Amsterdam, De Boelelaan 1117, Amsterdam, The Netherlands and Amsterdam Public Health Methodology, Amsterdam, The Netherlands}}
\date{ }

\begin{titlepage}
\maketitle
\thispagestyle{empty}

%\author[id=au1,addressref={add1}]{\fnms{M\'at\'e}~\snm{Kormos}\ead[label=e1]{m.kormos@tudelft.nl}}
%\author[id=au2,addressref={add21,add22}]{\fnms{St\'ephanie}~\snm{van der Pas}\ead[label=e2]{s.l.vanderpas@amsterdamumc.nl}}
%\author[id=au3,addressref={add1}]{\fnms{Aad}~\snm{van der Vaart}\ead[label=e3]{a.w.vandervaart@tudelft.nl}}
%
%
%\address[id=add1]{%
%\orgdiv{Delft Institute of Applied Mathematics},
%\orgname{Delft University of Technology, Mekelweg 4, Delft, The Netherlands}}
%
%\address[id=add21]{%
%\orgdiv{Epidemiology and Data Science},
%\orgname{Amsterdam UMC location Vrije Universiteit Amsterdam, De Boelelaan 1117, Amsterdam, The Netherlands}}
%
%\address[id=add22]{%
%\orgdiv{Amsterdam Public Health Methodology},
%\orgname{Amsterdam, The Netherlands}}

%
\begin{abstract}
Caliper matching is used to estimate causal effects of a binary treatment from observational data by comparing matched treated and control units. Units are matched when their propensity scores, the conditional probability of receiving treatment given pretreatment covariates, are within a certain distance called caliper.  So far, theoretical results on caliper matching are lacking, leaving practitioners with ad-hoc caliper choices and inference procedures. We bridge this gap by proposing a caliper that balances the quality and the number of matches. We prove that the resulting estimator of the average treatment effect, and average treatment effect on the treated, is asymptotically unbiased and normal at parametric rate. 
We describe the conditions under which semiparametric efficiency is obtainable, and show that when the parametric propensity score is estimated, the variance is increased for both estimands. Finally, we construct asymptotic confidence intervals for the two estimands.
\end{abstract}

\vfill

\scriptsize{\noindent M\'{a}t\'{e} Kormos thanks \href{https://www.hongdeng-hd.com}{Hong Deng} for her help and \href{https://sites.google.com/view/lisavoois}{Lisa Voois} for suggestions (both of them are from Tinbergen Institute and Erasmus School of Economics, Erasmus University Rotterdam). St\'{e}phanie van der Pas acknowledges that the research leading to these results has received funding from the Dutch Research Council (NWO) under grant VI.Veni.192.087. This research was also partially funded by a Spinoza grant from NWO.}

\end{titlepage}

%%%%%%%%%%%%%%%%%%%%%%%%%%%%%%%%%%%%%%%%%%%%%%%%%%%%%%%%%%%%%%%%%%%%%%%%%
%%%% Main text entry area:
%%%%%%%%%%%%%%%%%%%%%%%%%%%%%%%%%%%%%%%%%%%%%%%%%%%%%%%%%%%%%%%%%%%%%%%%%

%%%%%%%%%%%%%%%%%%%%%%%%%%%%%%%%%%%%%%%%%%%%%%%%%%%%%%%%%
%%%%%%%%%%%%%%%%%%%%%%%%%%%%%%%%%%%%%%%%%%%%%%%%%%%%%%%%%

%%%%%%%%%%%%%%%%%%%% 		INTRODUCTION		%%%%%%%%%%%%%%%%%%%%%%

%%%%%%%%%%%%%%%%%%%%%%%%%%%%%%%%%%%%%%%%%%%%%%%%%%%%%%%%%
%%%%%%%%%%%%%%%%%%%%%%%%%%%%%%%%%%%%%%%%%%%%%%%%%%%%%%%%%

\pagestyle{fancy}
\rhead{Kormos, Van der Pas, Van der Vaart}
\lhead{Caliper Matching}

% ------------- TODO: more up to date references for applications; stress contribution to A&I2016; what does efficiency means for practicioners;   

\section{Introduction}
% Intro
\noindent Matching is applied in empirical studies to estimate the causal effect of a binary treatment from observational data. The estimate is the mean difference in the outcome of interest of matched treated and control units. Matches may be formed in various ways. We consider matching on the propensity score, the conditional probability of receiving treatment given the observed pretreatment covariates \citep{rosenbaum_central_1983}. Specifically, we consider caliper matching, where a treated and a control unit are matched if their propensity scores are within a certain distance called \emph{caliper} \citep{cochran_controlling_1973,dehejia_propensity_1998}. Caliper matching is applied in empirical research such as labour \citep{dehejia_propensity_2002,huber_workplace_2015} and health economics \citep{erhardt_microfinance_2017,salmasi_immigration_2015,keng_effect_2013}, policy evaluation \citep{bannor_adoption_2020,patel-campillo_breaking_2022}, business and finance \citep{shen_ambition_2009,heese_is_2017} as well as healthcare \citep{capogrossi_influence_2017,cho_effects_2018,vecchio_effect_2018,izudi_does_2019,wang_mental_2020,brenna_should_2021,krishnamoorthy_impact_2022}. Nonetheless, no rigorous results have been established on the choice of the caliper and the limiting distribution of the estimator.

% Contribution
Our contribution is a theory driven caliper choice, the derivation of the asymptotic distribution of the caliper matching estimator based on propensity scores, and the construction of asymptotic confidence intervals. We consider the estimation of the Average Treatment Effect (ATE) and the Average Treatment Effect on the Treated (ATT). We show that when the order of the caliper decreases at the right speed as the sample size $n$ increases, the estimators of both estimands are asymptotically unbiased and normal at $\sqrt{n}$-rate, even when the parametric propensity score is estimated. In the rest of this section, we situate our contribution in the literature.

% Literature review
Matching has attracted much attention in the literature, with the idea of comparing similar units dating back to at least \cite{densen_studies_1952}; see \cite{cochran_matching_1953}. \cite{cochran_controlling_1973} review then-available matching methods applicable to observational studies. The reader is referred to  \cite{rubin_matched_2006} for a collection of historical results and to \cite{stuart_matching_2010} for a comprehensive survey. \cite{abadie_large_2006} present a key result closely related to ours. They study nearest neighbor matching, where the $M$ closest units in terms of covariates are matched to a given unit. They show that nearest neighbor matching on covariates is asymptotically normal, but unbiased only when we match on a scalar variable, such as the propensity score. Providing the identification results for unbiasedness, the foundations of propensity score matching is laid down by \cite{rosenbaum_central_1983}. \cite{abadie_matching_2016} derive some asymptotic properties of nearest neighbor matching on the estimated parametric propensity score. They discretise the maximum likelihood estimator of the propensity score parameter and show that the resulting matching estimator converges to a normal distribution as, first, the sample size increases and, \emph{then}, the discretisation gets finer. Since their approach changes the estimator, this asymptotic result is not equivalent to the asymptotic normality of nearest neighbor matching on the estimated parametric propensity score. In contrast, we do not change the estimator, nor do we appeal to discretisation arguments and double limits. Employing sample-splitting to estimate the propensity score, we establish the asymptotic normality of caliper matching on the estimated parametric propensity score as the sample size increases. Consequently, we are able to construct confidence intervals for ATE and ATT, centered at the caliper matching estimator based on the estimated propensity scores, which get more reliable as the sample size increases.

The first mention of caliper matching appears to be in \cite{cochran_controlling_1973}. Therein, it is analysed for a few specific models and is compared with other matching methods, such as nearest neighbor. The caliper is chosen based on the variances of the outcome in the treatment and control group. \cite{rosenbaum_constructing_1985} seem to be the first to consider caliper matching involving the propensity score as well as the covariates. They assume a logistic model for the propensity score, and match on the logit of the propensity score, that is, a linear function of the covariates. They choose the caliper based on the variances of the logit in the treatment and control group. The caliper choices of \cite{cochran_controlling_1973} and \cite{rosenbaum_constructing_1985} may lead to a large enough number of matches to reduce the variance of the caliper matching estimator. However, they do not make the bias of the caliper matching estimator converge to zero --- unless the caliper is used in combination with nearest neighbor matching; see next paragraph ---, for that the caliper needs to shrink with the sample size as we show in our present work.

Some authors, including \cite{rosenbaum_constructing_1985}, use the term caliper matching to refer to nearest neighbor matching with a caliper restriction: the $M$ nearest units are to be matched, but only if they are within the caliper. Others, for instance \cite{dehejia_propensity_1998}, use the term to mean that all units within the caliper are matched, even though they may be differently weighted.\footnote{The two interpretations coincide when $M$ is taken to be, for example, $n$ in the caliper restriction case. As $M$ is usually set to a constant independent of $n$, it is reasonable to distinguish the two interpretations.} We adopt the latter approach with uniform weights, sometimes also called radius matching \citep{huber_radius_2015}, because of its simplicity. Caliper matching can then be regarded as a kernel matching method with rectangular kernel and the bandwidth equal to the caliper. As such, the seminal work of \cite{heckman_matching_1998}, establishing the asymptotic normality of the kernel matching estimator of ATT even for nonparametrically estimated propensity score --- with bandwidth choice further investigated by \cite{frolich_matching_2005} ---, is closely related to our work. However, their results do not apply to caliper matching because they require the kernel to be Lipschitz continuous. The rectangular kernel fails to be so, prohibiting the asymptotic linear expansion of the kernel matching estimator, which is key to their argument. The work of \cite{lee_efficient_2018} is similar in spirit. It extends \cite{heckman_matching_1998} to a richer set of estimands beyond average effects using kernel matching methods, but also assuming a smooth kernel, excluding the rectangular one of caliper matching.

We overcome the nonsmoothness of the rectangular kernel by employing empirical process theory in \cite{alexander_rates_1987} and \Citet[]{van_der_vaart_weak_1996}. Writing the number of matches in terms of empirical measures enables us to characterise the asymptotic behaviour of caliper matching using ratio and tail bounds for empirical measures and processes. Furthermore, we can establish the efficiency properties of caliper matching. More efficient estimators have smaller variance and thus yield narrower confidence intervals. The efficiency of caliper matching depends on the estimand, the observed sample, the regression of the outcome on the covariates, and the knowledge of the propensity score.

First, we consider the case when the propensity score is known. We prove that if we only observe the propensity scores in our sample but not the covariates, or the regression of the outcome on the covariates only depends on the covariates through the propensity score, then the limiting variance of the caliper matching estimator of (i) ATE reaches the semiparametric lower bound; (ii) ATT reaches the semiparametric lower bound for \textit{unknown} propensity score \citep{hahn_role_1998}. The latter is not the best possible result as the lower bound for ATT, unlike ATE, is smaller when the propensity score is known \citep{hahn_role_1998}. Yet, we show that caliper matching is more efficient than nearest neighbor matching on the propensity scores studied by \cite{abadie_large_2006,abadie_matching_2016}, yielding narrower confidence intervals for ATE as well as ATT --- regardless of whether we observe the covariates in the sample or whether the outcome regression depends on the covariates or the propensity scores.

 Second, if the propensity score is unknown, but we assume and estimate a parametric specification such as the logit or probit model, then the limiting variance of the caliper matching estimator of both estimands is in general larger compared to when the propensity score is known. Consequently, it remains unclear whether the caliper or the nearest neighbor matching \citep{abadie_matching_2016} on the estimated  propensity scores  is more efficient. 
 
 % Assumptions
Our assumptions include the usual common support for the propensity score, and smoothness conditions for the conditional moments of the outcome and for (the density of) the propensity score. We verify our assumptions for a logit or probit model for the propensity score and for smooth, potentially nonlinear and heteroskedastic, regression of the outcome on the covariates with a well-behaved density on a compact support.

The rest of the paper is organised as follows. In \Cref{sec:preliminaries}, we introduce the conceptual framework and the caliper matching estimator. \Cref{sec:asymptotics} contains our contributions, the caliper choice and the asymptotic properties of the estimator. \Cref{sec:conclusion} concludes.

%%%%%%%%%%%%%%%%%%%%%%%%%%%%%%%%%%%%%%%%%%%%%%%%%%%%%%%%%
%%%%%%%%%%%%%%%%%%%%%%%%%%%%%%%%%%%%%%%%%%%%%%%%%%%%%%%%%

%%%%%%%%%%%%%%%%%%%% 		PRELIMINARIES		%%%%%%%%%%%%%%%%%%%%%%

%%%%%%%%%%%%%%%%%%%%%%%%%%%%%%%%%%%%%%%%%%%%%%%%%%%%%%%%%
%%%%%%%%%%%%%%%%%%%%%%%%%%%%%%%%%%%%%%%%%%%%%%%%%%%%%%%%%
\section{Preliminaries}
\label{sec:preliminaries}

%%%%%%%%%%%%%%%%%%%% 		Framework		%%%%%%%%%%%%%%%%%%%%%%
\subsection{Framework}
\label{sec:preliminaries:subsec:framework}

We adopt the potential outcome framework of \cite{neyman_jerzy_applications_1924} and \cite{rubin_estimating_1974} with no interference between the units (stable unit-treatment value assumption, \cite{rosenbaum_central_1983}). Let $D$ be the treatment indicator with value one corresponding to treatment and zero to control. The real-valued $Y^1,Y^0$ are the potential outcomes under treatment and control, respectively. We observe exactly one of $Y^1$ and $Y^0$, depending on $D$, so that the observed outcome is $Y=DY^1+(1-D)Y^0$. The estimands of interest, ATE and ATT, are defined respectively as
\begin{align*}
\tau \ceq \Eb{Y^1-Y^0}, \quad
\taut \ceq \Ebc{Y^1-Y^0}{D=1}.
\end{align*}
To identify ATE and ATT from observational data, we assume that the observed pretreatment covariates $X$, taking values in $\Xsupp\subset\real^K$, account for all the systematic differences between treated and control units. Formally, the potential outcomes are assumed to be independent of the treatment participation given the covariates, which is a standard assumption of causal inference \citep{rubin_estimating_1974}.

\begin{assumption}[Unconfoundedness]
\label{ass:ucf}
$Y^0\indep D\mid X$ and $Y^1\indep D\mid X$.
\end{assumption} 

Let $\pi(x)\ceq\probc{D=1}{X=x}$ be the propensity score with conditional distribution function $F_d(p)\ceq \probc{\pi(X)\leq p}{D=d}$. The $F_d$ are assumed to satisfy \Cref{ass:propscoredist}.

\begin{assumption}[Propensity Score Distribution]
\label{ass:propscoredist}
\begin{enumerate_noindent_nonewline}
\item $F_0,F_1$ admit densities $f_0, f_1$, respectively.
\item $f_0, f_1$ have the same compact support $[\underline p, \ba p]$, $0<\underline p<\ba p<1$. 
\item $f_0,f_1$ are strictly positive on their support.
\item $f_0,f_1$ are continuous on their support.
\end{enumerate_noindent_nonewline}
\end{assumption}

\noindent \Cref{ass:propscoredist} imposes the same requirements on the propensity score distribution as \cite{abadie_matching_2016}, except that it also requires the densities $f_0,f_1$ to be strictly positive. This requirement ensures that the quantile functions $F_d^{-1}$ have bounded derivatives, which we use for the caliper choice. It also plays a role in the proof of the asymptotic normality by ensuring that ratio bounds for empirical processes apply.\footnote{The strict positivity of $f_d$ implies that $\inf_{p\in[\underline p,\ba p]}\int_{p-\delta}^{p+\delta}f_d(\tld p)\deriv \tld p\gtrsim \delta>0$, so that the denominator in the ratios of empirical to true measures is bounded away from zero, keeping the ratios finite.}

\Cref{ass:propscoredist} implies that if there is a unit with propensity score in some region of $[0,1]$, then there is a positive probability of finding a unit from the opposite treatment group therein. This ensures that treated and control units can be compared in terms of their propensity scores. In combination with \cref{ass:ucf}, this yields the identification of the estimands from observed variables, by comparing treated and control units with the same propensity scores \citep{rosenbaum_central_1983}:
\begin{align}
\tau &=\Eb{\Eb{Y\mid D=1, \pi(X)}-\Eb{Y\mid D=0, \pi(X)}}, \label{eq:ident:ate} \\
\taut &=\Ebc{\Eb{Y\mid D=1, \pi(X)}-\Eb{Y\mid D=0, \pi(X)}}{D=1}. \label{eq:ident:att}
\end{align}

%%%%%%%%%%%%%%%%%%%% 		Estimator		%%%%%%%%%%%%%%%%%%%%%%
\subsection{Caliper Matching Estimator}
\label{sec:preliminaries:subsec:estimator}

We wish to construct estimators based on identification formulae \eqref{eq:ident:ate} and \eqref{eq:ident:att} from an independently and identically distributed (i.i.d.) sample from the distribution of $(Y,D,X)$, denoted by $((Y_i,D_i,X_i))_{i\in[n]}$, where $[n]\ceq\set{1,2,\ldots,n}$. This would necessitate finding sample units with the same value of the propensity score, which is infeasible for continuously distributed propensity scores. Rather, matching estimators look for units with \emph{similar} propensity scores. The caliper matching estimator explicitly controls the extent of similarity with the caliper $\delta$, whose choice is discussed later on in \Cref{sec:asymptotics}.

Suppose for now that the propensity score is known. Given $\delta>0$, the caliper matching estimator constructs the match set $\calJC(i)\ceq\{j\in[n]:D_j\neq D_i, |\pi(X_j)-\pi(X_i)|\leq\delta\}$ of unit $i\in[n]$. Next, it estimates the missing potential outcome of the unit with the mean outcome of units in the match set. Averaging out the difference between the (estimated) potential outcomes then gives the estimate of the causal effect. Let $M_i\ceq |\calJC(i)|$ be the number of matches of unit $i\in[n]$, and  write $N_0\ceq \sumn{i} (1-D_i)$, $N_1\ceq\sumn{i} D_i$ for the number of control and treated units, respectively. The estimators of ATE and ATT are defined respectively as
\begin{align*}
\thatpi &\ceq \nsumn{i}\left[D_i\left(Y_i-\fr{1}{M_i}\sum_{j\in\calJC(i)}Y_j\right)+(1-D_i)\left(\fr{1}{M_i}\sum_{j\in\calJC(i)}Y_j - Y_i \right) \right]\indic{M_i>0}, \\
\tthatpi &\ceq \fr{1}{N_1}\sum_{i\in[n]}D_i\left(Y_i-\fr{1}{M_i}\sum_{j\in\calJC(i)}Y_j\right)\indic{M_i>0}.
\end{align*}
The indicator $\indic{M_i>0}$, being one if unit $i$ has matches and zero if not, ensures that only units that have matches are included in the estimate. 

In practice, the propensity score is usually unknown. Often, it is assumed to follow a smooth parametric model, such as logit or probit. Following \cite{abadie_matching_2016}, we also make this assumption.

\begin{assumption}[Smooth Parametric Propensity Score]
\label{ass:parametric_propscore}
\begin{enumerate_noindent_nonewline}
\item \label{ass:parametric_propscore:model} The propensity score is $\probc{D=1}{X}=\pi(X,\theta_0)$ for a parametric model $\{\pi(\cdot,\theta):\theta\in\Theta\subset\real^K\}$ with $\theta_0$ in the interior of $\Theta$.
\item \label{ass:parametric_propscore:derivative}  $\theta\mapsto \pi(x,\theta)$ is differentiable in the neighbourhood of $\theta_0$ for all $x\in\Xsupp$.
\item \label{ass:parametric_propscore:bounded_derivative} The derivative in \Cref{ass:parametric_propscore}\ref{ass:parametric_propscore:derivative}  is bounded uniformly in $x\in\Xsupp$ in the neighbourhood of $\theta_0$.
\end{enumerate_noindent_nonewline}
\end{assumption}

\noindent The caliper matching estimator is then defined by a plug-in rule. Let $\calJC_{\theta}(i)\ceq\{j\in[n]:D_j\neq D_i, |\pi(X_j,\theta)-\pi(X_i,\theta)|\leq\delta\}$ be the match set and $M_i(\theta)\ceq|\calJC_{\theta}(i)|$ its cardinality for some $\theta\in\Theta$.  For an estimator $\hatth$ of $\theta_0$, the matching estimators of ATE and ATT are, respectively,
\begin{align*}
\thatpihat \ceq& \nsumn{i}\left[D_i\left(Y_i-\fr{1}{M_i(\hatth)}\sum_{j\in\calJC_{\hatth}(i)}Y_j\right) +(1-D_i)\left(\fr{1}{M_i(\hatth)}\sum_{j\in\calJC_{\hatth}(i)}Y_j - Y_i \right) \right]\indic{M_i(\hatth)>0}, \\
\tthatpihat \ceq & \fr{1}{N_1}\sum_{i\in[n]}D_i\left(Y_i-\fr{1}{M_i(\hatth)}\sum_{j\in\calJC_{\hatth}(i)}Y_j\right)\indic{M_i(\hatth)>0}.
\end{align*}

%%%%%%%%%%%%%%%%%%%%%%%%%%%%%%%%%%%%%%%%%%%%%%%%%%%%%%%%%
%%%%%%%%%%%%%%%%%%%%%%%%%%%%%%%%%%%%%%%%%%%%%%%%%%%%%%%%%

%%%%%%%%%%%%%%%%%%%% 		ASYMPTOTICS		%%%%%%%%%%%%%%%%%%%%%%

%%%%%%%%%%%%%%%%%%%%%%%%%%%%%%%%%%%%%%%%%%%%%%%%%%%%%%%%%
%%%%%%%%%%%%%%%%%%%%%%%%%%%%%%%%%%%%%%%%%%%%%%%%%%%%%%%%%
\section{Asymptotics}
\label{sec:asymptotics}

In this section, we state our main results: the caliper choice (\Cref{sec:asymptotics:subsec:caliper}), the asymptotic normality of the caliper matching estimators of ATE and ATT for known (\Cref{sec:asymptotics:subsec:pi}) and estimated (\Cref{sec:asymptotics:subsec:pi}) propensity scores, and the variance estimation (\Cref{sec:asymptotics:subsec:variance_estimation}).

%%%%%%%%%%%%%%%%%%%% 		Caliper		%%%%%%%%%%%%%%%%%%%%%%
\subsection{Caliper Choice}
\label{sec:asymptotics:subsec:caliper}

A smaller caliper means that the propensity scores of matched treated and control units are closer, so the match quality is better. At the same time, a smaller caliper leads to fewer matches. Hence, the caliper controls directly the quality and, indirectly, the number of matches, which, in turn, govern the properties of the matching estimator. The match quality determines the bias: comparing dissimilar units threatens the identification of estimands in \eqref{eq:ident:ate} and \eqref{eq:ident:att}. The number of matches determines the bias --- by excluding units with no matches --- as well as the variance of the estimator: since the estimator involves averages over the match set, a small match set gives large variance.

Thus, the right caliper choice must balance the quality and the number of matches. As the sample size increases, we expect that under \Cref{ass:propscoredist}, we can find both treated and control units in every region of $\Pisupp$ with increasing probability. It is then reasonable to aim for finding matches for each unit in the large sample limit. If we were to set the caliper to $\maxminDelta\ceq \max_{i\in[n]}\min_{j\in[n]:D_j\neq D_i}|\pi(X_i)-\pi(X_j)|$, the largest closest distance between treated and control units, we would have at least one match for each unit. The order of $\E\maxminDelta$ can be concisely described in terms of the sample size, relying on the results of \cite{shorack_empirical_2009} on spacings (all proofs are presented in \Cref{app:sec:proofs} and in the \hyperref[app:online]{Supplement}).

\begin{proposition}[Order of Expected Largest Closest Distance]
\label{prop:maxmindistorder}
Under \Cref{ass:propscoredist}, there exist constants $0<n_0,c<\infty$ such that $\E\maxminDelta\leq c\fr{\log n}{n}$ for all $n\geq n_0$.
\end{proposition}

This suggests that the caliper choices, for $n\geq 2$,
\begin{align}
\delta\ceq\delta_n\ceq
\caliperlognn \quad\text{ or }\quad
\delta\ceq\delta_n\ceq\caliperDelta
\label{eq:caliper}
\end{align}
 for any fixed constant $s>0$ are asymptotically of the same order and large enough to guarantee matches for each unit, although the data-dependent choice $\delta_n=\caliperDelta$ can better accommodate smaller samples thus it is generally preferred. %\footnote{Taking the maximum of $\maxminDelta$ and the $\fr{\log N_d}{N_d+1}$, $d\in\dummy$, facilitates the asymptotic analysis of the estimator in \Cref{sec:asymptotics:subsec:pi}.} 
Indeed, \Cref{prop:number_of_matches} shows that, in fact, the implied number of matches is of the order $\log n$.

\begin{proposition}[Number of Matches]
\label{prop:number_of_matches}
Let the caliper satisfy \eqref{eq:caliper}. If \Cref{ass:propscoredist} holds, then there exist constants $0<c_l, c_u<\infty$ such that
$$c_l (1+\smallOP{1})\log n\leq \min_{i\in[n]}M_i \leq \max_{i\in[n]}M_i\leq c_u(1+\smallOP{1})\log n$$
as $n\to\infty$. Thus, $\prob{\min_{i\in[n]}M_i\geq 1}\to1$ as $n\to\infty$.
\end{proposition}

%%%%%%%%%%%%%%%%%%%% 		Known Propscore		%%%%%%%%%%%%%%%%%%%%%%
\subsection{Known Propensity Score}
\label{sec:asymptotics:subsec:pi}

Assume for now that the propensity score $x\mapsto\pi(x)$ is known. We derive the asymptotic distribution of caliper matching in this setting, and show that the caliper choice \eqref{eq:caliper} not only leads to a number of matches increasing in the sample size, but also to the asymptotic unbiasedness of the matching estimator.

In the following, we make a series of assumptions amounting to the asymptotic normality of caliper matching, and we prove that, for instance, the models of \Cref{ex:admissible_models} satisfy these assumptions. Popular models, including the logit and probit for the propensity score and smooth heteroskedastic outcome regressions, are all covered by \Cref{ex:admissible_models} as long as the covariates admit a well-behaved density.\footnote{For simplicity of exposition, we assume throughout the paper that $X$ does not include an intercept. The intercept can be accommodated by redefining the distributional assumptions on $X$ to refer to the nonintercept coordinates of $X$.} The condition of having $K\geq2$ continuously distributed covariates with nonzero propensity score parameters is not restrictive; for if we had only one, then matching on the propensity score and matching on the covariate would be akin.\footnote{Replacing the propensity score with the scalar covariate in \Cref{ass:propscoredist,ass:variance,ass:lip_regression} would yield a version of \Cref{prop:maxmindistorder,prop:number_of_matches,thm:asymnorm_known_propscore_ate,thm:asymnorm_known_propscore_att} with the propensity score replaced with the covariate.} Regarding other conditions of \Cref{ex:admissible_models}, $\nu_d\indep D\mid X$ implies \Cref{ass:ucf}, while differentiability of $x\mapsto \Ebc{\nu_d^2}{X=x}$ allows for smooth heteroskedastic models. 

\begin{example}[Admissible Models]
\label{ex:admissible_models}
Let $g:\real\to[0,1]$ be a strictly increasing function that is twice continuously differentiable on $\real$, with first derivative $g'$ satisfying $\sup_{t\in\real}g'(t)<\infty$. The $(K\geq2)$-dimensional covariates have density $\Xdensity$, which is strictly positive on the compact support $\Xsupp$ and continuously differentiable. The propensity score and the potential outcomes satisfy
\begin{align*}
\pi(x) &= g(\theta_0^\intercal x) \\
Y^d & = m_d(X) + \nu_d,\quad\Ebc{\nu_d}{X}=0, \quad d\in\dummy,
\end{align*}
where $\theta_0$ is in the interior of $\Theta\subset\real^K$, and it has at least two nonzero coordinates, $\Theta$ is bounded, and the $m_d$ are continuously differentiable. For all  $d\in\dummy$, $\nu_d\indep D\mid X$, the $x\mapsto \Ebc{\nu_d^r}{X=x}$, $r\in\set{2,4}$, are continuously differentiable on $\Xsupp$, and $\inf_{x\in\Xsupp}\Ebc{\nu_d^2}{X=x}>0$.
\end{example}

We can rewrite $\thatpi,\tthatpi$ as weighted averages of the outcome variable $Y$ as follows:
\begin{align*}
\thatpi & =\nsumn{i}(2D_i-1)(\indic{M_i>0}+ w_i)Y_i,\quad \tthatpi = \fr{1}{N_1}\sum_{i\in[n]}(\indic{M_i>0}D_i-(1-D_i)w_i)Y_i, \\
w_i & \ceq \sum_{j\in\calJC(i)}\fr{1}{M_j},
\end{align*}
where $M_j=0$ only if $\calJC(i)$ is empty, in which case the sum in $w_i$ is taken to be zero.\footnote{This follows from the symmetry of caliper matching: $j\in\calJC(i)$ if and only if $i\in\calJC(j)$.}

Let $\mu^d(p)\ceq\Eb{Y\mid D=d, \pi(X)=p}$ be the regression function and $\varepsilon \ceq Y-\mu^{D}(\pi(X))$ be the corresponding disturbance term with conditional variance $$\sigma_d^2(p)\ceq \Vbc{\varepsilon}{D=d, \pi(X)=p}=\Vbc{Y}{D=d, \pi(X)=p}, \quad d\in\dummy.$$ 
When we apply caliper matching to imitate \eqref{eq:ident:ate} and \eqref{eq:ident:att}, we make two approximations. First, we compare  the outcome $Y$, rather than the regression $\mu^D(\pi(X))$, of the units. The error we make in doing so is $\varepsilon$. Second, we compare units with similar, rather than the same, propensity scores. Therefore, some assumptions must be imposed on the magnitude of $\varepsilon$ and the smoothness of $\mu^d$. The magnitude of $\varepsilon$ cannot be too large, but also, for convenience, not too small either to avoid degenerate limits. \Cref{ass:variance,ass:lip_regression} are the same as Assumption 4 in \cite{abadie_large_2006}, adapted to matching on the propensity score $\pi(X)$, rather than on the covariates $X$.

\begin{assumption}[Disturbance Term]
\label{ass:variance}
\begin{enumerate_noindent_nonewline}
\item \label{ass:variance:bound} The $\sigma_d^2$ satisfy
 $\inf_{d\in\dummy, p\in[\underline p,\ba p]}\sigma_d^2(p)>0$ and $\sup_{d\in\dummy, p\in[\underline p,\ba p]}\sigma_d^2(p)<\infty.$
\item \label{ass:variance:varepsilon} $\sup_{d\in\dummy, p\in[\underline p,\ba p]}\Eb{\varepsilon^4\mid D=d, \pi(X)=p}<\infty$.
\end{enumerate_noindent_nonewline}
\end{assumption}

\begin{assumption}[Lipschitz Regression Functions]
\label{ass:lip_regression}
The $\mu^d$ are Lipschitz continuous: there exists a constant $0< L_\mu<\infty$ such that $|\mu^d(p)-\mu^d(p')|\leq L_\mu|p-p'|$ for all $p,p'\in [\underline p,\ba p]$ for all $d\in\dummy$.
\end{assumption}
\noindent  Lipschitz continuity guarantees that when the propensity scores $\pi(X_i)$ and  $\pi(X_j)$ are close, which we control with $\delta_n$, then so are $\mu^d(\pi(X_i))$ and $\mu^d(\pi(X_j))$. This is in agreement with identification formulae \eqref{eq:ident:ate} and \eqref{eq:ident:att}, leading to asymptotic unbiasedness. Similarly to \cite{abadie_large_2006}, write the ATE estimator as
\begin{align}
\thatpi  = &\, \overline{\tau(\pi(X))} + E + B,  \label{eq:thatpi_decomp} \\
\overline{\tau(\pi(X))} \ceq &\, \nsumn{i} \tau(\pi(X_i)),\quad \tau(\pi(X_i))\ceq \mu^1(\pi(X_i))-\mu^0(\pi(X_i)), \label{eq:ydecomptau} \\
E \ceq &\,\nsumn{i}E_i,\quad E_i\ceq (2D_i-1)(\indic{M_i>0}+ w_i)\varepsilon_i,\label{eq:ydecompecdo} \\
B \ceq &\,\nsumn{i}B_i, \label{eq:ydecompbcdo}  \\
 B_i\ceq&\, (2D_i-1)\fr{\indic{M_i>0}}{M_i}\sum_{j\in\calJC(i)}(\mu^{1-D_i}(\pi(X_i))-\mu^{1-D_i}(\pi(X_j))) \nonumber \\
 &\,+(2D_i-1)(\indic{M_i>0}-1)(\mu^{1-D_i}(\pi(X_i))-\mu^{D_i}(\pi(X_i))). \label{eq:ydecompbicdo}
\end{align}
The first term $\overline{\tau(\pi(X))}$ has mean $\tau$ and the second term $E$ has mean zero. After centering at $\tau$, the first two terms shall be shown to be asymptotically jointly normal and independent at $\sqrt{n}$-rate. The third term $B$ has two sources of bias. The first term in \eqref{eq:ydecompbicdo} is the bias stemming from imperfect matches. If matches were exact, this term would be zero. By \Cref{ass:lip_regression}, the magnitude of this term is $\delta_n$, hence it tends to zero even when multiplied with $\sqrt{n}$. The second term in \eqref{eq:ydecompbicdo} is due to discarding unmatched units, which may happen for the caliper choice $\delta_n=\caliperlognn$, unlike for the data-dependent choice $\delta_n=\caliperDelta$. This leads to a bias because we introduce an artificial sample selection based on $\delta_n$. If every unit had at least one match, as is the case for the data-dependent caliper choice, this term would be zero. But, as shown in \cref{prop:number_of_matches}, this happens in the large sample limit, giving the asymptotic normality of the ATE estimator $\thatpi$.

\begin{theorem}[Asymptotic Normality for Known Propensity Score (ATE)]
\label{thm:asymnorm_known_propscore_ate}
Suppose that $x\mapsto\pi(x)$ is known and the caliper $\delta_n$ satisfies \eqref{eq:caliper}. If \Cref{ass:ucf,ass:propscoredist,ass:variance,ass:lip_regression} all hold, then 
$$\sqrt{n}(\thatpi-\tau)\rsquig\mathcal{N}(0,V)\quad\text{ as $n\to\infty$},$$ where $V\ceq V_{\tau}+V_{\sigma,\pi}$ with $V_\tau\ceq \Eb{(\tau(\pi(X))-\tau)^2}$ and $V_{\sigma,\pi}\ceq\Eb{\fr{\sigma_0^2(\pi(X))}{1-\pi(X)}+\fr{\sigma_1^2(\pi(X))}{\pi(X)}}$.
\end{theorem}

\cite{abadie_large_2006} prove that nearest neighbor matching is asymptotically unbiased only when we match on a scalar covariate. Caliper matching is very much alike. If we were to match on the $K$-dimensional covariates, similar arguments show that, under regularity conditions, the bias of $\sqrt{n}(\thatpi-\tau)$ would be of the order $\sqrt{n}(\delta_n+\indic{\set{\exists i\in[n]:M_i=0}})$ and the number of matches would be of the order $n\delta_n^K$. It would then be impossible to have sufficiently good match quality and enough matches at the same time for $K\geq2$, so the bias $B$ would not vanish. Therefore, it is crucial that we match on the scalar propensity score. When we do so, the ATT estimator $\tthatpi$ is also asymptotically normal.

\begin{theorem}[Asymptotic Normality for Known Propensity Score (ATT)]
\label{thm:asymnorm_known_propscore_att}
Suppose that $x\mapsto\pi(x)$ is known and the caliper $\delta_n$ satisfies \eqref{eq:caliper}. Let $p_1\ceq\E\pi(X)$. If \Cref{ass:ucf,ass:propscoredist,ass:variance,ass:lip_regression} all hold, then 
$$\sqrt{n}(\tthatpi-\taut)\rsquig\mathcal{N}(0,V_{\mathrm{t}})\quad\text{ as $n\to\infty$},$$ where $V_{\mathrm{t}}\ceq V_{\taut}+V_{\mathrm{t},\sigma,\pi}$ with $V_{\taut}\ceq \fr{1}{p_1^2}\Eb{\pi(X)(\tau(\pi(X))-\taut)^2}$ and 
$$V_{\mathrm{t},\sigma,\pi}\ceq \fr{1}{p_1^2}\Eb{\fr{\pi(X)^2\sigma_0^2(\pi(X))}{1-\pi(X)}+\pi(X)\sigma_1^2(\pi(X))}.$$
\end{theorem}

To examine the efficiency of $\thatpi$ and $\tthatpi$, let 
\begin{align*}
\mu_\Xsupp^d(x)\ceq\Ebc{Y}{D=d,X=x} \text{ and } \sigma_{\Xsupp,d}^2(x) \ceq \Vbc{Y}{D=d,X=x}
\end{align*}
for $d\in\dummy$. The semiparametric efficiency bound of ATE is
\begin{align}
V_\mathrm{eff}\ceq \Eb{(\mu_\Xsupp^1(X)-\mu_\Xsupp^0(X)-\tau)^2+\fr{\sigma_{\Xsupp,0}^2(X)}{1-\pi(X)}+\fr{\sigma_{\Xsupp,1}^2(X)}{\pi(X)}}, \label{eq:semipara_eff_bound_ate}
\end{align}
irrespective of whether or not the propensity scores are known (\citet[Theorems 1 and 2]{hahn_role_1998}). The semiparametric efficiency bound of ATT is
\begin{align*}
V_{\mathrm{t},\mathrm{eff},\pi}\ceq \fr{1}{p_1^2}\Eb{(\mu_\Xsupp^1(X)-\mu_\Xsupp^0(X)-\taut)^2\pi(X)^2+\fr{\pi(X)^2\sigma_{\Xsupp,0}^2(X)}{1-\pi(X)}+\pi(X)\sigma_{\Xsupp,1}^2(X)} %\label{eq:semipara_eff_bound_att_knownpropscore}
\end{align*}
if the propensity scores are known, and 
\begin{align*}
V_{\mathrm{t},\mathrm{eff}}\ceq \fr{1}{p_1^2}\Eb{(\mu_\Xsupp^1(X)-\mu_\Xsupp^0(X)-\taut)^2\pi(X)+\fr{\pi(X)^2\sigma_{\Xsupp,0}^2(X)}{1-\pi(X)}+\pi(X)\sigma_{\Xsupp,1}^2(X)}  %\label{eq:semipara_eff_bound_att_unknownpropscore}
\end{align*}
if the propensity scores are unknown (\citet[Theorems 1 and 2]{hahn_role_1998}). The limiting variance $V$  of $\thatpi$ resembles the efficiency bound $V_\mathrm{eff}$, except that $V$ involves moments of the outcome conditional on the propensity score $\pi(X)$, rather than on the covariates $X$ as in $V_\mathrm{eff}$. Hence, if we were to observe only $\pi(X)$ in our sample, instead of $X$, $\thatpi$ would be semiparametrically efficient, reaching $V_\mathrm{eff}$. It is also immediate from \Cref{thm:asymnorm_known_propscore_ate} and \eqref{eq:semipara_eff_bound_ate}, that if we had $\mu_\Xsupp^d(X)=\mu^d(\pi(X))$ and $\sigma_{\Xsupp,d}^2(X)=\sigma_d^2(\pi(X))$ for all $d\in\dummy$ --- so that the conditional moments of the outcome given the covariates only depended on the propensity score ---, then too, the ATE estimator $\thatpi$ would be semiparametrically efficient. In truth, a more precise result in \Cref{prop:semipara_eff} holds.

\begin{proposition}[Semiparametric Efficiency]
\label{prop:semipara_eff}
Suppose that \Cref{ass:ucf} holds. Then $V_\mathrm{eff}\leq V$ and $V_{\mathrm{t},\mathrm{eff}}\leq V_{\mathrm{t}}$ with equality in both cases if and only if
\begin{align}
\mu_\Xsupp^D(X)=\mu^D(\pi(X))\quad \text{almost surely}. \label{cond:semipara_eff:c2}
\end{align}
\end{proposition}

Suppose that \eqref{cond:semipara_eff:c2} in \Cref{prop:semipara_eff} holds. Even then, in contrast to the ATE estimator $\thatpi$, the ATT estimator $\tthatpi$ only reaches $V_{\mathrm{t},\mathrm{eff}}$, the semiparametric efficiency bound for \emph{unknown} propensity scores, which is larger than the bound $V_{\mathrm{t},\mathrm{eff},\pi}$ for known propensity scores. The difference between them, under \eqref{cond:semipara_eff:c2}, is
\begin{align}
V_{\mathrm{t},\mathrm{eff}}-V_{\mathrm{t},\mathrm{eff},\pi}=\fr{1}{p_1^2}\Eb{\pi(X)(1-\pi(X))(\tau(\pi(X))-\taut)^2}\geq 0. \label{eq:efficiency_loss}
\end{align}
As $\pi(X)(1-\pi(X))\leq1/2$, the difference is bounded by $\fr{1}{2p_1^2}\Eb{(\tau(\pi(X))-\taut)^2}$. Thus, the more homogeneous the treatment effects are across $\pi(X)$ (equivalently, under \eqref{cond:semipara_eff:c2}, across $X$) and the treatment groups $D$, the smaller the difference is. %On the other hand, $x\mapsto\tau(x)$ is bounded under \Cref{ass:lip_regression}, so \eqref{eq:efficiency_loss} is bounded by $\fr{1}{p_1^2}\Eb{\pi(X)(1-\pi(X))}$ up to a constant factor. The closer the propensity score tends to be to one, the smaller this bound is. As the density of $\pi(X)$ is the mixture $(1-p_1)f_0+p_1f_1$, the more mass $f_0$ and $f_1$ have in the same region close to one, the smaller \eqref{eq:efficiency_loss} is. Intuitively, it is then easier to find matches for treated units, leading to larger match sets and thereby smaller variance, which mitigates the efficiency loss \eqref{eq:efficiency_loss} of ATT.

The efficiency loss \eqref{eq:efficiency_loss} is not specific to caliper matching. In fact, the limiting variance of the ATT estimator in \cref{thm:asymnorm_known_propscore_att} is lower than that of the nearest neighbor matching estimator in \citet[Proposition 1]{abadie_matching_2016}. The difference is
\begin{align}
\fr{1}{2Mp_1^2}\Eb{\sigma_0^2(\pi(X))\pi(X)\left(2+\fr{\pi(X)}{1-\pi(X)}\right)}\geq 0, \label{eq:efficiency_diff}
\end{align}
where the \emph{constant} $M$ is the number of nearest neighbors to match. This shows that the efficiency gain \eqref{eq:efficiency_diff} of caliper matching is smaller for larger $M$. However, there is no proof that letting $M$ to infinity closes the gap as the results of \cite{abadie_matching_2016} are contingent on a fixed $M$. In contrast, with the caliper choice of \Cref{thm:asymnorm_known_propscore_att}, the number of matches for caliper matching goes to infinity by \Cref{prop:number_of_matches}, thereby cutting variance. Unless $p\mapsto \sigma_0^2(p)$ decreases rapidly around one, which is ruled out by \Cref{ass:variance}\ref{ass:variance:bound}, \eqref{eq:efficiency_diff} is larger when  the propensity score tends to be close to one. In that case, we gain even more by using caliper instead of nearest neighbor matching, although then $V_{\mathrm{t},\sigma,\pi}$, and thus $V_{\mathrm{t}}$, increases too. 

We close the case for the known propensity score by verifying the assumptions of \cref{thm:asymnorm_known_propscore_ate,thm:asymnorm_known_propscore_att} for the models of \cref{ex:admissible_models}.

\begin{proposition}[Admissible Models (Known Propensity Score)]
\label{prop:admissible_models_known_propscore}
The family of models described in \cref{ex:admissible_models} satisfies all \cref{ass:ucf,ass:propscoredist,ass:lip_regression,ass:variance}.
\end{proposition}

%%%%%%%%%%%%%%%%%%%% 		Estimated Propscore		%%%%%%%%%%%%%%%%%%%%%%
\subsection{Estimated Propensity Score}
\label{sec:asymptotics:subsec:pihat}

Suppose that the propensity score $\pi(\cdot,\theta_0)$ of \Cref{ass:parametric_propscore} is estimated. A reasonable estimator of $\theta_0$ will converge to $\theta_0$. We then expect that if local versions of \Cref{ass:propscoredist,ass:variance,ass:lip_regression} hold in the neighbourhood of $\theta_0$, then the caliper matching estimators on the estimated propensity scores will also be asymptotically normal, provided they are smooth enough in $\theta$.

To this end, we require the conditional distribution $F_{d,\theta}(p)\ceq \probcs{\theta_0}{\pi(X,\theta)\leq p}{D=d}$ to resemble that of the true propensity score, but only locally. Extending \cref{ass:propscoredist}, we need that the densities $f_{0,\theta}, f_{1,\theta}$ are not only continuous but differentiable, and that they depend smoothly on $\theta$. For some arbitrary fixed constant $\epsilon>0$, let $\ntheta\ceq \left\{\theta\in\Theta: \normu{\theta-\theta_0}<\epsilon \right\}$ denote a neighbourhood of $\theta_0$, and further let $$\mathcal{S}_{\theta_0,\epsilon}\ceq\left\{(\theta,p):p\in [\underline p_{\theta}, \ba p_{\theta}], \theta\in\ntheta \right\}.$$

 \begin{assumption}[Distribution of the Parametric Propensity Score]
\label{ass:propscoredist_estimated}
\begin{enumerate_noindent_newline}
\item \label{ass:propscoredist_estimated:density} $F_{0,\theta}$, $F_{1,\theta}$ admit densities $f_{0,\theta}, f_{1,\theta}$, respectively, for all $\theta\in\ntheta$. 

\item \label{ass:propscoredist_estimated:common_supp} $f_{0,\theta}, f_{1,\theta}$ have the same support $[\underline p_\theta, \ba p_\theta]$ with $0<\underline p_\theta< \ba p_\theta<1$ for all $\theta\in\ntheta$.

\item \label{ass:propscoredist_estimated:bounded_from_zero} %$f_{0,\theta},f_{1,\theta}$ are bounded away from zero on $[\underline p_{\theta}, \ba p_{\theta}]$ for all $\theta\in\ntheta$. That is, $\inf_{\theta\in\ntheta}\inf_{p\in[\underline p_\theta, \ba p_\theta]}f_{d,\theta}(p)>0$ for all $d\in\dummy$.
$f_{0,\theta},f_{1,\theta}$ are bounded away from zero: $\inf_{\theta\in\ntheta}\inf_{p\in[\underline p_\theta, \ba p_\theta]}f_{d,\theta}(p)>0$ for all $d\in\dummy$.

%\item \label{ass:propscoredist_estimated:differentiability_p} $p\mapsto f_{d,\theta}(p)$ is continuously differentiable on $[\underline p_{\theta}, \ba p_{\theta}]$ for all $\theta\in\ntheta$ for all $d\in\dummy$. %That is, for all $(d, \theta) \in\dummy\times\ntheta$, there exists function $\fr{\deriv f_{d,\theta}}{\deriv p}:[\underline p_\theta, \ba p_\theta]\to\real$ satisfying $\fr{\deriv f_{d,\theta}}{\deriv p}(p)=\lim_{h\to 0}\fr{|f_{d,\theta}(p+h)-f_{d,\theta}(p)|}{|h|}$ and for all $\zeta>0$, there exists $\eta_\theta>0$ such that for all $p\in [\underline p_{\theta}, \ba p_{\theta}]$, $|\fr{\deriv f_{d,\theta}}{\deriv p}(p)-\fr{\deriv f_{d,\theta}}{\deriv p}(p')|\leq\zeta$ if $|p-p'|\leq \eta_\theta$ for some $p'\in [\underline p_{\theta}, \ba p_{\theta}]$.

%\item \label{ass:propscoredist_estimated:differentiability_theta} $\theta\mapsto f_{d,\theta}(p)$ is continuously differentiable  for all $p\in[\underline p_{\theta}, \ba p_{\theta}]$ at all $\theta\in\ntheta$ with derivative $\dot f_{d,\theta}:[\underline p_{\theta}, \ba p_{\theta}]\to\real^K$ satisfying $\sup_{\theta\in\ntheta}\sup_{p\in [\underline p_{\theta}, \ba p_{\theta}]}\normu{\dot f_{d,\theta}(p)}<\infty$ for all $d\in\dummy$. 

\item \label{ass:propscoredist_estimated:differentiability_p_and_theta} $(\theta,p)\mapsto f_{d,\theta}(p)$ is continuously differentiable on $\mathcal{S}_{\theta_0,\epsilon}$ for all $d\in\dummy$.

\end{enumerate_noindent_newline}
\end{assumption}

Next, we decompose the outcome in a way that depends on the propensity score parameter $\theta$. Rather than the continuity of \cref{ass:lip_regression}, we need that the regression function $\mu^d(\theta, p)\ceq\Ebc{Y}{D=d, \pi(X,\theta)=p}$ is continuously differentiable, also in $\theta$. In combination with \cref{ass:parametric_propscore}, \cref{ass:differentiability_regression} implies that $\theta\mapsto \mu^d(\theta,\pi(x,\theta))$ can be approximated in the neighbourhood of $\theta_0$ with an error of the order $\normu{\theta-\theta_0}$. Specifically, they imply that the derivative of $\theta\mapsto \mu^d(\theta,\pi(x,\theta))$ exists for all $(\tld\theta,x)\in\ntheta\times\Xsupp$ and it takes the form $\Lambda^d(\tld\theta,x)\ceq  \fr{\partial \mu^d}{\partial \theta^\intercal}(\tld\theta,\pi(x, \tld\theta))+\fr{\partial \mu^d}{\partial p}(\tld\theta,\pi(x,\tld\theta))(\Deriv_\theta\pi)(x,\tld\theta)$.

\begin{assumption}[Differentiability of Regression Functions]
\label{ass:differentiability_regression}
The $(\theta,p)\mapsto \mu^d(\theta,p)$ are continuously differentiable on $\mathcal{S}_{\theta_0,\epsilon}$ with partial derivatives $\fr{\partial \mu^d}{\partial\theta}: \Theta \times [0,1]\to\real^K$ and $\fr{\partial \mu^d}{\partial p}:\Theta\times [0,1]\to\real$ uniformly bounded on $\mathcal{S}_{\theta_0,\epsilon}$ for all $d\in\dummy$.
\end{assumption}

To ensure the smoothness, and to control the magnitude of the disturbance term $\varepsilon_i(\theta)\ceq Y_i- \mu^{D_i}(\theta,\pi(X_i,\theta)),i\in[n]$, we require that the functions
\begin{align*}
\sigma_d^r(\theta, p)&\ceq \Ebc{(Y-\mu^D(\theta,p))^r}{D=d, \pi(X,\theta)=p},\quad r\in\left\{2,4\right\},\,d\in\dummy,
%\sigma_d^2(\theta, p)&\ceq \Ebc{(Y-\mu^d(\theta,p))^2}{D=d, \pi(X,\theta)=p},\\
%\sigma_d^4(\theta, p)&\ceq \Ebc{(Y-\mu^d(\theta,p))^4}{D=d, \pi(X,\theta)=p}
\end{align*}
satisfy the following conditions.

\begin{assumption}[Smooth Parametric Disturbance Term]
\label{ass:lip_conditional_var}
\begin{enumerate_noindent_nonewline}
\item \label{ass:lip_conditional_var:lip} The $\sigma_d^2$ satisfy the Lipschitz-condition $|\sigma_d^2(\theta,p)-\sigma_d^2(\theta',p')|\leq  L_\sigma (\normu{\theta-\theta'}+|p-p'|)$ for all $p\in[\underline p_\theta,\ba p_\theta]$ and $p'\in[\underline p_{\theta'},\ba p_{\theta'}]$ for all $\theta,\theta'\in\ntheta$ for some constant $0< L_\sigma<\infty$ and the lower bound $\inf_{p\in[\underline p_{\theta_0},\ba p_{\theta_0}]}\sigma_d^2(\theta_0,p)>0$ for all $d\in\{0,1\}$.
\item \label{ass:lip_conditional_var:bound} The $\sigma_d^4$ satisfy the condition 
$\sup_{\theta\in\ntheta}\sup_{p\in[\underline p_\theta,\ba p_\theta]}\sigma_d^4(\theta, p)<\infty$ for all $d\in\{0,1\}$.
\end{enumerate_noindent_nonewline}
\end{assumption}

Finally, we need that the estimator $\hatth$ of the propensity score parameter converges to $\theta_0$ in an appropriate sense. For instance, if $\hatth$ is the maximum likelihood estimator, it converges appropriately under regularity conditions. We further assume that $\theta_0$ is estimated from a sample that is independent of $((Y_i,D_i,X_i))_{i\in[n]}$. In practice, sample splitting may be applied to ensure the independence: one can halve a $2n$-large sample and use the first half to estimate $\theta_0$, and plug the resulting estimator $\hatth$ back into the second half to compute $\thatpihat,\tthatpihat$.%\footnote{Cross-fitting as in  \cite{chernozhukov_doubledebiased_2018}, averaging two estimators from two halves of the sample to gain back efficiency, does not apply here. This is because the effect of the nuisance parameter estimation vanishes in \cite{chernozhukov_doubledebiased_2018} due to orthogonalization, in contrast to the effect of $\theta_0$ estimation in our case, which contributes to the limiting distribution of $\thatpihat,\tthatpihat$. See the proof of \Cref{thm:asymnorm_estimated_propscore_ate}.}

\begin{assumption}[Estimator of the Propensity Score Parameter]
\label{ass:theta_estimator}
\begin{enumerate_noindent_nonewline}
\item \label{ass:theta_estimator:asymnorm}  $\hatth$ is asymptotically normal with $\sqrt{n}(\hatth-\theta_0)\rsquig\mathcal{N}(0, V_{\theta_0})$ as $n\to\infty$ for a finite invertible matrix $V_{\theta_0}$.
\item \label{ass:theta_estimator:indep} $\hatth$ is independent of the data set from which the matching estimator is computed: $\hatth\indep ((Y_i, X_i,D_i))_{i\in[n]}$.
\end{enumerate_noindent_nonewline}
\end{assumption}
To accommodate the propensity score estimation, we introduce $$\maxminDeltahat\ceq\max_{i\in[n]}\min_{j\in[n]:D_j\neq D_i}|\pi(X_i,\hatth)-\pi(X_j,\hatth)|,$$ the estimated analogue of $\maxminDelta$, and the corresponding caliper choices
\begin{align}
\delta_n\ceq\caliperlognn &\quad\text{ or }\quad 
\delta_n\ceq\caliperDeltahat \label{eq:caliper_hat}
\end{align}
for any fixed constant $s>0$. \Cref{prop:number_of_matches_estimated} shows that the number of matches based on the estimated propensity scores and the caliper choice \eqref{eq:caliper_hat} is also of the order $\log n$ as in \Cref{prop:number_of_matches}. This yields \Cref{thm:asymnorm_estimated_propscore_ate,thm:asymnorm_estimated_propscore_att}, establishing the asymptotic normality of caliper matching on the estimated propensity score.

\begin{proposition}[Number of Matches for Estimated Propensity Score]
\label{prop:number_of_matches_estimated}
Suppose that the caliper $\delta_n$ satisfies \eqref{eq:caliper_hat}. If \Cref{ass:propscoredist_estimated,ass:theta_estimator} hold, then there exist constants $0<\ba c_l, \ba c_u<\infty$ such that
$$\ba c_l (1+\smallOP{1})\log n\leq \min_{i\in[n]}M_i(\hatth) \leq \max_{i\in[n]}M_i(\hatth)\leq \ba c_u (1+\smallOP{1})\log n$$
as $n\to\infty$. Thus, $\prob{\min_{i\in[n]}M_i(\hatth)\geq1}\to 1$ as $n\to\infty$.
\end{proposition}

\begin{theorem}[Asymptotic Normality for Estimated Propensity Score (ATE)]
\label{thm:asymnorm_estimated_propscore_ate}
Suppose that the caliper $\delta_n$ satisfies \eqref{eq:caliper_hat}. If \cref{ass:ucf,ass:parametric_propscore,ass:propscoredist_estimated,ass:differentiability_regression,ass:lip_conditional_var,ass:theta_estimator} all hold, then 
$$\sqrt{n}(\thatpihat-\tau)\rsquig\mathcal{N}(0,V_{\hat\pi})\quad\text{ as $n\to\infty$},$$ where $V_{\hat\pi}\ceq V_{\tau}+V_{\sigma,\pi}+(q_1-q_0)^\intercal V_{\theta_0}(q_1-q_0)$ for $V_\tau,V_{\sigma,\pi}$ of \Cref{thm:asymnorm_known_propscore_ate}, $V_{\theta_0}$ of \cref{ass:theta_estimator}, and $q_d\in\real^K$ arising as the probability limit $\nsumn{i}\Lambda^d(\hatth,X_i)\convprob q_d^\intercal$ as $n\to\infty$ for $d\in\dummy$.
 \end{theorem}

\begin{theorem}[Asymptotic Normality for Estimated Propensity Score (ATT)]
\label{thm:asymnorm_estimated_propscore_att}
Suppose that the caliper $\delta_n$ satisfies \eqref{eq:caliper_hat}. If \cref{ass:ucf,ass:parametric_propscore,ass:propscoredist_estimated,ass:differentiability_regression,ass:lip_conditional_var,ass:theta_estimator} all hold, then 
$$\sqrt{n}(\tthatpihat-\taut)\rsquig\mathcal{N}(0,V_{\mathrm{t},\hat\pi})\quad\text{ as $n\to\infty$},$$ where $V_{\mathrm{t},\hat\pi}\ceq V_{\taut}+V_{\mathrm{t},\sigma,\pi}+(1/p_1^2)(q_{\mathrm{t},1}-q_{\mathrm{t},0})^\intercal V_{\theta_0}(q_{\mathrm{t},1}-q_{\mathrm{t},0})$ for $V_{\taut}$, $V_{\mathrm{t},\sigma,\pi}$ and $p_1$ of \Cref{thm:asymnorm_known_propscore_att}, $V_{\theta_0}$ of \cref{ass:theta_estimator}, and $q_{\mathrm{t},d}\in\real^K$ arising as the probability limit $\nsumn{i}D_i\Lambda^d(\hatth,X_i)\convprob q_{\mathrm{t},d}^\intercal$ as $n\to\infty$ for $d\in\dummy$.
\end{theorem}

Compared to \Cref{thm:asymnorm_known_propscore_ate,thm:asymnorm_known_propscore_att}, the variances are increased by $(q_1-q_0)^\intercal V_{\theta_0}(q_1-q_0)$ for the ATE estimator $\thatpihat$, and by $(1/p_1^2)(q_{\mathrm{t},1}-q_{\mathrm{t},0})^\intercal V_{\theta_0}(q_{\mathrm{t},1}-q_{\mathrm{t},0})$ for the ATT estimator $\tthatpihat$, representing the uncertainty from the propensity score estimation. %\footnote{This comparison of variances is correct if and only if we only use half of a $2n$-large sample to compute $\thatpi,\tthatpi$ and the caliper \eqref{eq:caliper}, because of the sample-splitting argument employed in $\thatpi,\tthatpi$ and the caliper \eqref{eq:caliper_hat}.} 
 The more precisely we can estimate the propensity score, the smaller $V_{\theta_0}$ is, resulting in smaller differences. Alternatively, if $q_1\approx q_0$ or $q_{\mathrm{t},1}\approx q_{\mathrm{t},0}$, then the respective increments  are also small. This is the case if the derivatives $\Lambda^1$ and $\Lambda^0$ are close to each other, although it is difficult to see if and when that happens, even for simple linear regressions in \Cref{ex:admissible_models}. 
As a consequence, it remains unclear whether caliper or nearest neighbor matching \citep{abadie_matching_2016} is more efficient when the parametric propensity score is estimated. %However, in \Cref{sec:asymptotics:subsec:variance_estimation}, we provide consistent estimators for the derivatives and other variance components in $V_{\hat\pi}, V_{\mathrm{t},\hat\pi}$, so practitioners can compare the estimated variances of the two matching estimators using the variance estimator in \cite{abadie_matching_2016}. 

\begin{remark}[Variance Comparison]
The asymptotic variances of  $\thatpi,\tthatpi$ and  $\thatpihat,\tthatpihat$ are comparable as in the preceding paragraph if and only if we use only half of a $2n$-large sample to compute $\thatpi,\tthatpi$ and the caliper \eqref{eq:caliper}, because of the sample-splitting in the computation of $\thatpihat,\tthatpihat$ and \eqref{eq:caliper_hat}. If we use the whole $2n$-large sample to compute $\thatpi,\tthatpi$, \eqref{eq:caliper}, and only $n$ observations to evaluate $\thatpihat,\tthatpihat$, \eqref{eq:caliper_hat} --- with the remaining $n$ observations reserved to estimate $\theta_0$ ---, then the standard error of $\thatpi$ is $\sqrt{\fr{V_{\tau}+V_{\sigma,\pi}}{2n}}$, while that of $\thatpihat$ is $\sqrt{\fr{V_{\tau}+V_{\sigma,\pi}+(q_1-q_0)^\intercal V_{\theta_0}(q_1-q_0)}{n}}$, which is an increment by a factor of $\sqrt{2}$ even without the contribution of $(q_1-q_0)^\intercal V_{{\theta_0}}(q_1-q_0)$. The same applies to $\tthatpi$ and $\tthatpihat$. 
\end{remark}

\cite{abadie_matching_2016} account for the estimation of the propensity score by considering a shifted law of $(Y,D,X)\sim \bbP_{\theta_0}$. They assume that conditional expectations under the shifted law converge weakly to conditional expectations under the nonshifted law. We pursue a different approach. Our \Cref{ass:propscoredist_estimated,ass:differentiability_regression,ass:lip_conditional_var} do not involve shifted laws. Rather, they impose smoothness of conditional expectations in $\theta$ and may be regarded as local versions of \Cref{ass:propscoredist,ass:variance,ass:lip_regression} in the neighbourhood of $\theta_0$. Moreover, we verify the assumptions of \cref{thm:asymnorm_estimated_propscore_ate,thm:asymnorm_estimated_propscore_att} for the models in \cref{ex:admissible_models}.

\begin{proposition}[Admissible Models (Estimated Propensity Score)]
\label{prop:admissible_models_estimated_propscore}
Consider the family of models described in \cref{ex:admissible_models}, with the propensity score model $\{\pi(x,\theta)=g(\theta^\intercal x):\theta\in\Theta\}$ estimated with maximum likelihood on an independent $n$-large i.i.d.\ sample from the distribution of $(D,X)$. Then \cref{ass:ucf,ass:parametric_propscore,ass:propscoredist_estimated,ass:differentiability_regression,ass:lip_conditional_var,ass:theta_estimator} are all satisfied.
\end{proposition}

%%%%%%%%%%%%%%%%%%%% 		Variance Estimation		%%%%%%%%%%%%%%%%%%%%%%
\subsection{Variance Estimation}
\label{sec:asymptotics:subsec:variance_estimation}

In this section, we provide consistent estimators for the components of $V_{\hat\pi}$ and $V_{\mathrm{t},\hat\pi}$ so that we can construct asymptotically valid confidence intervals for ATE and ATT. To prove consistency, we impose some further assumptions, which are all in accordance with the models in \Cref{ex:admissible_models}.

Namely, we need that certain estimators are almost surely bounded, which is implied if the outcome is almost surely bounded. Furthermore, $\theta\mapsto \pi(\cdot,\theta)$ may take many forms in general, which renders $\Lambda^d$ intractable. Requiring that the propensity score follows a single-index model, such as the logit or the probit, and that the covariates have a well-behaved density, alleviates these difficulties, provided the outcome regression is smooth enough. Imposing $K\geq 2$ continuously distributed covariates and certain smoothness conditions implies that $\Lambda^d$ is expressible in a way suitable for showing consistency.

 \begin{assumption}[Outcome and Covariate Distribution]
\label{ass:covar_outcome_distribution}
\begin{enumerate_noindent_nonewline}
\item \label{ass:covar_outcome_distribution:outcome} The outcome is almost surely bounded: there exists a constant $0<\ba y<\infty$ such that $\prob{|Y|>\ba y}=0$.
\item \label{ass:covar_outcome_distribution:covariate} The covariate vector $X$ has at least $K\geq2$ coordinates, and $X$ admits a density $\Xdensity$ on the compact $\mathcal{X}$; the $\Xdensity$ is as specified in \Cref{ex:admissible_models}.
\end{enumerate_noindent_nonewline}
\end{assumption}

 \begin{assumption}[Single-Index Propensity Score and Smooth Outcome Regression]
\label{ass:singleindex_propscore_smooth_outcome}
\begin{enumerate_noindent_nonewline}
\item \label{ass:singleindex_propscore_smooth_outcome:propscore} The propensity score model of \Cref{ass:parametric_propscore} is $\pi(x,\theta)=g(\theta^\intercal x)$ for $g$ as specified in \Cref{ex:admissible_models}.
\item \label{ass:singleindex_propscore_smooth_outcome:regression} The $m(x)\ceq\Ebc{Y}{X=x}$ is bounded, and there exist two covariates --- $X_1$ and $X_2$ without loss of generality --- such that $\fr{\partial m}{\partial x_1}$ and $\fr{\partial m}{\partial x_2}$ are well-defined and continuous for all $x\in\mathcal{X}$.
\end{enumerate_noindent_nonewline}
\end{assumption}

The variance estimators are 
\begin{align}
\hat V_{\hat\pi} &\ceq \hat V_{\tau}+\hat V_{\sigma,\pi}+(\hat q_1-\hat q_0)^\intercal \hat V_{\theta_0}(\hat q_1-\hat q_0), \label{eq:vpihat_estimator} \\
\hat V_{\mathrm{t},\hat\pi}&\ceq \hat V_{\taut}+\hat V_{\mathrm{t},\sigma,\pi}+(1/\hat p_1^2)(\hat q_{\mathrm{t},1}-\hat q_{\mathrm{t},0})^\intercal \hat V_{\theta_0}(\hat q_{\mathrm{t},1}-\hat q_{\mathrm{t},0}), \label{eq:vtpihat_estimator}
\end{align}
where the component estimators are as follows. We assume that $\hat V_{\theta_0}\convprob V_{\theta_0}$ is a consistent estimator of $V_{\theta_0}$. In practice, $\hatth$ is usually the maximum likelihood estimator, as supported by \Cref{prop:admissible_models_estimated_propscore}, in which case, under \Cref{ass:singleindex_propscore_smooth_outcome},
\begin{align*}
\hat V_{\theta_0}\ceq \left(\fr{1}{n}\sum_{i\in[n]} \fr{(g'(\hatth^\intercal X_i))^2}{g(\hatth^\intercal X_i)(1-g(\hatth^\intercal X_i))}X_iX_i^\intercal  \right)^{-1}
\end{align*} 
is well-known to be consistent for $V_{\theta_0}$. The $p_1$ is consistently estimated with $\hat p_1\ceq \nsumn{i}D_i$ by the law of large numbers. The nonparametric estimators of the remaining components in \eqref{eq:vpihat_estimator} and \eqref{eq:vtpihat_estimator} are  developed in \Cref{app:sec:variance_estimation}.

\begin{proposition}[Consistent Variance Estimators]
\label{prop:consistent_variance}
Suppose that \cref{ass:ucf,ass:parametric_propscore,ass:propscoredist_estimated,ass:differentiability_regression,ass:lip_conditional_var,ass:theta_estimator,ass:covar_outcome_distribution,ass:singleindex_propscore_smooth_outcome} all hold, that the caliper $\delta_n$ satisfies \eqref{eq:caliper_hat}, and that $\hat V_{\theta_0}\convprob V_{\theta_0}$ as $n\to\infty$. Then $\hat V_{\hat\pi}\convprob V_{\hat\pi}$ and $\hat V_{\mathrm{t},\hat\pi}\convprob V_{\mathrm{t},\hat\pi}$ as $n\to\infty$. In particular, the estimators on the right side of \eqref{eq:vpihat_estimator} and \eqref{eq:vtpihat_estimator}  are all consistent for their respective estimands.
\end{proposition}

In view of \cref{thm:asymnorm_estimated_propscore_ate,thm:asymnorm_estimated_propscore_att}, an immediate implication is that we can construct asymptotic confidence intervals for ATE and ATT. Let $z_{1-\alpha/2}$ be the $(1-\alpha/2)$th quantile of the standard normal distribution for $\alpha\in(0,1)$, and let $[a\mypm b]$ denote the interval $[a-b,a+b]$ for $a,b\in\real$, $b\geq0$. Then the intervals $[\thatpihat \mypm z_{1-\alpha/2} (\hat V_{\hat\pi}/n)^{1/2}]$ and $[\tthatpihat \mypm z_{1-\alpha/2} (\hat V_{\mathrm{t},\hat\pi} /n)^{1/2}]$ are asymptotically valid confidence intervals for ATE and ATT, respectively.

\begin{corollary}[Asymptotic Confidence Intervals]
\label{coro:confidence_intervals}
Suppose that \cref{ass:ucf,ass:parametric_propscore,ass:propscoredist_estimated,ass:differentiability_regression,ass:lip_conditional_var,ass:theta_estimator,ass:covar_outcome_distribution,ass:singleindex_propscore_smooth_outcome} all hold, that the caliper $\delta_n$ satisfies \eqref{eq:caliper_hat}, and that $\hat V_{\theta_0}\convprob V_{\theta_0}$ as $n\to\infty$. Then $\prob{\left[\thatpihat \mypm z_{1-\alpha/2} \sqrt{\fr{\hat V_{\hat\pi}}{n}} \right]\ni \tau}\to 1-\alpha$ and $\prob{\left[\tthatpihat \mypm z_{1-\alpha/2} \sqrt{\fr{\hat V_{\mathrm{t},\hat\pi}}{n}} \right]\ni \taut}\to 1-\alpha$ as $n\to\infty$.
\end{corollary}

%%%%%%%%%%%%%%%%%%%%%%%%%%%%%%%%%%%%%%%%%%%%%%%%%%%%%%%%%
%%%%%%%%%%%%%%%%%%%%%%%%%%%%%%%%%%%%%%%%%%%%%%%%%%%%%%%%%

%%%%%%%%%%%%%%%%%%%% 		CONCLUSION		%%%%%%%%%%%%%%%%%%%%%%

%%%%%%%%%%%%%%%%%%%%%%%%%%%%%%%%%%%%%%%%%%%%%%%%%%%%%%%%%
%%%%%%%%%%%%%%%%%%%%%%%%%%%%%%%%%%%%%%%%%%%%%%%%%%%%%%%%%
\section{Conclusion}
\label{sec:conclusion}

We study the caliper matching estimator when matching is performed on the (estimated) propensity scores. We propose a caliper, and prove that the resulting estimator of the Average Treatment Effect (ATE), and of the Average Treatment Effect on the Treated (ATT), is asymptotically unbiased and normal. 

When the propensity score is known, our estimator of ATE reaches the semiparametric lower bound in the restricted model where only the propensity scores and not the covariates are observed in the sample or where the outcome regression on the covariates only depend on the propensity score. In this restricted model, the estimator of ATT only reaches the larger lower bound corresponding to unknown propensity score.  Even in the unrestricted model, both our estimators are more efficient than nearest neighbor matching estimators on the known propensity scores, and are, therefore, preferred over the latter method in the large sample limit, provided our assumptions hold. When the parametric propensity score is estimated, the variances of both our estimators increase, hence it remains unclear whether caliper or nearest neighbor matching will be more efficient. %In this case, practitioners are advised to report both estimators and their corresponding estimated variances.

We facilitate empirical application of the estimator by verifying our assumptions for a family of often employed models, and by constructing asymptotic confidence intervals for the average treatment effects. An interesting avenue for future research is to study in-sample estimation of the propensity score, and to allow for nonparametric propensity score estimators. The main challenge arising is to see how uncertainty from the propensity score estimation propagates to the matching estimator, which is more difficult to quantify for nonparametric models.

% Future research? Seems to be optional in recent Econometrica papers. Conclusions are also very much abstract-like or even shorter.

%%%%%%%%%%%%%%%%%%%%%%%%%%%%%%%%%%%%%%%%%%%%%%
%% Single Appendix:            %%
%%%%%%%%%%%%%%%%%%%%%%%%%%%%%%%%%%%%%%%%%%%%%%
%\begin{appendix}
%\section*{} %% if no title is needed, leave empty \section*{}.

%\end{appendix}

%%%%%%%%%%%%%%%%%%%%%%%%%%%%%%%%%%%%%%%%%%%%%%
%% Multiple Appendixes:        %%
%%%%%%%%%%%%%%%%%%%%%%%%%%%%%%%%%%%%%%%%%%%%%%
\newpage
\begin{appendix}

%%%%%%%%%%%%%%%%%%%%%%%%%%%%%%%%%%%%%%%%%%%%%%%%%%%%%%%%%
%%%%%%%%%%%%%%%%%%%%%%%%%%%%%%%%%%%%%%%%%%%%%%%%%%%%%%%%%

%%%%%%%%%%%%%%%%%%%% 		VARIANCE		%%%%%%%%%%%%%%%%%%%%%%

%%%%%%%%%%%%%%%%%%%%%%%%%%%%%%%%%%%%%%%%%%%%%%%%%%%%%%%%%
%%%%%%%%%%%%%%%%%%%%%%%%%%%%%%%%%%%%%%%%%%%%%%%%%%%%%%%%%
\section{Variance Estimation}
\label{app:sec:variance_estimation}
In this section, we define the nonparametric variance estimators of \Cref{sec:asymptotics:subsec:variance_estimation} for the components in \eqref{eq:vpihat_estimator} and \eqref{eq:vtpihat_estimator}. Let $K(u)=(2\pi)^{-1/2}e^{-u^2/2},u\in\real,$ be the Gaussian kernel and $K'(u)$ its derivative. Let $0<\gamma_n\lesssim a_n$ be two arbitrary sequences $\gamma_n\ceq\kappa_0 n^{-\beta}$, $a_n\ceq\kappa_1 n^{-\alpha}$ for fixed finite constants $0<\alpha<\beta<1/4$ and $\kappa_0,\kappa_1>0$.  We employ a truncation strategy to avoid bias at the boundaries. Define the intervals $A_n\ceq [\uphatth+a_n,\bphatth-a_n]$ and $\hat A_n\ceq [\min_{i\in[n]}g(\hatth^\intercal X_i)+a_n,\max_{i\in[n]}g(\hatth^\intercal X_i)-a_n]$, which are well-defined with probability tending to one as $\hatth\convprob\theta_0$ under \Cref{ass:theta_estimator}; see the proof of \Cref{prop:consistent_variance}.\footnote{In practice, especially for moderate sample sizes, $\gamma_n$ and $a_n$ should be chosen carefully to ensure nonnegative variance estimates. The $a_n$ should be chosen small enough to enlarge $A_n,\hat A_n$; for instance, one could choose $\kappa_1$ arbitrary close to zero and $\alpha\ceq1/(4+\varepsilon_\alpha)$ for an $\varepsilon_\alpha>0$ arbitrarily close to zero. As a rule, $\gamma_n$ should be set small too to minimise the bias of the variance estimates by standard nonparametric theory, thereby avoiding negative values; to accommodate $\alpha<\beta$, one can set $\beta=1/(4+\varepsilon_\beta)$ with $0<\varepsilon_\beta<\varepsilon_\alpha$, for example, $\varepsilon_\beta\ceq\varepsilon_\alpha/2$. The $\kappa_0$ should be chosen to accommodate the different scales of $(g(\hatth^\intercal X_i))_{i\in[n]}$ and $(\hatth^\intercal X_i)_{i\in[n]}$ present in the estimation of the $\mu^d$ and their derivate, respectively. A small $\kappa_0$ is a safe but conservative choice.
Note that asymptotically the effect of truncation disappears ($\E \ntruncate/n\to 1$) as shown in \Cref{prop:consistent_variance}.}  Let $\ntruncate\ceq \sum_{i\in[n]}\indic{g(\hatth^\intercal X_i)\in\hat A_n}$.
The estimators of the first components are
\begin{align}
\hat V_\tau & \ceq \left(\fr{1}{\ntruncate}\sum_{i\in[n]}[\hat \mu^1(\hatth,g(\hatth^\intercal X_i))-\hat\mu^0(\hatth,g(\hatth^\intercal X_i))]^2\indic{g(\hatth^\intercal X_i)\in\hat A_n}\right)-\thatpihat^2, \label{eq:hatVtau} \\
\hat V_{\taut} &\ceq \fr{1}{\hat p_1^2} \left(\fr{1}{\ntruncate}\sum_{i\in[n]}D_i[\hat \mu^1(\hatth,g(\hatth^\intercal X_i))-\hat\mu^0(\hatth,g(\hatth^\intercal X_i))]^2\indic{g(\hatth^\intercal X_i)\in\hat A_n}\right)-\fr{\tthatpihat^2}{\hat p_1}, \nonumber
\end{align}
where
\begin{align}
\hat \mu^d(\theta,p) &\ceq \fr{\hat q_{\mu,d}(\theta, p)}{\hat h_d(\theta,p)},\nonumber  \quad
\hat q_{\mu,d}(\theta, p) \ceq  \fr{1}{N_d\gamma_n}\sum_{j\in[n]}\indic{D_j=d}Y_jK\left(\fr{g(\theta^\intercal X_j)-p}{\gamma_n}\right), \nonumber \\
\hat h_d(\theta,p) &\ceq \hat f_{\theta,d}(p) \ceq \fr{1}{N_d\gamma_n}\sum_{j\in[n]}\indic{D_j=d}K\left(\fr{g(\theta^\intercal X_j)-p}{\gamma_n}\right),\quad d\in\dummy; \nonumber
\end{align}
and those of the second components are
\begin{align*}
\hat V_{\sigma,\pi} \ceq &\, \fr{1}{\ntruncate}\sum_{i\in[n]}\left(\fr{\hat\sigma_0^2(\hatth,g(\hatth^\intercal X_i))}{1-g(\hatth^\intercal X_i)} + \fr{\hat\sigma_1^2(\hatth,g(\hatth^\intercal X_i))} {g(\hatth^\intercal X_i)} \right)\indic{g(\hatth^\intercal X_i)\in\hat A_n}, \\
\hat V_{\mathrm{t},\sigma,\pi} \ceq&\, \fr{1}{\hat p_1^2\ntruncate}\sum_{i\in[n]}\Bigg(\fr{g(\hatth^\intercal X_i)^2\hat\sigma_0^2(\hatth,g(\hatth^\intercal X_i))}{1-g(\hatth^\intercal X_i)}+ g(\hatth^\intercal X_i) \hat\sigma_1^2(\hatth,g(\hatth^\intercal X_i))  \Bigg) \indic{g(\hatth^\intercal X_i)\in\hat A_n},
\end{align*}
where 
\begin{align*}
\hat\sigma_d^2(\theta,p) &\ceq \hat\mu_2^d(\theta, p) - (\hat\mu^d(\theta,p))^2, \quad \hat\mu_2^d(\theta, p) \ceq \fr{\hat q_{\mu_2,d}(\theta, p)}{\hat h_d(\theta,p)}, \\
\hat q_{\mu_2,d}(\theta, p) &\ceq  \fr{1}{N_d\gamma_n}\sum_{j\in[n]}\indic{D_j=d}Y_j^2K\left(\fr{g(\theta^\intercal X_j)-p}{\gamma_n}\right), \quad d\in\dummy, \nonumber
\end{align*}
is an estimator of $\sigma_d^2(\theta,p)=\mu_2^d(\theta,p)-(\mu^d(\theta,p))^2$ with 
$$\mu_2^d(\theta,p)\ceq\Ebc{Y^2}{D=d,\pi(X,\theta)=p},\quad d\in\dummy. \nonumber$$
Last, the probability limits of the derivatives are estimated by
\begin{align*}
\hat q_d^\intercal &\ceq \fr{1}{\ntruncate}\sum_{i\in[n]}\hat \Lambda^d(\hatth,X_i)\indic{g(\hatth^\intercal X_i)\in\hat A_n},\quad \hat q_{\mathrm{t},d}^\intercal \ceq \fr{1}{\ntruncate}\sum_{i\in[n]}D_i\hat \Lambda^d(\hatth,X_i)\indic{g(\hatth^\intercal X_i)\in\hat A_n}, \\
\hat \Lambda^d(\theta,x) &\ceq \widehat{\left(\fr{\partial \mu^d}{\partial \theta^\intercal}\right)}(\theta,g(\theta^\intercal x)) + \widehat{\left(\fr{\partial \mu^d}{\partial p}\right)}(\theta,g(\theta^\intercal x))g'(\theta^\intercal x)x^\intercal,  
\end{align*}
where
\begin{align*}
\widehat{\left(\fr{\partial \mu^d}{\partial \theta_k}\right)}(\theta,p)& \ceq \fr{ \widehat{\left(\fr{\partial q_{\mu,d}}{\partial \theta_k}\right)}(\theta,p)\hat h_d(\theta,p)-\hat q_{\mu,d}(\theta,p)\widehat{\left(\fr{\partial h_d}{\partial \theta_k}\right)}(\theta,p)}{(\hat h_d(\theta,p))^2}, \\
\widehat{\left(\fr{\partial q_{\mu,d}}{\partial \theta_k}\right)}(\theta,p) &\ceq \fr{(\ginv)'(p)}{N_d\gamma_n^2}\sum_{j\in[n]}\indic{D_j=d}Y_jX_{j,k}K'\left(\fr{\theta^\intercal X_j-\ginv(p)}{\gamma_n} \right), \\
\widehat{\left(\fr{\partial h_d}{\partial \theta_k}\right)}(\theta,p) &\ceq \fr{(\ginv)'(p)}{N_d\gamma_n^2}\sum_{j\in[n]}\indic{D_j=d}X_{j,k}K'\left(\fr{\theta^\intercal X_j-\ginv(p)}{\gamma_n} \right),
\end{align*}
with $\theta_k$ ($X_{j,k}$) being the $k$th coordinate of $\theta$ ($X_j$) for $k\in[K]$, and
\begin{align*}
\widehat{\left(\fr{\partial \mu^d}{\partial p}\right)}(\theta,p)& \ceq \fr{\left(\fr{\partial}{\partial p}\hat q_{\mu,d}(\theta,p)\right)\hat h_d(\theta,p)-\hat q_{\mu,d}(\theta,p)\fr{\partial}{\partial p}\hat h_{d}(\theta,p)}{(\hat h_d(\theta,p))^2}, \\
\fr{\partial}{\partial p}\hat q_{\mu,d}(\theta,p) &= -\fr{1}{N_d\gamma_n^2}\sum_{j\in[n]}\indic{D_j=d}Y_jK'\left(\fr{g(\theta^\intercal X_j)-p}{\gamma_n}\right), \\
\fr{\partial}{\partial p}\hat h_d(\theta,p) &=  -\fr{1}{N_d\gamma_n^2}\sum_{j\in[n]}\indic{D_j=d}K'\left(\fr{g(\theta^\intercal X_j)-p}{\gamma_n}\right)
\end{align*}
for $d\in\dummy$. An intercept in the propensity score model can be accommodated by defining $X_{j,K+1}\ceq 1$ for $j\in[n]$ and considering derivatives with respect to $\theta_{K+1}$ too. 

% Note to self, regarding computation of estimates: normalising by $n$ or $N_d$  in $\hat q_{\mu,d}$, $\hat q_{\mu_2,d}$, $\hat h_d$ and their derivates has no effect on the value of the estimates of $\mu^d$, $\mu_2^d$ and their derivatives as the normalisation factor cancels.

%%%%%%%%%%%%%%%%%%%%%%%%%%%%%%%%%%%%%%%%%%%%%%%%%%%%%%%%%
%%%%%%%%%%%%%%%%%%%%%%%%%%%%%%%%%%%%%%%%%%%%%%%%%%%%%%%%%

%%%%%%%%%%%%%%%%%%%% 		PROOFS		%%%%%%%%%%%%%%%%%%%%%%

%%%%%%%%%%%%%%%%%%%%%%%%%%%%%%%%%%%%%%%%%%%%%%%%%%%%%%%%%
%%%%%%%%%%%%%%%%%%%%%%%%%%%%%%%%%%%%%%%%%%%%%%%%%%%%%%%%%
\section{Proofs}
\label{app:sec:proofs}

%\textcolor{red}{DON'T FORGET TO SET BLACKQUARES BACK TO DEFAULT END-OF-PROOF OPTION!!!}

In this section, we prove the main results, \Cref{prop:maxmindistorder,prop:number_of_matches,thm:asymnorm_known_propscore_ate,thm:asymnorm_known_propscore_att,prop:admissible_models_known_propscore,prop:number_of_matches_estimated,thm:asymnorm_estimated_propscore_ate,thm:asymnorm_estimated_propscore_att,prop:admissible_models_estimated_propscore} together with supporting \Cref{lem:convergence_of_ratios,lem:error,lem:lindeberg_feller_bound,lem:efficiency}. The proofs of \Cref{prop:semipara_eff,prop:consistent_variance} and \Cref{lem:spacingsmoments,lem:exponential_orderstats_moments,lem:conditional_mclt} are in the \hyperref[app:online]{Supplement}. For simplicity, we give the proofs for the caliper choice $\delta_n=\caliperlognn$, $s\ceq1$, and provide remarks for the choices $\delta_n=\caliperDelta$ and $\delta_n=\caliperDeltahat$ when necessary.

We adopt the following notation. Let $D^{(n)}\ceq (D_i)_{i\in[n]}$ and $\Pvector\ceq (\pi(X_i))_{i\in[n]}$. For $a,b\in\real, a<b$, let 
\begin{align}
F_d[a,b]&\ceq\probc{\pi(X)\in[a,b]}{D=d}, \nonumber \\
\mathbbm{F}_{N_d}[a,b]&\ceq\fr{1}{N_d}\sum_{i:D_i=d}\indic{\pi(X_i)\in[a,b]}, \label{eq:def_measures}
\end{align}
be the conditional (empirical) measures of intervals $[a,b]$ for $d\in\dummy$. Similarly, under \Cref{ass:parametric_propscore}, define the conditional (empirical) measures
\begin{align}
F_{d,\theta}[a,b]&\ceq\probc{\pi(X,\theta)\in[a,b]}{D=d}, \nonumber \\
\mathbb{F}_{N_d, \theta}[a,b]&\ceq \fr{1}{N_d}\sum_{i:D_i=d}\indic{\pi(X_i,\theta)\in [a,b]}\quad\text{ for $\theta\in\Theta$.}\label{eq:def_measures_theta}
\end{align}
Let $\gproc_{N_d}\ceq \sqrt{N_d}(\mathbbm{F}_{N_d}-F_d)$ be the empirical process of $((\pi(X_{i}))_{i:D_i=d}\mid \Dvector)\overset{\text{i.i.d.}}{\sim}F_d$, and $[a\mypm b]$ denote the interval $[a-b,a+b]$. %We write $g(X_n,n,c)\lesssim (\gtrsim)Y_n$, for some measurable function $g$, sample size $n$, fixed constant $c$, and random variables $X_n, Y_n$, to mean that there exists a finite (nonzero) constant $C$, potentially depending on $c$, but not on $(X_n,Y_n, n)$, such that $g(X_n,n,c)\leq (\geq)CY_n$ for all $n\geq2$. In the same notation, $g(X_n,n,c)\simeq Y_n$ means $g(X_n,n,c)=CY_n$ for all $n\geq 2$ with the restriction that $C\neq0$. 
In the proofs, the value of constants may change from equation to equation without explicit notice.

%%%%%%%%%%%%%%%%%%%% 		Caliper Choise		%%%%%%%%%%%%%%%%%%%%%%
\begin{proof}[Proof of \Cref{prop:maxmindistorder}]

\emph{Spacings and Their Order.} Let $U_{(1)}\leq U_{(2)}\leq \ldots\leq U_{(N_1)}$ be the order statistics of $(U_1,\ldots, U_{N_1}\mid \Dvector)\overset{\text{i.i.d.}}{\sim}\text{Uniform}(0,1)$. Let $\tilde U_1\ceq U_{(1)}$, $\tilde U_i\ceq U_{(i)}-U_{(i-1)}$ for $i=2,\ldots,N_1$ and $\tilde U_{N_1+1}\ceq 1-U_{(N_1)}$ be the spacings generated by $(U_i)_{i\in[N_1]}$. Let $\tilde U_{(1)}\leq \tilde U_{(2)}\leq\ldots\leq \tilde U_{(N_1+1)}$ be the ordered spacings. \citet[Chapter 21]{shorack_empirical_2009} prove that $\Ebc{\tilde U_{(N_1+1)}}{\Dvector}=\fr{1}{N_1+1}\sum_{i=1}^{N_1+1}\fr{1}{N_1+2-i}$, which we apply as follows.

\emph{Bounding $\E\maxminDelta$ with Spacings.} Let $\minDelta_i\ceq\min_{j:D_j\neq D_i}|\pi(X_i)-\pi(X_j)|$ for $i\in[n]$, so that $\maxminDelta\leq\sum_{d\in\dummy}\max_{i:D_i=d}\minDelta_i$. Consider $\max_{i:D_i=0}\minDelta_i$. Let $\pi_{(1)}\leq \pi_{(2)}\leq\ldots\leq\pi_{(N_1)}$ be the order statistics of $(\pi(X_i))_{i:D_i=1}$, and let $\tilde\pi_1\ceq \pi_{(1)}-\underline p$, $\tilde\pi_{i}\ceq \pi_{(i)}-\pi_{(i-1)}$ for $i\in\{2,\ldots,N_1\}$ and $\tilde\pi_{N_1+1}\ceq \ba p-\pi_{(N_1)}$ be the corresponding spacings for $\underline p,\ba p$ of \Cref{ass:propscoredist}. Let $\tilde \pi_{(1)}\leq\tilde\pi_{(2)}\leq\ldots\leq \tilde\pi_{(N_1+1)}$ be the order statistics of these spacings. Every propensity score $\pi(X_j)$ of the control units falls either to the left of $\pi_{(1)}$ or to the right of $\pi_{(N_1)}$ or between two propensity scores $\pi_{(i)}$ and $\pi_{(i-1)}$ for some $i\in\set{2,3,\ldots,N_1}$. In all three cases, the closest treated propensity score to $\pi(X_j)$ is within $\tilde\pi_{(N_1+1)}$-distance. Hence, $\max_{i:D_i=0}\minDelta_i\leq \tilde\pi_{(N_1+1)}$. 
 
Given $\Dvector$, the $\pi(X_i)$ restricted to $i:D_i=1$ are i.i.d., with $(\pi(X_i)\mid D_i=1)\sim F_1$. By \Cref{ass:propscoredist}, $\inf_{p\in[\underline p,\ba p]}f_1(p)>0$, thus $F_1$ is strictly increasing on $[\underline p, \ba p]$, and therefore has a strictly increasing inverse $F_1^{-1}$ on $[F_1(\underline p), F_1(\ba p)]$.   Because $F_1^{-1}$ is increasing, $((\pi_{(i)})_{i\in[N_1]}\mid \Dvector) \sim ((F_1^{-1}(U_{(i)}))_{i\in[N_1]}\mid \Dvector)$ by the quantile transform. We distinguish three cases. 
%\vspace{-1.85cm}
\begin{itemize}%[label=(\roman*)]
\item Case 1: $\tilde\pi_{(N_1+1)}=\pi_{(i)}-\pi_{(i-1)}$ for some $i\in\{2,\ldots,N_1\}$. Then we bound $\tilde\pi_{(N_1+1)}$ by noting that  $(\pi_{(i)}-\pi_{(i-1)}\mid \Dvector)\sim (F_1^{-1}(U_{(i)})-F_1^{-1}(U_{(i-1)})\mid\Dvector)$, which is bounded by $\supnorm{(F_1^{-1})'}(U_{(i)}-U_{(i-1)})=\supnorm{(F_1^{-1})'}\tilde U_{i}$, because $F_1^{-1}$ is Lipschitz with constant $\supnorm{(F_1^{-1})'}$ with  $(F_1^{-1})'(u)=\fr{1}{f_1(F_1^{-1}(u))}$ finite as $\inf_{p\in[\underline p,\ba p]}f_1(p)>0$ by \Cref{ass:propscoredist}.

\item Case 2: $\tilde\pi_{(N_1+1)}=\pi_{(1)}-\underline p$. Then write $\underline p=F_1^{-1}(F_1(\underline p))=F_1^{-1}(0)$, so $(\tilde\pi_{(N_1+1)}\mid\Dvector)$ is distributed as a random variable that is bounded by $\supnorm{(F_1^{-1})'}(U_{(1)}-0)=\supnorm{(F_1^{-1})'}\tilde U_{1}$.

\item  Case 3: $\tilde\pi_{(N_1+1)}=\ba p -\pi_{(N_1)}$. Then write $\ba p =F_1^{-1}(F_1(\ba p))=F_1^{-1}(1)$, so $(\tilde\pi_{(N_1+1)}\mid\Dvector)$ is distributed as a random variable that is bounded by $\supnorm{(F_1^{-1})'}(1-U_{(N_1)})=\supnorm{(F_1^{-1})'}\tilde U_{N_1+1}$.
\end{itemize}
Conclude that $(\tilde\pi_{(N_1+1)}\mid\Dvector)$ is distributed as a random variable that is bounded by $\supnorm{(F_1^{-1})'}\tilde U_{(N_1+1)}$. Thus,
\begin{align}
\E\max_{i:D_i=0}\minDelta_i & \leq \E \Ebc{\tilde\pi_{(N_1+1)}}{\Dvector} \lesssim \Eb{\fr{1}{N_1+1}\sum_{i=1}^{N_1+1}\fr{1}{N_1+2-i}}\nonumber\\
 &=\sum_{n_1=0}^{n}\left[\fr{1}{n_1+1}\sum_{i=1}^{n_1+1}\fr{1}{n_1+2-i}\right]{n \choose n_1}p_1^{n_1}(1-p_1)^{n-n_1} \nonumber \\
 &= (1-p_1)^{n} + \sum_{n_1=1}^{n}\left[\fr{1}{n_1+1}\sum_{i=1}^{n_1+1}\fr{1}{n_1+2-i}\right]{n \choose n_1}p_1^{n_1}(1-p_1)^{n-n_1}
  \label{eq:exp_delta_bound}
\end{align}
by \cite{shorack_empirical_2009} where $p_1=\prob{D=1}$. The first term in \eqref{eq:exp_delta_bound} decays exponentially. In the second term of \eqref{eq:exp_delta_bound}, the integrand in the square brackets is asymptotic to $\fr{\log n_1}{n_1+1}\leq \fr{\log n}{n_1+1}$. That is, there exist some constants $c,\ba n_1>0$ such that $\fr{1}{n_1+1}\sum_{i=1}^{n_1+1}\fr{1}{n_1+2-i}\leq c\fr{\log n}{n_1+1}$ if $n_1> \ba n_1$. Since the integrand in the square brackets is bounded by one, it follows that the second term in \eqref{eq:exp_delta_bound} is bounded by 
$$\sum_{n_1=1}^{\ba n_1}{n \choose n_1}p_1^{n_1}(1-p_1)^{n-n_1} + c(\log n) \Eb{(1+N_1)^{-1}},$$
 where the first term is $\bigO{n^{{\ba n_1}}(1-p_1)^n}=\bigO{\log n/n}$ and the second term is $\bigO{\log n/n}$ by \cite{cribari-neto_note_2000}. Similar arguments hold for $\E\max_{i:D_i=1}\minDelta_i$ by symmetry.
%\begin{align}
%\sum_{n_1=1}^{\ba n_1}{n \choose n_1}p_1^{n_1}(1-p_1)^{n-n_1} + c(\log n) \Eb{(1+N_1)^{-1}}   \label{eq:exp_delta_bound_2}.
%\end{align}
%The first term in \eqref{eq:exp_delta_bound_2} is $\bigO{n^{{\ba n_1}}(1-p_1)^n}=\bigO{\log n/n}$. The second term in \eqref{eq:exp_delta_bound_2} is $\bigO{\log n/n}$ by \cite{cribari-neto_note_2000}. Similar arguments hold for $\E\max_{i:D_i=1}\minDelta_i$ by symmetry.
\end{proof}

%%%%%%%%%%%%%%%%%%%% 		Number of matches		%%%%%%%%%%%%%%%%%%%%%%
\begin{proof}[Proof of \Cref{prop:number_of_matches}]
We have for the $R_{di}$ in \eqref{eq:def_r} of \Cref{lem:convergence_of_ratios},
\begin{align}
\min_{i\in[n]}M_i &= \min_{i\in[n]}N_{1-D_i} F_{1-D_i}[\pi(X_i)\mypm\delta_n](1+R_{1-D_i,i}). \label{eq:minM}%\\  
%& \geq 2\left(\min_{d\in\dummy}\inf_{p\in[\underline p, \ba p]}f_d(p)\right) (N_0\wedge N_1)\delta_n \min_{i\in[n]}(1+R_{1-D_i,i}).
\end{align}
If $\min_{i\in[n]}(1+R_{1-D_i,i})\geq 0$, which happens with probability tending to one, then \eqref{eq:minM} is larger than or equal to $$2\left(\min_{d\in\dummy}\inf_{p\in[\underline p, \ba p]}f_d(p)\right) (N_0\wedge N_1)\delta_n \min_{i\in[n]}(1+R_{1-D_i,i}).$$ 
By \Cref{ass:propscoredist}, $\inf_{p\in[\underline p, \ba p]}f_d(p)>0$. By the strong law of large numbers and the continuous mapping theorem, $(1/\log n)(N_0\wedge N_1)\delta_n=(N_0\wedge N_1)/n\convas (1-p_1)\wedge p_1>0$. By \Cref{lem:convergence_of_ratios}, $\max_{i\in[n]}|R_{1-D_i,i}|=\smallOP{1}$. Then $\min_{i\in[n]}M_i\simeq (1+\smallOP{1})\log n$, from which the lower bound in \Cref{prop:number_of_matches}, and thus $\prob{\min_{i\in[n]}M_i\geq1}\to1$, follows. The same reasoning applies to the upper bound in \Cref{prop:number_of_matches}. 

When the caliper is $\delta_n=\caliperDelta$, $\min_{i\in[n]}M_i\geq1$. \Cref{lem:convergence_of_ratios}\ref{r_convergence_theta}--\ref{rcheck_convergence_theta}, the continuous mapping theorem, combined with the law of large numbers and \Cref{prop:maxmindistorder} prove the assertion.
\end{proof}

\begin{proof}[Proof of \Cref{prop:number_of_matches_estimated}]
Follows along arguments in the proof of \Cref{prop:number_of_matches} and \Cref{lem:convergence_of_ratios} \ref{r_convergence_theta}--\ref{rcheck_convergence_theta}. 
When the caliper is $\delta_n=\caliperDeltahat$, it is bounded by $\maxminDeltahat+\fr{\log N_0}{N_0+1}+\fr{\log N_1}{N_1+1}$, where $\E\fr{\log N_0}{N_0+1}=\bigO{\fr{\log n}{n}}$ by \cite{cribari-neto_note_2000}. A spacings argument on $(\pi(X_i,\hatth))_{i\in[n]}$, similarly to the proof of \Cref{prop:maxmindistorder}, combined with \Cref{ass:propscoredist_estimated,ass:theta_estimator} yields $\maxminDeltahat=\bigOP{\fr{\log n}{n}}$. Then arguments in \Cref{prop:number_of_matches} and \Cref{lem:convergence_of_ratios}\ref{r_convergence_theta}--\ref{rcheck_convergence_theta} prove \Cref{prop:number_of_matches_estimated}.
\end{proof}

%%%%%%%%%%%%%%%%%%%% 		AsymNorm Known propscore ATE		%%%%%%%%%%%%%%%%%%%%%%
\begin{proof}[Proof of \Cref{thm:asymnorm_known_propscore_ate}]
By \Cref{ass:propscoredist}, $\tau(\pi(X))$ of \eqref{eq:ydecomptau} is well-defined, with $\E\tau(\pi(X))=\tau$ by \Cref{ass:ucf}. By \Cref{ass:lip_regression}, the $\mu^d$ are Lipschitz continuous on the compact set $[\underline p,\ba p]$, hence are bounded, and then so is $V_\tau<\infty$. Then the central limit theorem implies $\sqrt{n}(\overline{\tau(\pi(X))}-\tau)\rsquig \mathcal{N}(0,V_\tau)$. Combine this with \Cref{lem:error}, to get 
\begin{align*}
\begin{bmatrix}
V_\tau^{-1/2}\sqrt{n}(\overline{\tau(\pi(X))}-\tau) \\
V_E^{-1/2}\sqrt{n}E
\end{bmatrix}
\rsquig \mathcal{N}(0,\mathrm{I}_2),
\end{align*}
along subsequences, where $\mathrm{I}_2$ is the 2-by-2 identity matrix. By \Cref{lem:efficiency}, $V_E\convprob V_{\sigma,\pi}$, which is finite by \Cref{ass:propscoredist,ass:variance}. Then the continuous mapping theorem and Slutsky's lemma imply $(V_\tau+V_{\sigma,\pi})^{-1/2}\sqrt{n}(\overline{\tau(\pi(X))}-\tau+E)\rsquig\mathcal{N}(0,1)$. The event $\{\min_{i\in[n]} M_i>0\}$ happens with probability tending to one by \Cref{prop:number_of_matches}. On this event, $\sqrt{n}|B|\lesssim \sqrt{n}\delta_n=\fr{\log n}{\sqrt{n}}=\smallO{1}$ by \Cref{ass:lip_regression}. Thus $\sqrt{n}B=\smallOP{1}$.

When the caliper is $\delta_n=\caliperDelta$, \Cref{lem:error} continues to apply. Then the continuous mapping theorem, the law of large numbers and \Cref{prop:maxmindistorder} imply $\sqrt{n}|B|=\smallOP{1}$.
\end{proof}

%%%%%%%%%%%%%%%%%%%% 		AsymNorm Known propscore ATT		%%%%%%%%%%%%%%%%%%%%%%
\begin{proof}[Proof of \Cref{thm:asymnorm_known_propscore_att}]
A decomposition similar to \eqref{eq:thatpi_decomp}--\eqref{eq:ydecompbicdo} holds, whereby
\begin{align*}
\sqrt{n}(\tthatpi-\taut) =&\, \sqrt{n}(\overline{\taut(\pi(X))}-\taut) + \sqrt{n}E_\text{t} + \sqrt{n}B_\text{t} \\
\overline{\taut(\pi(X))} \ceq &\, \fr{1}{N_1}\sumn{i} D_i\tau(\pi(X_i)) \\
 E_\text{t} \ceq &\, \fr{1}{N_1}\sumn{i} E_{\text{t},i},\quad E_{\text{t},i}\ceq (\indic{M_i>0}D_i-(1-D_i)w_i)\varepsilon_i \\
 B_\text{t}  \ceq &\, \fr{1}{N_1}\sumn{i} B_{\text{t},i}, \\
 B_{\text{t},i} \ceq &\, D_i(\indic{M_i>0}-1)(\mu^{D_i}(\pi(X_i))+\mu^{1-D_i}(\pi(X_i))) \\
 	&+ D_i\fr{\indic{M_i>0}}{M_i}\sum_{j\in\calJC(i)} (\mu^0(\pi(X_i))-\mu^0(\pi(X_j))).
\end{align*}
As $\E D(\tau(\pi(X))-\taut)=\Ebc{\tau(\pi(X))-\taut}{D=1}p_1=0$, $\sqrt{n}(\overline{\taut(\pi(X))}-\taut)$ is mean zero with finite variance by \cref{ass:ucf,ass:lip_regression}. By the central limit theorem, continuous mapping and Slutsky's lemma, $\sqrt{n}(\overline{\taut(\pi(X))}-\taut)=(N_1/n)^{-1}n^{-1/2}\sum_{i\in[n]}D_i(\tau(\pi(X_i))-\taut)\rsquig \mathcal{N}(0, V_{\taut})$ since $(N_1/n)\convas p_1$. The $\sqrt{n}E_\mathrm{t}$ has mean zero and a Lindeberg-Feller central limit theorem establishes that $\sup_{x\in\real}\abs{\probc{V_{E_\mathrm{t}}^{-1/2}\sqrt{n}E_\mathrm{t}\leq x}{\Dvector,\Pvector}-\Phi(x)}\convprob 0$, where $V_{E_\mathrm{t}}\ceq \nsumn{i} (\indic{M_i>0}D_i-(1-D_i)w_i)^2\sigma_{D_i}^2(\pi(X_i))$,  similarly to arguments in \Cref{lem:error}. By arguments similar to those of \Cref{lem:efficiency}, $V_{E_{\mathrm{t}}}\convprob V_{\mathrm{t},\sigma,\pi}$. By \Cref{prop:number_of_matches} and arguments in \Cref{thm:asymnorm_known_propscore_ate}, $\sqrt{n}B_\mathrm{t}=\smallOP{1}$.
\end{proof}

%%%%%%%%%%%%%%%%%%%% 		AsymNorm Estimated propscore ATE		%%%%%%%%%%%%%%%%%%%%%%
\begin{proof}[Proof of \Cref{thm:asymnorm_estimated_propscore_ate}]
Let $w_i(\hatth)\ceq \sum_{j\in\calJC_{\hatth}(i)}\fr{1}{M_j(\hatth)}, i\in[n]$, and, as in \eqref{eq:thatpi_decomp}--\eqref{eq:ydecompbicdo}, decompose
\begin{align}
\sqrt{n}(\thatpihat-\tau) =\, \fr{1}{\sqrt n}\sum_{i\in[n]}\left\{\mu^1(\hatth, \pi(X_i,\hatth))-\mu^0(\hatth, \pi(X_i,\hatth)) - \tau \right\} \label{eq:tauhat_thetahat_decomp1} \\
 +\fr{1}{\sqrt n}\sum_{i\in[n]}(2D_i-1)(\indic{M_i(\hatth)>0}+w_i(\hatth))\varepsilon_i(\hatth)\label{eq:tauhat_thetahat_decomp2} \\
 +\fr{1}{\sqrt n}\sum_{i\in[n]}(2D_i-1)(\indic{M_i(\hatth)>0}-1)(\mu^{1-D_i}(\hatth, \pi(X_i,\hatth))-\mu^{D_i}(\hatth,\pi(X_i,\hatth))) \label{eq:tauhat_thetahat_decomp3} \\
 +\fr{1}{\sqrt n}\sum_{i\in[n]}(2D_i-1)\fr{\indic{M_i(\hatth)>0}}{M_i(\hatth)}\sum_{j\in\calJC_{\hatth}(i)}\left[\mu^{1-D_i}(\hatth, \pi(X_i,\hatth))-\mu^{1-D_i}(\hatth, \pi(X_j,\hatth))\right]. \label{eq:tauhat_thetahat_decomp4}
\end{align}
We show first that \eqref{eq:tauhat_thetahat_decomp1} and \eqref{eq:tauhat_thetahat_decomp2} are, asymptotically, jointly normal and independent, and then that \eqref{eq:tauhat_thetahat_decomp3} and \eqref{eq:tauhat_thetahat_decomp4} are asymptotically negligible.

\emph{Terms \eqref{eq:tauhat_thetahat_decomp1} and \eqref{eq:tauhat_thetahat_decomp2}}. We apply the following result with \eqref{eq:tauhat_thetahat_decomp1} and \eqref{eq:tauhat_thetahat_decomp2} corresponding to $V_n$ and $W_n$ respectively.  Let $V_{n}$, $W_{n}$, $n=1,2,\ldots$ be two sequences of random variables defined on some probability space. To show that $(V_n,W_n)\rsquig(V,W)\sim\mathcal{N}(0, \Sigma)$, for a diagonal matrix $\Sigma=\textup{diag}(\sigma_V^2,\sigma_W^2)$, it suffices that $\E h_1(V_n)h_2(W_n)\to (\E h_1(V))(\E h_2(W))$ for all bounded continuous functions $h_1$, $h_2:\real\to\real$. Let $\filt_{n0}$ be a sub-$\sigma$-algebra such that $V_{n}$ is $\filt_{n0}$-measurable for all $n\geq 1$. As $\E h_1(V_n)h_2(W_n)=\Eb{h_1(V_n)\Ebc{h_2(W_n)}{\filt_{n0}}}$, if suffices, by the Portmanteau lemma, that $V_n\rsquig \mathcal{N}(0,\sigma_V^2)$ and $\probc{W_n\leq w}{\filt_{n0}}\convprob \Phi(w/\sigma_W)$ for all $w\in\real$. 

\emph{Convergence of \eqref{eq:tauhat_thetahat_decomp1}.} Expand \eqref{eq:tauhat_thetahat_decomp1} as
\begin{align}
\fr{1}{\sqrt n}\sum_{i\in[n]}(\mu_{\theta_0}^1(\pi(X_i,\theta_0))-\mu_{\theta_0}^0(\pi(X_i,\theta_0))-\tau)   \label{eq:tauhat_thetahat_decomp1_1} \\
 + \fr{1}{\sqrt n}\sum_{i\in[n]}\Big\{\mu^1(\hatth, \pi(X_i,\hatth))-\mu^0(\hatth, \pi(X_i,\hatth)) -\left[\mu_{\theta_0}^1(\pi(X_i,\theta_0))-\mu_{\theta_0}^0(\pi(X_i,\theta_0))\right]\Big\}.\label{eq:tauhat_thetahat_decomp1_2}
 \end{align}
By \Cref{ass:ucf,ass:lip_regression}, \eqref{eq:tauhat_thetahat_decomp1_1} converges weakly to $\mathcal{N}(0,V_\tau)$ by the standard central limit theorem. By \cref{ass:parametric_propscore,ass:differentiability_regression}, $\Lambda^d(\tld\theta,x)$ is well-defined for all $(\tld\theta,x)\in\ntheta\times\Xsupp$. Then by the mean-value theorem, $\mu^d(\hatth, \pi(x,\hatth))=\mu_{\theta_0}^d(\pi(x,\theta_0))+\Lambda^d(\tld\theta^d,x)(\hatth-\theta_0)$, for some $\tld\theta^d$ on the line segment between $\hatth$ and $\theta_0$. Rewrite \eqref{eq:tauhat_thetahat_decomp1_2} as
\begin{align}
\left(\fr{1}{n}\sum_{i\in[n]}\left[\Lambda^1(\tld\theta^1,X_i)-\Lambda^0(\tld\theta^0,X_i) \right]\right) \sqrt{n}(\hatth-\theta_0). \label{eq:tauhat_thetahat_decomp1_2_rewrite}
\end{align}
By \cref{ass:parametric_propscore,ass:differentiability_regression}, $\tld\theta\mapsto \Lambda^d(\tld\theta,x)$ is continuous and uniformly bounded for all $(\tld\theta,x)\in\ntheta\times\Xsupp$, therefore $\nsumn{i}\Lambda^d(\tld\theta^d,X_i)\convprob q_d^\intercal$ for some finite $q_d\in\real^K$. 
Then \eqref{eq:tauhat_thetahat_decomp1_2_rewrite} and Slutsky's lemma imply that \eqref{eq:tauhat_thetahat_decomp1_2} converges weakly to $\mathcal{N}(0,(q_1-q_0)^\intercal V_{\theta_0}(q_1-q_0))$ by \cref{ass:theta_estimator}. But, by \cref{ass:theta_estimator} again, \eqref{eq:tauhat_thetahat_decomp1_1} is independent of $\sqrt{n}(\hatth-\theta_0)$, thus \eqref{eq:tauhat_thetahat_decomp1}, being the sum of \eqref{eq:tauhat_thetahat_decomp1_1} and \eqref{eq:tauhat_thetahat_decomp1_2}, converges weakly to $\mathcal{N}(0,V_\tau+(q_1-q_0)^\intercal V_{\theta_0}(q_1-q_0))$.

\emph{Conditional Convergence of \eqref{eq:tauhat_thetahat_decomp2}.} Let $\filt_{n0}\ceq \sigma\{D_1,\ldots, D_n,\pi(X_1,\hatth),\ldots ,\pi(X_n,\hatth),\hatth\}$, so that \eqref{eq:tauhat_thetahat_decomp1} is $\filt_{n0}$-measurable. We show that given $\filt_{n0}$, \eqref{eq:tauhat_thetahat_decomp2} converges weakly to a normal variate in probability. We construct a martingale array and apply  \Cref{lem:conditional_mclt}. Let $\xi_{ni}\ceq (2D_i-1)(1+ w_i(\hatth))\varepsilon_i(\hatth)/{\sqrt n}$ and 
$$\mathcal{F}_{ni}\ceq \sigma\{D_1,\ldots, D_n,\pi(X_1,\hatth),\ldots ,\pi(X_n,\hatth), \varepsilon_1(\hatth),\ldots,\varepsilon_{i}(\hatth),\hatth\}$$
for $i\in[n]$. Assume temporarily that $\min_{i\in[n]} M_i(\hatth)>0$, so that \eqref{eq:tauhat_thetahat_decomp2} is equal to $\sum_{i=1}^n\xi_{ni}$. One can verify that $\xi_{n1},\xi_{n2},\ldots,\xi_{nn}$ are martingale differences relative to the filtration $\mathcal{F}_{n1}\subset \mathcal{F}_{n2}\subset\ldots\subset \mathcal{F}_{nn}$, using that $\mu^{D_i}({\hatth},\pi(X_i,\hatth))$ is $\mathcal{F}_{n,i-1}$-measurable and \Cref{ass:theta_estimator}\ref{ass:theta_estimator:indep}, which implies that the observations are i.i.d. given $\hatth$.

First, we verify the variance condition \eqref{eq:mclt_var} of Lemma \ref{lem:conditional_mclt}. Consider $\sum_{i=1}^n\Ebc{\xi_{ni}^2}{\mathcal{F}_{n,i-1}}$. We have
 \begin{align*}
\Ebc{\xi_{ni}^2}{\mathcal{F}_{n,i-1}} & = (1+w_i(\hatth))^2\sigma_{D_i}^2(\hatth,\pi(X_i,\hatth))/n, 
\end{align*}
where we used \Cref{ass:theta_estimator}\ref{ass:theta_estimator:indep} again, and that $w_i(\hatth)$ is $\mathcal{F}_{n,i-1}$-measurable. But then $\Ebc{\xi_{ni}^2}{\mathcal{F}_{n,i-1}}$ is $\filt_{n0}$-measurable for all $i\in[n]$ for all $n\geq 1$, thus for condition \eqref{eq:mclt_var} it suffices that $\sum_{i=1}^n\Ebc{\xi_{ni}^2}{\mathcal{F}_{n,i-1}}$ converges in $\bbP$-probability to a finite constant. \Cref{ass:parametric_propscore}\ref{ass:parametric_propscore:bounded_derivative} implies that $\max_{i\in[n]}|\pi(X_i,\hatth)-\pi(X_i,\theta_0)|=\bigOP{\normu{\hatth-\theta_0}}$. Then, under Assumption \ref{ass:lip_conditional_var}, we can write $\sigma_d^2(\hatth,\pi(X_i,\hatth))=\sigma_{d}^2(\theta_0,\pi(X_i,\theta_0))+S_i$, where $\max_{i\in[n]}|S_i|=\bigOP{\normu{\hatth-\theta_0}}$. Write
\begin{align}
\sum_{i=1}^n\Ebc{\xi_{ni}^2}{\mathcal{F}_{n,i-1}} =&\, \fr{1}{n}\sum_{i\in[n]}(1+w_i(\hatth))^2\sigma_{D_i}^2(\theta_0,\pi(X_i,\theta_0)) +\fr{1}{n}\sum_{i\in[n]}(1+ w_i(\hatth))^2S_i. \label{eq:varcondition}
\end{align}
Under \Cref{ass:parametric_propscore,ass:propscoredist_estimated,ass:theta_estimator}, the arguments in the proof of  \Cref{lem:efficiency} continue to apply in view of \Cref{lem:convergence_of_ratios}\ref{r_convergence_theta}--\ref{rcheck_convergence_theta}. Then for any sequence $(Q_i)_{i\in[n]}$ of random variables and any fixed constant $r\in\real$, we have, by \Cref{lem:convergence_of_ratios}\ref{r_convergence_theta}--\ref{rcheck_convergence_theta},
\begin{align}
\sum_{i\in[n]} w_i(\hatth)^rQ_i  = &\, (1+\smallOP{1})\left\{\left(\fr{N_1}{N_0}\right)^r\sum_{i:D_i=0}\left(\fr{f_{1,\hatth}(\pi(X_i,\hatth))}{f_{0,\hatth}(\pi(X_i,\hatth))}\right)^rQ_i  \right. \nonumber \\
& \left. + \left(\fr{N_0}{N_1}\right)^r\sum_{i:D_i=1}\left(\fr{f_{0,\hatth}(\pi(X_i,\hatth))}{f_{1,\hatth}(\pi(X_i,\hatth))}\right)^rQ_i  \right\}. \label{eq:wtilde_expansion}
\end{align}
The second term in \eqref{eq:varcondition} is bounded by
$$\bigOP{\normu{\hatth-\theta_0}}\left(1+\fr{1}{n}\sum_{i\in[n]}w_i(\hatth)^2\right).$$
Assumption \ref{ass:propscoredist_estimated}, bounding the ratios $f_{d,\theta}(p)/f_{1-d,\theta}(p)$ uniformly in $p\in[\underline p_{\theta},\ba p_{\theta}]$ and $\theta\in\ntheta$, combined with \eqref{eq:wtilde_expansion} and  $(N_{d}/N_{1-d})^r\convas\left(\fr{p_d}{1-p_d}\right)^r$, where $p_0\ceq 1-p_1$, bounds $\fr{1}{n}\sum_{i\in[n]}w_i(\hatth)^2$ by a $(1+\smallOP{1})$-term up to a constant factor. Thus, the second term in \eqref{eq:varcondition} is $\bigOP{\normu{\hatth-\theta_0}}=\smallOP{1}$ under \Cref{ass:theta_estimator}. Put $Q_i\ceq\sigma_{D_i}^2(\theta_0,\pi_i)$ and apply a mean-value expansion in $\theta$ around $\theta_0$ to the right side of \eqref{eq:wtilde_expansion}. This is feasible under \Cref{ass:parametric_propscore,ass:propscoredist_estimated}, which also bound the derivative uniformly in $(x,\tld\theta)\in\Xsupp\times\ntheta$. Then \Cref{ass:lip_conditional_var,ass:theta_estimator} imply that the first term in \eqref{eq:varcondition} is $V_{\sigma,\pi}+\smallOP{1}$ under \Cref{ass:parametric_propscore}, as in the proof of \Cref{lem:efficiency}, where $V_{\sigma,\pi}$ is finite by \Cref{ass:propscoredist_estimated,ass:lip_conditional_var}. Conclude that  \eqref{eq:mclt_var} holds.

Second, we verify the Lindeberg-condition \eqref{eq:mclt_lindeberg} of \Cref{lem:conditional_mclt}.
We need to show 
\begin{align*}
\sum_{i=1}^n\Ebc{\xi_{ni}^2\indic{|\xi_{ni}|\geq\eta}}{\filt_{n0}}\convprob 0\quad \text{for each } \eta>0.
\end{align*}
As $\indic{|\xi_{ni}|\geq\eta}$ is bounded by $\xi_{ni}^2/\eta^2$,
 \begin{align*}
\sum_{i=1}^n\Ebc{\xi_{ni}^2\indic{|\xi_{ni}|\geq\eta}}{\filt_{n0}}&\leq \sum_{t=1}^n\fr{\Ebc{\xi_{ni}^4}{\filt_{n0}}}{\eta^2}=\fr{1}{n\eta^2}\sum_{i\in[n]}\fr{(1+\tld w_i(\hatth))^4\Ebc{\varepsilon_i(\hatth)^4}{\filt_{n0}}}{n},
\end{align*}
with $\Ebc{\varepsilon_i(\hatth)^4}{\filt_{n0}}=\sigma_{D_i}^4(\hatth,\pi(X_i,\hatth))$ by \Cref{ass:theta_estimator}\ref{ass:theta_estimator:indep}, which is bounded uniformly in $i\in[n]$ by \Cref{ass:lip_conditional_var}. In view of \eqref{eq:wtilde_expansion},  $\fr{1}{n^2}\sum_{i\in[n]}(1+\tld w_i(\hatth))^4=\smallOP{1}$,  so \eqref{eq:mclt_lindeberg}  is met. 

Conclude that under the temporary assumption $\min_{i\in[n]} M_i(\hatth)>0$, Lemma \ref{lem:conditional_mclt} applies,  so \eqref{eq:tauhat_thetahat_decomp2} converges weakly to $\mathcal{N}(0,V_{\sigma,\pi})$ in probability. To remove this assumption, define the $\filt_{n0}$-measurable set $A_n\ceq\{\min_{i\in[n]} M_i(\hatth)>0\}$. On $A_n$, \eqref{eq:tauhat_thetahat_decomp2} is equal to $\sum_{i\in[n]}\xi_{ni}$. As $\prob{A_n}\to1$ by \Cref{prop:number_of_matches_estimated}, the desired convergence follows.

\emph{Vanishing \eqref{eq:tauhat_thetahat_decomp3}.}  By Assumption \ref{ass:differentiability_regression}, the $p\mapsto \mu^d(\theta,p)$ are continuous on a compact set $[\underline p_{\theta},\ba p_{\theta}]$ for all $\theta\in\ntheta$. By \Cref{ass:theta_estimator},  $\prob{\hatth\in\ntheta}\to 1$. By \Cref{prop:number_of_matches_estimated}, $\prob{A_n}\to1$, thus \eqref{eq:tauhat_thetahat_decomp3} is $\smallOP{1}$.

\emph{Vanishing \eqref{eq:tauhat_thetahat_decomp4}.}  By Assumption \ref{ass:differentiability_regression}, $|\mu^d(\theta,p')-\mu^d(\theta,p)|\leq L|p'-p|$ for all $p',p\in [\underline p_\theta, \ba p_\theta]$ for all $\theta\in\ntheta$. Then
$$\max_{i\in[n]}\max_{j\in\calJC_{\hatth}(i)}|\mu^{1-D_i}(\hatth,\pi(X_i,\hatth))-\mu^{1-D_i}(\hatth,\pi(X_j,\hatth))|\lesssim \max_{i\in[n]}\max_{j\in\calJC_{\hatth}(i)}|\pi(X_i,\hatth)-\pi(X_j,\hatth)|.$$
By the construction of $\calJC_{\hatth}(i)$, the right side is bounded by $\delta_n$, where $\sqrt{n}\delta_n\to 0$.

When the caliper is $\delta_n=\caliperDeltahat$, the above arguments continue to hold. Specifically, in showing the conditional convergence of \eqref{eq:tauhat_thetahat_decomp2}, $w_i(\hatth)$ is still $\mathcal{F}_{n,i-1}$-measurable; \eqref{eq:tauhat_thetahat_decomp3} is exactly zero, and, in \eqref{eq:tauhat_thetahat_decomp4}, $\sqrt{n}\delta_n\leq \sqrt{n}\left(\maxminDeltahat+\fr{\log n}{N_0+1}+\fr{\log n}{N_1+1}\right)=\smallOP{1}$ in view of the proof of \Cref{prop:number_of_matches_estimated}.
\end{proof}

%%%%%%%%%%%%%%%%%%%% 		AsymNorm Estimated propscore ATT		%%%%%%%%%%%%%%%%%%%%%%
\begin{proof}[Proof of \Cref{thm:asymnorm_estimated_propscore_att}]
Follows those of \Cref{thm:asymnorm_known_propscore_att} and \Cref{thm:asymnorm_estimated_propscore_ate}.
\end{proof}

%%%%%%%%%%%%%%%%%%%% 		Admissible models		%%%%%%%%%%%%%%%%%%%%%%
\begin{proof}[Proof of \Cref{prop:admissible_models_estimated_propscore}]
\Cref{ass:ucf,ass:parametric_propscore,ass:theta_estimator}\ref{ass:theta_estimator:indep} hold by construction as $\Xsupp$ is bounded.  The assumptions of \cref{ex:admissible_models} imply \Cref{ass:theta_estimator}\ref{ass:theta_estimator:asymnorm} by standard asymptotic theory \Citep{van_der_vaart_asymptotic_1998}.  We assume $K=2$ covariates in the following so that $\Xsupp=\Xsupp_1\times\Xsupp_2$; the general case $K\geq 2$ follows analogously. First, we establish some general results. Let $\theta_k$ ($\theta_{0,k}$) denote the $k$th entry of $\theta$ ($\theta_0$). For $t\in\mathcal{T}\ceq\left\{\theta^\intercal x:\theta\in\Theta,x\in\Xsupp\right\}$, $\theta^\intercal X$ has density $f_{\theta^\intercal X}(t)=\int_{\Xsupp_1}\Xdensity\left(x_1,\fr{t-\theta_1x_1}{\theta_2}\right)\deriv x_1$, which is strictly positive by assumptions on $\Xdensity$, the density of $X$.\footnote{If $X$ included an intercept, then we would have $f_{\theta^\intercal X}(t)=\int_{\Xsupp_1}\Xdensity\left(x_1,\fr{t-\theta_3-\theta_1x_1}{\theta_2}\right)\deriv x_1$, where $\theta_3$ is the coefficient on the intercept. Below, the right side of \eqref{eq:exph_givent}, $f_{\theta^\intercal X \mid D}(t\mid D)$ and \eqref{eq:exph_givendt} would need to be adjusted in a similar manner to accommodate an intercept.} Let $h:\mathcal{X}\to \real^J,J\geq 1$, be an arbitrary integrable function. We have
\begin{align}
 \Ebc{h(X_1, X_2)}{\theta^\intercal X=t} = \fr{1}{f_{\theta^\intercal X}(t)} \int_{\Xsupp_1} h\left(x_1, \fr{t-\theta_1x_1}{\theta_2}\right) \Xdensity\left(x_1,\fr{t-\theta_1x_1}{\theta_2}\right)\deriv x_1. \label{eq:exph_givent}
 \end{align}
 Combine this with the tower property of expectations to get the conditional density
 \begin{align*}
f_{\theta^\intercal X\mid D}(t\mid d) = \begin{cases}
 \fr{1}{1-p_1} \int_{\Xsupp_1} (1-g_{\theta_0}(x_1,\theta,t)) \Xdensity\left(x_1,\fr{t-\theta_1x_1}{\theta_2}\right)\deriv x_1 \quad &\text{ if } d=0, \\
  \fr{1}{p_1} \int_{\Xsupp_1} g_{\theta_0}(x_1,\theta,t)\Xdensity\left(x_1,\fr{t-\theta_1x_1}{\theta_2}\right)\deriv x_1 \quad & \text{ if } d=1,
\end{cases}
\end{align*}
where $g_{\theta_0}(x_1,\theta,t)\ceq g\left(\theta_{0,1}x_1+\theta_{0,2}\fr{t-\theta_1x_1}{\theta_2}\right)\in(0,1)$ as $0<g(t')<1$ for all $t'$ in the bounded $\mathcal{T}$. Hence, $f_{\theta^\intercal X\mid D}$ is strictly positive. It is also continuously differentiable in $t$ by assumptions on $g$ and $\Xdensity$. Moreover, $\theta\mapsto f_{\theta^\intercal X\mid D}(t\mid d)$ is continuously differentiable at $\tld
\theta\in\ntheta$ with bounded derivative for a $\theta_{0,2}\neq 0$ as $\sup_{t\in\real}g'(t)<\infty$ and the derivatives of $\Xdensity$ are bounded.  One can also show that $\Ebc{h(X_1, X_2)}{D=d,\theta^\intercal X=t}$ is
\begin{align}
&\fr{1-p_1}{f_{\theta^\intercal X\mid D}(t\mid0)}\int_{\Xsupp_1} h\left(x_1,\fr{t-\theta_1x_1}{\theta_2}\right)\left(1-g_{\theta_0}(x_1,\theta,t)\right)\Xdensity\left(x_1,\fr{t-\theta_1x_1}{\theta_2}\right)\deriv x_1\quad&\text{ if } d=0, \nonumber \\
&\fr{p_1}{f_{\theta^\intercal X\mid D}(t\mid1)}\int_{\Xsupp_1} h\left(x_1,\fr{t-\theta_1x_1}{\theta_2}\right)g_{\theta_0}(x_1,\theta,t)\Xdensity\left(x_1,\fr{t-\theta_1x_1}{\theta_2}\right)\deriv x_1 \quad&\text{ if } d=1,\label{eq:exph_givendt}
\end{align}
which is continuously differentiable in $t$ and $\theta$ by assumptions on $g,\Xdensity$, and the properties of $f_{\theta^\intercal X\mid D}$ derived above, provided $h$ is continuously differentiable. We are now ready to verify the remaining assumptions.

\emph{\Cref{ass:propscoredist_estimated}.} The distributions are
\begin{align*}
F_{d,\theta}(p)= \begin{cases}
 \fr{1}{1-p_1} \int_\Xsupp \indic{g(\theta^\intercal x)\leq p} (1-g(\theta_0^\intercal x)) \Xdensity(x)\deriv x \quad &\text{ if } d=0, \\
  \fr{1}{p_1} \int_\Xsupp \indic{g(\theta^\intercal x)\leq p} g(\theta_0^\intercal x)\Xdensity(x)\deriv x \quad & \text{ if } d=1.
\end{cases}
\end{align*}
Since $g$ is increasing and $\mathcal{T}_\theta\ceq\left\{\theta^\intercal x: x\in\Xsupp\right\}$ is compact, for all $\theta\in\Theta$ there exist $0<\underline p_\theta<\ba p_\theta<1$ such that $F_{d,\theta}(\underline p_\theta)=0$ and $F_{d,\theta}(\ba p_\theta)=1$. Specifically, $\underline p_\theta=g(\inf\mathcal{T}_\theta)$ and $\ba p _\theta=g(\sup \mathcal{T}_\theta)$. On $[\underline p_\theta,\ba p_\theta]$, the $F_{d,\theta}$ admit densities
$$f_{d,\theta}(p)=f_{\theta^\intercal X\mid D}(\ginv(p)\mid d) (\ginv)'(p)=\fr{f_{\theta^\intercal X\mid D}(\ginv(p)\mid d)}{g'(\ginv(p))},$$
where $\ginv$ is well-defined, as well as its derivative by the inverse function theorem. Then the assumptions on $g$ and the properties of $f_{\theta^\intercal X\mid D}$ imply that \Cref{ass:propscoredist_estimated} holds.

\emph{\Cref{ass:differentiability_regression}.} Under \Cref{ass:ucf}, the tower property of expectation gives
$$\mu^d(\theta,p)=\Ebc{Y}{D=d,\theta^\intercal X=\ginv(p)}=\Ebc{m_d(X)}{D=d,\theta^\intercal X=\ginv(p)}.$$
But then the properties of \eqref{eq:exph_givendt}, $g$ and $m_d$ imply that \Cref{ass:differentiability_regression} holds.

\emph{\Cref{ass:lip_conditional_var}.} The lower bound in \ref{ass:lip_conditional_var:lip} is satisfied by assumptions on $\nu_d$ in \Cref{ex:admissible_models}. The Lipschitz condition in \ref{ass:lip_conditional_var:lip} and \ref{ass:lip_conditional_var:bound} both follow by the mean-value theorem as the $\sigma_d^2$ and $\sigma_d^4$ are polynomials in terms of the form $\Ebc{h(X_1, X_2)}{D=d,\theta^\intercal X=\ginv(p)}$ for continuously differentiable functions $h$ by assumptions on $m_d$ and $\nu_d$, so \eqref{eq:exph_givendt} applies.
\end{proof}

\begin{proof}[Proof of \Cref{prop:admissible_models_known_propscore} ]
Follows that of \Cref{prop:admissible_models_estimated_propscore}.
\end{proof}

%%%%%%%%%%%%%%%%%%%% 		Ratios		%%%%%%%%%%%%%%%%%%%%%%
\begin{lemma}[Convergence of Ratios]
\label{lem:convergence_of_ratios}
Suppose that the caliper $\delta_n$ satisfies \eqref{eq:caliper}. For the measures in \eqref{eq:def_measures}, define
\begin{align}
R_{di}&\ceq \fr{\mathbbm{F}_{N_d}[\pi(X_i)\mypm\delta_n]}{F_{d}[\pi(X_i)\mypm\delta_n]}-1,   \label{eq:def_r} \\
\tld R_{di}&\ceq \fr{F_d[\pi(X_i)\mypm\delta_n]}{2\delta_n f_d(\pi(X_i))}-1, \label{eq:def_rtld} \\
\check R_{dji} &\ceq \fr{f_d(\pi(X_j))}{f_d(\pi(X_i))}-1 \label{eq:def_rcheck}
\end{align}
for $j\in\calJC(i)$, $i\in\set{i\in[n]:D_i=1-d}$ and $d\in\dummy$. If \Cref{ass:propscoredist} holds, then
\begin{enumerate}[label=(\roman*)]
\item \label{r_convergence} $\max_{i:D_i=1-d}|R_{di}|=\smallOP{1}$; and
\item \label{rtld_convergence} $\max_{i:D_i=1-d}|\tld R_{di}|=\smallOP{1}$; and
\item \label{rcheck_convergence} $\max_{i:D_i=1-d}\max_{j\in\calJC(i)}|\check R_{dji}|=\smallOP{1}$
\end{enumerate}
as $n\to\infty$ for all $d\in\dummy$. Suppose that \cref{ass:parametric_propscore} holds and the caliper $\delta_n$ satisfies \eqref{eq:caliper_hat}. For the measures in \eqref{eq:def_measures_theta}, define 
\begin{align}
R_{di}^{\theta}&\ceq \fr{\mathbbm{F}_{N_d,\theta}[\pi(X_i,\theta)\mypm\delta_n]}{F_{d,\theta}[\pi(X_i,\theta)\mypm\delta_n]}-1,   \label{eq:def_r_theta} \\
\tld R_{di}^{\theta}&\ceq \fr{F_{d,\theta}[\pi(X_i,\theta)\mypm\delta_n]}{2\delta_n f_{d,\theta}(\pi(X_i,\theta))}-1, \label{eq:def_rtld_theta} \\
\check R_{dji}^{\theta} &\ceq \fr{f_{d,\theta}(\pi(X_j,\theta))}{f_{d,\theta}(\pi(X_i,\theta))}-1 \label{eq:def_rcheck_theta}
\end{align}
for $j\in\calJC_{\theta}(i)$, $i\in\set{i\in[n]:D_i=1-d}$, $d\in\dummy$ and $\theta\in\ntheta$. If \Cref{ass:propscoredist_estimated} holds, then for any $\hatth$ satisfying \Cref{ass:theta_estimator},
\begin{enumerate}[label=(\roman*)]
\setcounter{enumi}{3}
\item \label{r_convergence_theta} $\max_{i:D_i=1-d}|R_{di}^{\hatth}|=\smallOP{1}$; and
\item \label{rtld_convergence_theta} $\max_{i:D_i=1-d}|\tld R_{di}^{\hatth}|=\smallOP{1}$; and
\item \label{rcheck_convergence_theta} $\max_{i:D_i=1-d}\max_{j\in\calJC_{\hatth}(i)}|\check R_{dji}^{\hatth}|=\smallOP{1}$
\end{enumerate}
as $n\to\infty$ for all $d\in\dummy$.
\end{lemma}
\begin{proof}
Assertion \ref{r_convergence}. Consider
\begin{align}
|R_{1i}|\leq \sup_{p\in[\underline p, \ba p]}\bigg\vert\fr{\mathbbm{F}_{N_1}[p\mypm\delta_n]}{F_1[p\mypm\delta_n]}-1\bigg\vert =\sup_{p\in[\underline p, \ba p]}\fr{|\gproc_{N_1}[p\mypm\delta_n]|}{\sqrt{N_1}F_1[p\mypm\delta_n]}\eqc W_1.\label{eq:rbound}
\end{align}
Fix a constant $\zeta>0$. We bound $\prob{W_1>\zeta \mid \Dvector}$ using ratio and tail bounds of empirical processes. To this end, note that for any finite $\delta_n>0$, $\mathcal{C}_{\delta_n}\ceq \{[p\mypm\delta_n]:p\in[\underline p, \ba p]\}$ is a VC-class, with VC-dimension equal to two (e.g. \Citet[Example 2.6.1]{van_der_vaart_weak_1996}). Let $\gamma_{n}\ceq 2\delta_n \inf_{p\in[\underline p,\ba p]}f_1(p)$, so that $\inf_{p\in[\underline p,\ba p]}F_1[p\mypm\delta_n]\geq \gamma_{n}$. The event $\{W_1> \zeta\}$ is equal to 
\begin{align}
 \left\{\sup\left\{\fr{|\gproc_{N_1}[p\mypm\delta_n]|}{F_1[p\mypm\delta_n]}: p\in[\underline p, \ba p], F_1[p\mypm\delta_n]\geq \gamma_{n}\right\}>\sqrt{N_1}\zeta\right\} \subset \nonumber \\
\left( \left\{\sup\left\{\fr{|\gproc_{N_1}[p\mypm\delta_n]|}{F_1[p\mypm\delta_n]}: p\in[\underline p, \ba p], F_1[p\mypm\delta_n]\geq \gamma_{n}, F_1[p\mypm\delta_n]\leq \fr{1}{2} \right\}>\sqrt{N_1}\zeta\right\} \right. \cup \label{eq:unionset1} \\
  \left. \left\{\sup\left\{\fr{|\gproc_{N_1}[p\mypm\delta_n]|}{F_1[p\mypm\delta_n]}: p\in[\underline p, \ba p], F_1[p\mypm\delta_n]\geq \gamma_{n}, F_1[p\mypm\delta_n]> \fr{1}{2} \right\}>\sqrt{N_1}\zeta\right\} \right).\label{eq:unionset2}
\end{align}
First, we bound the probability of the event in \eqref{eq:unionset1}. Take $A\ceq\{p\in[\underline p, \ba p]: F_1[p\mypm\delta_n]\geq\gamma_{n}\}$, $B\ceq\{p\in[\underline p, \ba p]: F_1[p\mypm\delta_n]\leq \fr{1}{2}\}$.  If $\gamma_{n}>\fr{1}{2}$, $A\cap B$ is empty and by the convention $\sup\emptyset=-\infty$, the set \eqref{eq:unionset1} has measure zero. So assume without loss of generality that $\gamma_{n}\leq F_1[p\mypm\delta_n]\leq \fr{1}{2}$. Thus, $\fr{\gamma_{n}}{2}\leq \sigma_1^2[p\mypm\delta_n]\ceq (F_1[p\mypm\delta_n])(1-F_1[p\mypm\delta_n])\leq \fr{1}{4}$, where note that $\sigma_1^2[p\mypm\delta_n]< F_1[p\mypm\delta_n]$ for $F_1[p\mypm\delta_n]\geq\gamma_{n}>0$. As $N_1\gamma_n\convas\infty$ and $N_1^{-1}\log(e\vee\log(e\vee N_1))=\smallO{\gamma_n}$ a.s., conditions (2.2)--(2.3)  in \citet[Theorem 2.1]{alexander_rates_1987} hold a.s.. Therefore, the bounds in the proof of Theorem 5.1 therein apply. Hence, as in (7.66)--(7.67) of \cite{alexander_rates_1987},
\begin{align*}
\probc{\sup\left\{\fr{|\gproc_{N_1}[p\mypm\delta_n]|}{F_1[p\mypm\delta_n]}: p\in[\underline p, \ba p], F_1[p\mypm\delta_n]\geq \gamma_{n}, F_1[p\mypm\delta_n]\leq \fr{1}{2} \right\}>\sqrt{N_1}\zeta}{\Dvector} \\
\leq \probc{|\gproc_{N_1}[p\mypm\delta_n]|>(\sigma_1^2[p\mypm\delta_n])\sqrt{N_1}\zeta\text{ for some } p\in[\underline p, \ba p]: \fr{\gamma_{n}}{2}\leq \sigma_1^2[p\mypm\delta_n]\leq \fr{1}{4}}{\Dvector} \\
\leq 36\int_{\gamma_{n}/2}^{1/4}t^{-1}e^{-\zeta^2N_1t/512} \deriv t+68e^{-\zeta N_1\gamma_{n}/256} \\
%\leq 36 \int_{\gamma_{n}/2}^{1/4}\left(\fr{\gamma_{n}}{2}\right)^{-1}e^{-x^2N_1t/512} \deriv t +68e^{-xN_1\gamma_{n}/256} \\
% =\fr{36}{-x^2N_1\gamma_{n}}\left(e^{-x^2N_1/(4\cdot 512)}-e^{-x^2N_1\gamma_{n}/(2\cdot 512)}\right)+68e^{-xN_1\gamma_{n}/256} \\
\leq \fr{36}{\zeta^2N_1\gamma_{n}} e^{-\zeta^2N_1\gamma_n/1024} +68e^{-\zeta N_1\gamma_{n}/256}.
\end{align*}
Second, the probability of the event in \eqref{eq:unionset2} is bounded by
\begin{align}
\probc{\sup \left\{\fr{|\gproc_{N_1}[p\mypm\delta_n]|}{F_1[p\mypm\delta_n]}:p\in[\underline p, \ba p], F_1[p\mypm\delta_n]>\fr{1}{2}\right\}>\sqrt{N_1}x}{\Dvector} \nonumber \\
\leq \probc{\sup \left\{|\gproc_{N_1}[p\mypm\delta_n]|:p\in[\underline p, \ba p], F_1[p\mypm\delta_n]>\fr{1}{2}\right\}>\fr{\sqrt{N_1}}{2}x}{\Dvector}\nonumber  \\
\leq \probc{\sup_{p\in[\underline p, \ba p]}|\gproc_{N_1}[p\mypm\delta_n]|>\fr{\sqrt{N_1}}{2}x}{\Dvector} \label{eq:lindeberg_tailprob}
\end{align}
as the supremum over a larger set cannot decrease. \Citet[Theorem 2.14.9]{van_der_vaart_weak_1996}  bound \eqref{eq:lindeberg_tailprob} by $ cN_1\zeta^2e^{-\fr{N_1\zeta^2}{2}}$. Therefore,
\begin{align}
\probc{W_1>\zeta}{\Dvector}\leq \fr{c_1}{\zeta^2N_1\gamma_{n}}e^{-\zeta^2N_1\gamma_{n}/1024}+c_2 e^{-\zeta N_1\gamma_{n}/256}+c_3 N_1\zeta^2 e^{-N_1 \zeta^2/2}, \label{eq:wbound}
\end{align}
on the set where $N_0,N_1\geq 1$, which happens with probability tending to one. The left side is bounded by one, the right side converges to zero in probability; then the left side also converges to zero in expectation. Similar arguments hold for $W_0\ceq \sup_{p\in[\underline p, \ba p]}\fr{|\gproc_{N_0}|[p\mypm \delta_n]|}{\sqrt{N_0} F_0[p\mypm\delta_n]}$, bounding $\max_{i:D_i=1}|R_{0i}|$. Conclude that $\max_{i:D_i=1-d}|R_{di}|=\smallOP{1}$. When the caliper is $\delta_n=\caliperDelta$, the same arguments yield the assertion by letting $\gamma_n\ceq \gamma_{N_d}\ceq 2\inf_{p\in[\underline p,\ba p]}f_d(p)\fr{\log N_d}{N_d+1} $ when bounding $W_d$, $d\in\dummy$.

Assertion \ref{rtld_convergence}. By \Cref{ass:propscoredist}, the $f_d$ are continuous on the compact set $[\underline p,\ba p]$, hence are uniformly continuous. By the mean-value theorem, uniform continuity of $f_d$ implies uniform differentiability of $F_d$. By uniform differentiability of $F_0$, $\sup_{j:D_j=1}|F_0[\pi(X_j)\mypm\delta_n]-2\delta_n f_0(\pi(X_j))|=\smallOP{\delta_n}$ for $\delta_n=\smallOP{1}$. To see this, fix a constant $\zeta>0$. By uniform differentiability of $F_0$, for all $\zeta$ there exists a constant $\ba \delta>0$ such that 
$$\sup_{p\in[\underline p, \ba p]}\fr{|F_0(p+\delta_n)-F_0(p)-\delta_n f_0(p)|}{\delta_n}\leq \zeta$$
whenever $\delta_n\leq \ba\delta$. The event $\{\sup_{p\in[\underline p, \ba p]}\fr{|F_0(p+\delta_n)-F_0(p)-\delta_n f_0(p)|}{\delta_n}> \zeta\}$ is equal to
\begin{align*} 
\set{\sup_{p\in[\underline p, \ba p]}\fr{|F_0(p+\delta_n)-F_0(p)-\delta_n f_0(p)|}{\delta_n}> \zeta,\delta_n\leq\ba\delta} \\\cup\set{\sup_{p\in[\underline p, \ba p]}\fr{|F_0(p+\delta_n)-F_0(p)-\delta_n f_0(p)|}{\delta_n}> \zeta, \delta_n>\ba\delta}.
\end{align*}
The first event has measure zero by uniform differentiability. The probability of the second event is dominated by $\prob{\delta_n>\ba\delta}$, which is $\smallO{1}$. Then the statement follows by noting that $F_0[p\mypm\delta_n]=F_0(p+\delta_n)-F_0(p)+F_0(p)-F_0(p-\delta_n)$ and that $\max_{i\in[n]}\max_{j\in\calJC(i)}|f_0(\pi(X_i))-f_0(\pi(X_j))|=\smallOP{1}$ (see proof of Assertion \ref{rcheck_convergence}), so $2f_0(p)=f_0(p)+f_0(p-\delta_n)+f_0(p)-f_0(p-\delta_n)=f_0(p)+f_0(p-\delta_n)+\smallOP{1}$. As $\inf_{p\in[\underline p, \ba p]} f_0(p)>0$ by \Cref{ass:propscoredist}, Assertion \ref{rtld_convergence} follows. When the caliper is $\delta_n=\caliperDelta$, \Cref{prop:maxmindistorder}, the continuous mapping theorem and the law of large numbers imply $\delta_n=\smallOP{1}$, whence the assertion follows by the above arguments.

Assertion \ref{rcheck_convergence}. As $f_0$ is uniformly continuous by \Cref{ass:propscoredist}, 
$$\max_{i\in[n]}\max_{j\in\calJC(i)}|f_0(\pi(X_j))-f_0(\pi(X_i))|=\smallOP{1}.$$ 
To see this, fix a constant $\zeta>0$. By uniform continuity of $f_0$, for all $\zeta$ there exists an $\eta>0$ such that $|f_0(p)-f_0(p')|\leq\zeta$ whenever $|p-p'|\leq\eta$ for all $p,p'\in[\underline p, \ba p]$. The event $\{|f_0(\pi(X_i))-f_0(\pi(X_j))|>\zeta\}$ is equal to $\{|f_0(\pi(X_i))-f_0(\pi(X_j))|>\zeta,|\pi(X_i)-\pi(X_j)|\leq\eta\}\cup\{|f_0(\pi(X_i))-f_0(\pi(X_j))|>\zeta, |\pi(X_i)-\pi(X_j)|>\eta\}$. The first event has measure zero by uniform continuity. As $j\in\calJC(i)$, the probability of the second event is dominated by $\prob{|\pi(X_i)-\pi(X_j)|>\eta}\leq\prob{\delta_n>\eta}$, which is $\smallO{1}$. Hence $\max_{i:D_i=0}\max_{j\in\calJC(i)}|\check R_{0ji}|=\smallOP{1}$, because $\inf_{p\in[\underline p, \ba p]} f_0(p)>0$ by \Cref{ass:propscoredist}. When the caliper is $\delta_n=\caliperDelta$, arguments proving Assertion \ref{rtld_convergence} apply.

Assertions \ref{r_convergence_theta}, \ref{rtld_convergence_theta}, \ref{rcheck_convergence_theta} follow along the same arguments by conditioning on $\hatth$, exploiting \Cref{ass:propscoredist_estimated,ass:theta_estimator} and that the constants of the bounds of \cite{alexander_rates_1987} and \Citet[]{van_der_vaart_weak_1996} do not depend on the underlying distribution. This holds for any caliper choice in \eqref{eq:caliper_hat} as $\maxminDeltahat=\smallOP{1}$ by the proof of \Cref{prop:number_of_matches_estimated}.
\end{proof}

%%%%%%%%%%%%%%%%%%%% 		Error term		%%%%%%%%%%%%%%%%%%%%%%
\begin{lemma}[Error Term]
\label{lem:error}
Suppose \Cref{ass:propscoredist,ass:variance} hold, and the caliper $\delta_n$ satisfies \eqref{eq:caliper}. Then 
$$\sup_{x\in\real}\bigg\vert \probc{V_E^{-1/2}\sqrt{n}\eCDO\leq x}{\Dvector, \Pvector}-\Phi(x)\bigg\vert\convprob 0 \quad\text{ as $n\to\infty$},$$
for $V_E\ceq \fr{1}{n}\sum_{i\in[n]}(\indic{M_i>0}+w_i)^2\sigma_{D_i}^2(\pi(X_i))$ and standard normal distribution function $\Phi$.
\end{lemma}
\begin{proof}
We apply a Lindeberg--Feller central limit theorem as $E_i$, given $(\Dvector, \Pvector)$, are independently, but not identically, distributed with mean zero across $i\in[n]$ (the $M_i,w_i$ are constants given $(\Dvector,\Pvector)$). By \Cref{ass:propscoredist}, the $\mu^d$, and hence $\varepsilon$, are well-defined. By definition of $V_E$, $\Vb{\sum_{i\in[n]}E_i/\sqrt{nV_E}\mid \Dvector, \Pvector}=1$. Thus, we only need to verify the Lindeberg--Feller condition:
\begin{align}
\sum_{i\in[n]}\Ebc{(E_i/\sqrt{nV_E})^2\indic{|E_i/\sqrt{nV_E}|\geq \eta}}{\Dvector, \Pvector}\convprob 0 \quad \text{for all constants } \eta>0. \label{eq:error_lindeberg}
\end{align}
Since $\indic{|E_i/\sqrt{nV_E}|\geq \eta}$ is bounded by $E_i^2/(\eta^2nV_E)$ on $\set{E_i^2/(\eta^2nV_E)\geq 1}$, the expectation in \eqref{eq:error_lindeberg} satisfies
\begin{align}
\fr{1}{nV_E}\Ebc{E_i^2\indic{|E_i|\geq \eta\sqrt{nV_E}}}{\Dvector, \Pvector} &\leq\fr{1}{nV_E}\Ebc{E_i^4/(\eta^2nV_E)}{\Dvector, \Pvector}  \nonumber \\ 
&=\fr{\Ebc{E_i^4}{\Dvector, \Pvector}}{(\eta n V_E)^2}, \label{eq:error_lindebergbound}
\end{align}
where $\Ebc{E_i^4}{\Dvector, \Pvector}=(\indic{M_i>0}+ w_i)^4\Ebc{\varepsilon_i^4}{D_i, \pi(X_i)}$, with $\Ebc{\varepsilon_i^4}{D_i, \pi(X_i)}\leq \sup_{d\in\bin,p\in[0,1]}\Eb{\varepsilon_i^4\mid D_i=d, \pi(X_i)=p}<\infty$ by \Cref{ass:variance}. By \Cref{lem:efficiency},  $V_E\convprob V_{\sigma,\pi}\in(0,\infty)$, so \eqref{eq:error_lindebergbound} is bounded by 
$(1+\smallOP{1})n^{-1}\left(\nsumn{i}(\indic{M_i>0}+w_i)^4\right)$ up to a constant factor. Fix a constant $C>0$. By Markov's inequality,
$$\prob{\nsumn{i}(\indic{M_i>0}+w_i)^4>C}\leq C^{-1}\max_{i\in[n]}\E(\indic{M_i>0}+w_i)^4.$$
Below, we show that $n^{-\varrho}\max_{i\in[n]}\E(\indic{M_i>0}+w_i)^4=\bigO{1}$ for any $\varrho>0$, so that $$\nsumn{i}(\indic{M_i>0}+w_i)^4=\bigOP{n^{\varrho}}.$$ Then \eqref{eq:error_lindebergbound} is bounded by $(1+\smallOP{1})\bigOP{n^{\varrho-1}}$, which is $\smallOP{1}$ for $\varrho<1$, so \eqref{eq:error_lindeberg} is met.

As $w_i\geq0$ and $(1+x)^4\leq (2x)^4$ for $x\geq 1$, 
$$(\indic{M_i>0}+w_i)^4\leq 2^4+2^4 w_i^4\indic{w_i>1}.$$
We have $w_i\leq M_i \max_{j\in\calJC(i)}M_j^{-1}$ and hence 
$$\E w_i^4\indic{w_i>1}\leq \Eb{\indic{w_i>1}M_i^4\max_{j\in\calJC(i)}M_j^{-4}}.$$
 This yields the result by \Cref{lem:lindeberg_feller_bound}.
 
When the caliper is $\delta_n=\caliperDelta$, it is $\sigma\set{\Dvector,\Pvector}$-measurable, hence the $M_i,w_i$ are constants given $(\Dvector,\Pvector)$. \Cref{lem:lindeberg_feller_bound,lem:efficiency} complete the proof.
\end{proof}

%%%%%%%%%%%%%%%%%%%% 		Lindeberg-Feller Bound		%%%%%%%%%%%%%%%%%%%%%%
\begin{lemma}[Lindeberg--Feller Bound]
\label{lem:lindeberg_feller_bound}
Suppose that the caliper $\delta_n$ satisfies \eqref{eq:caliper} and that \Cref{ass:propscoredist} holds. Then for any finite fixed constant integer $r\geq 2$ and any finite fixed constant $\varrho>0$,
\begin{align} 
\max_{i\in[n]} \E{\indic{w_i>1}\max_{j\in\calJC(i)}\left(\fr{M_i}{M_j}\right)^{r}}=\smallO{n^\varrho}. \label{eq:lindebergbound}
\end{align}
\end{lemma}
\begin{proof} 
If $w_i>1$ for some $i\in[n]$, then $M_i\geq1$ and hence also $M_j\geq 1$ for all $j\in\calJC(i)$, as $j\in\calJC(i)$ if and only if $i\in\calJC(j)$. This in turn implies $N_0,N_1\geq 1$. Thus, $\indic{w_i>1}\leq \indic{M_j\geq 1}\indic{N_0\geq 1}\indic{N_1\geq 1}$ for all $j\in\calJC(i)$ for all $i\in[n]$. By definition, $M_i=N_{1-d}\mathbbm{F}_{N_{1-d}}[\pi(X_i)\mypm\delta_n]$ if $D_i=d$. Then, in the notation of \eqref{eq:def_measures} and \eqref{eq:def_r}, the left side of \eqref{eq:lindebergbound} is bounded by
\begin{align}
\sum_{d\in\dummy}\max_{i:D_i=d} \Eb{\left(\fr{N_{1-d}\indic{N_d\geq1}}{N_d}\right)^r \max_{j\in\calJC(i)}\left(\fr{F_{1-d}[\pi(X_i)\mypm\delta_n]}{F_d[\pi(X_j)\mypm\delta_n]}\fr{1+R_{1-d,i}}{1+R_{d,j}}\indic{N_{1-d}\geq1}\indic{M_j\geq 1}\right)^r} \nonumber \\
\lesssim \sum_{d\in\dummy}\max_{i:D_i=d} \Eb{\left(\fr{N_{1-d}\indic{N_d\geq1}}{N_d}\right)^r \max_{j\in\calJC(i)}\left(\fr{1+R_{1-d,i}}{1+R_{d,j}}\indic{N_{1-d}\geq1}\indic{M_j\geq 1}\right)^r}, \label{eq:lindebergbound_bound1}
\end{align}
where the second line follows from \Cref{ass:propscoredist}: because $f_0,f_1$ have the same support, are bounded away from zero and infinity, we have, for $c\neq0$, $0<\fr{2c\inf_{p\in[\underline p,\ba p]} f_{1-d}(p)}{2c\sup_{p\in[\underline p, \ba p]} f_d(p)}\leq F_{1-d}[a\mypm c]/F_d[b\mypm c]\leq \fr{2c\sup_{p\in[\underline p,\ba p]} f_{1-d}(p)}{2c\inf_{p\in[\underline p, \ba p]} f_d(p)}<\infty$. We address the case $d=0$ in \eqref{eq:lindebergbound_bound1}; $d=1$ follows by symmetry. For the expectation in \eqref{eq:lindebergbound_bound1}, two applications of the Cauchy--Schwarz inequality give
\begin{align}
 \Eb{\left(\fr{N_{1}\indic{N_0\geq1}}{N_0}\right)^r \max_{j\in\calJC(i)}\left(\fr{1+R_{1i}}{1+R_{0j}}\indic{N_{1}\geq1}\indic{M_j\geq 1}\right)^r} \nonumber \\
  \leq \sqrt{\E(N_1\indic{N_0\geq1}/N_0)^{2r}\sqrt{\E(1+R_{1i})^{4r}\indic{N_{1}\geq1}\E\left(\max_{j\in\calJC(i)}(\indic{M_j\geq 1}/(1+R_{0j}))^{r}\right)^4}}. \label{eq:lindebergbound_bound2}
\end{align}
Here $\E(N_1\indic{N_0\geq1}/N_0)^{2r}\leq n^{2r}\E(\indic{N_0\geq1}/N_0)^{2r}$ with
\begin{align*}
\E(\indic{N_0\geq1}/N_0)^{2r} &= \sum_{n_0=1}^n {n \choose n_0}(1-p_1)^{n_0}p_1^{n-n_0}\left(\fr{1}{n_0}\right)^{2r} \\
&= \sum_{n_0=0}^{n-1} \fr{n}{n_0+1}{n-1 \choose n_0}(1-p_1)^{n_0+1}p_1^{n-(n_0+1)}\left(\fr{1}{n_0+1}\right)^{2r} \\
&= n \fr{1-p_1}{p_1}\sum_{n_0=0}^{n-1}{n-1 \choose n_0}(1-p_1)^{n_0}p_1^{n-n_0}\left(\fr{1}{n_0+1}\right)^{2r+1}.
\end{align*}
The last line is $\bigO{n^{-2r}}$ by \cite{cribari-neto_note_2000}, thus $\E(N_1\indic{N_0\geq1}/N_0)^{2r}=\bigO{1}$. 

Next, we bound the second factor of \eqref{eq:lindebergbound_bound2}. We have $\E(1+R_{1i})^{4r}\indic{N_{1}\geq1}\leq 2^{4r}+c\E |R_{1i}|^{4r}\indic{N_{1}\geq1}$ for some constant $c>0$, where $\max_{i:D_i=0}|R_{1i}|\leq W_1$ for $W_1$ in \eqref{eq:rbound}. Because $W_1\geq0$, 
$$\Ebc{W_1^{4r}}{\Dvector}=\int_0^\infty \prob{W_1^{4r}>w\mid \Dvector}\deriv w=\int_0^\infty \prob{W_1>w^{\fr{1}{4r}}\mid \Dvector}\deriv w.$$
By \eqref{eq:wbound} of \Cref{lem:convergence_of_ratios},
\begin{align}
\Ebc{W_1^{4r}}{\Dvector} \leq & \fr{c_0}{N_1\gamma_{n}}\int_{0}^\infty \fr{1}{w^{\fr{2}{4r}}}e^{-w^{\fr{2}{4r}}N_1\gamma_{n}/1024}\deriv w+c_1\int_0^\infty e^{-w^{\fr{1}{4r}}N_1\gamma_{n}/256}\deriv w 
\label{eq:expectedwbound1} \\
&+ cN_1\int_0^\infty {w^{\fr{2}{4r}}} e^{-\fr{N_1w^{\fr{2}{4r}}}{2}}\deriv w.\label{eq:expectedwbound2}
\end{align}
The integral in the first term of \eqref{eq:expectedwbound1} is $\fr{c_0}{N_1\gamma_{n}}\fr{1}{\lambda_{n,N_1}}\int_0^\infty t^{-1}\lambda_{n,N_1}e^{-\lambda_{n,N_1}t}t^{2r-1}\deriv t$, where $\lambda_{n,N_1}\ceq \fr{N_1\gamma_n}{1024}$ is strictly positive for $N_1\geq1$. This integral is the $(2r-2)$th moment of an Exponential$(\lambda_{n,N_1})$ variable, which is well-defined for $r\geq 2$ and for finite integer $r$ is bounded by $\left(\fr{1}{\lambda_{n,N_1}}\right)^{2r-2}$ up to a constant factor. Hence, the first term of \eqref{eq:expectedwbound1} is bounded by $c_0\left(\fr{1}{N_1\gamma_{n}}\right)^{2r}\simeq c_0\left(\fr{n}{N_1\log n}\right)^{2r}$. Similar arguments show that the second integrals of \eqref{eq:expectedwbound1} and \eqref{eq:expectedwbound2} are bounded by $c_1\left(\fr{1}{N_1\gamma_{n}}\right)^{4r}$ and $c\left(\fr{1}{N_1}\right)^{2r}$, respectively. Conclude that
\begin{align*}
\E(1+ R_{1i})^{4r}\indic{N_{1}\geq1} \leq \bigO{1}+c_0 n^{2r}\E \left(\fr{\indic{N_{1}\geq1}}{N_1}\right)^{2r} + c_1 n^{4r}\E \left(\fr{\indic{N_{1}\geq1}}{N_1}\right)^{4r}.
\end{align*}
By arguments bounding $\E (N_1\indic{N_0\geq 1}/N_0)^{2r}$ of \eqref{eq:lindebergbound_bound2}, the right side is $\bigO{1}$.

Finally, the last factor of \eqref{eq:lindebergbound_bound2} satisfies 
\begin{align}
\max_{j\in\calJC(i)}(\indic{M_j\geq 1}/(1+R_{0j}))^{r} =\max_{j\in\calJC(i)} \left(\fr{N_0F_0[\pi(X_j)\mypm\delta_n]}{M_j}\indic{M_j\geq 1}\right)^r 
& \leq (2n\supnorm{f_0}\delta_n)^r \label{eq:lindebergbound_bound2_bound} \\
&\lesssim (\log n)^r. \nonumber
\end{align}
As $(\log n)^r=\smallO{n^\varrho}$ for any $\varrho>0$, \eqref{eq:lindebergbound} follows.

When the caliper is $\delta_n=\caliperDelta$, \eqref{eq:expectedwbound1} holds with $\gamma_n\ceq 2\inf_{p\in[\underline p,\ba p]}f_1(p)\fr{\log N_1}{N_1+1}$ in view of the proof of \Cref{lem:convergence_of_ratios}, showing $\E(1+ R_{1i})^{4r}\indic{N_{1}\geq1}=\bigO{1}$.
Consider \eqref{eq:lindebergbound_bound2_bound}, wherein $\delta_n^r\leq 2^r\left(\maxminDelta^r+(2\log n)^r\left(\fr{1}{N_0+1}\right)^r+\left(\fr{1}{N_1+1}\right)^r\right)$. Here, $\E\left(\fr{1}{N_d+1}\right)^r=\bigO{n^{-r}}$ by \cite{cribari-neto_note_2000}. Arguments in the proof of \Cref{prop:maxmindistorder} and \Cref{lem:spacingsmoments} imply $n^r\E\maxminDelta^r=\bigO{(\log n)^r}$. Conclude that $n^r\E \delta_n^r=\bigO{(\log n)^r}$ and hence \eqref{eq:lindebergbound} follows from \eqref{eq:lindebergbound_bound2_bound}.
\end{proof}

%%%%%%%%%%%%%%%%%%%% 		Efficiency		%%%%%%%%%%%%%%%%%%%%%%
\begin{lemma}[Semiparametric Efficiency]
\label{lem:efficiency}
If \Cref{ass:propscoredist,ass:variance} hold, and the caliper $\delta_n$ satisfies \eqref{eq:caliper}, then $V_E\convprob V_{\sigma,\pi}$ as $n\to\infty$.
\end{lemma}
\begin{proof}
By \Cref{prop:number_of_matches}, $\min_{i\in[n]}M_i>0$ with probability tending to one, so with probability tending to one,
\begin{align}
V_E&= \fr{1}{n}\sum_{i\in[n]}(1+ w_i)^2\sigma_{D_i}^2(\pi(X_i)) 
	=\fr{1}{n}\sum_{i\in[n]}(1+2 w_i+ w_i^2)\sigma_{D_i}^2(\pi(X_i)) \nonumber \\
	&=\E \sigma_D^2(\pi(X))+\smallOP{1}+\fr{2}{n}\sum_{i\in[n]} w_i\sigma_{D_i}^2(\pi(X_i))+\fr{1}{n}\sum_{i\in[n]} w_i^2\sigma_{D_i}^2(\pi(X_i)) \label{eq:variance_efficiency}
\end{align}
by the law of large of large numbers. In the notation of \eqref{eq:def_measures}, we have, by definition,
\begin{align}
\sum_{i\in[n]}w_i^r\sigma_{D_i}^2(\pi(X_i))=&\, \left(\fr{N_1}{N_0}\right)^r\sum_{i:D_i=0}\left(\fr{1}{N_1}\sum_{j\in\calJC(i)}\fr{1}{\mathbbm{F}_{N_0}[\pi(X_j)\mypm\delta_n]}\right)^r\sigma_{D_i}^2(\pi(X_i))\nonumber \\
	&+ \left(\fr{N_0}{N_1}\right)^r\sum_{i:D_i=1}\left(\fr{1}{N_0}\sum_{j\in\calJC(i)}\fr{1}{\mathbbm{F}_{N_1}[\pi(X_j)\mypm\delta_n]}\right)^r\sigma_{D_i}^2(\pi(X_i)) \label{eq:variance_efficiency2}.
\end{align}
Write $\mathbbm{F}_{N_0}[\pi(X_j)\mypm\delta_n]=(1+R_{0j})F_0[\pi(X_j)\mypm\delta_n]$ and $F_0[\pi(X_j)\mypm\delta_n]=2\delta_n f_0(\pi(X_j))(1+\tld R_{0j})$ for $R_{0j},\tld R_{0j}$ of \eqref{eq:def_r}, \eqref{eq:def_rtld}. By \Cref{lem:convergence_of_ratios},  
\begin{align*}
\max_{j:D_j=1}|R_{0j}|=\smallOP{1}\quad\text{ and }\quad \max_{j:D_j=1}|\tld R_{0j}|=\smallOP{1}.
\end{align*}
Then we can write
\begin{align*}
\fr{1}{N_1}\sum_{j\in\calJC(i)}\fr{1}{\mathbbm{F}_{N_0}[\pi(X_j)\mypm\delta_n]} &=  \fr{1}{N_1}\sum_{j\in\calJC(i)}\fr{1}{F_0[\pi(X_j)\mypm\delta_n]}\fr{1}{1+R_{0j}} \nonumber \\
&=\fr{1}{N_1}\sum_{j\in\calJC(i)}\fr{1}{2\delta_n f_0(\pi(X_j))}\fr{1}{1+\tld R_{0j}}\fr{1}{1+R_{0j}} \nonumber \\
& = \fr{1+\smallOP{1}}{N_1}\sum_{j\in\calJC(i)}\fr{1}{2\delta_n f_0(\pi(X_j))}
\end{align*}
where the $\smallOP{1}$ terms are uniform in $i\in[n]$. Write $f_0(\pi(X_j))=f_0(\pi(X_i))(1+\check R_{0ji})$ for $\check R_{0ji}$ of \eqref{eq:def_rcheck}. By \Cref{lem:convergence_of_ratios}, $\max_{i:D_i=0}\max_{j\in\calJC(i)}|\check R_{0ji}|=\smallOP{1}$. Then, by the continuous mapping theorem,
\begin{align*}
\fr{1}{N_1}\sum_{j\in\calJC(i)}\fr{1}{2\delta_n f_0(\pi(X_j))}&=\fr{1}{N_1}\sum_{j\in\calJC(i)}\fr{1}{2\delta_n f_0(\pi(X_i))}\fr{1}{1+\check R_{0ji}} \\
&=(1+\smallOP{1})\fr{f_1(\pi(X_i))}{f_0(\pi(X_i))}\fr{\mathbbm{F}_{N_1}[\pi(X_i)\mypm\delta_n]}{2\delta_n f_1(\pi(X_i))},
\end{align*}
because $\inf_{p\in[\underline p, \ba p]} f_1(p)>0$ by \Cref{ass:propscoredist}. Write 
$$\fr{\mathbbm{F}_{N_1}[\pi(X_i)\mypm\delta_n]}{2\delta_n f_1(\pi(X_i))}=\fr{\mathbbm{F}_{N_1}[\pi(X_i)\mypm\delta_n]}{F_1[\pi(X_i)\mypm\delta_n]}\fr{F_1[\pi(X_i)\mypm\delta_n]}{2\delta_n f_1(\pi(X_i))},$$
where $\max_{i:D_i=0}\abs{\fr{\mathbbm{F}_{N_1}[\pi(X_i)\mypm\delta_n]}{F_1[\pi(X_i)\mypm\delta_n]}-1}=\smallOP{1}$ and $\max_{i:D_i=0}\abs{\fr{F_1[\pi(X_i)\mypm\delta_n]}{2\delta_n f_1(\pi(X_i))}-1}=\smallOP{1}$ by \Cref{lem:convergence_of_ratios}. By symmetry, similar arguments apply to the second term of \eqref{eq:variance_efficiency2}. Then  \eqref{eq:variance_efficiency2}, divided by $n$, is equal to
\begin{align}
\fr{1+\smallOP{1}}{n}\left[\left(\fr{N_1}{N_0}\right)^r\sum_{i:D_i=0}\left(\fr{f_1(\pi(X_i))}{f_0(\pi(X_i))}\right)^r\sigma_{D_i}^2(\pi(X_i)) \right. \nonumber \\
	\left. + \left(\fr{N_0}{N_1}\right)^r\sum_{i:D_i=1}\left(\fr{f_0(\pi(X_i))}{f_1(\pi(X_i))}\right)^r\sigma_{D_i}^2(\pi(X_i))\right] \nonumber \\
	 \convprob \Eb{(1-D)\left(\fr{p_1}{1-p_1}\fr{f_1(\pi(X))}{f_0(\pi(X))}\right)^r\sigma_{D}^2(\pi(X))} \label{eq:problim_exp_1} \\
	 +\Eb{D\left(\fr{1-p_1}{p_1}\fr{f_0(\pi(X))}{f_1(\pi(X))}\right)^r\sigma_D^2(\pi(X))} \label{eq:problim_exp_2},
\end{align}
by the weak law of large numbers as $(N_d/N_{1-d})^r\convas \left(\fr{p_d}{1-p_d}\right)^r$ and the $f_d/f_{1-d}$ are uniformly bounded by \Cref{ass:propscoredist} and the $\sigma_d^2$ are bounded by \Cref{ass:variance}.
For $r=1$, \eqref{eq:problim_exp_1} is 
\begin{align*}
\Ebc{\fr{p_1}{1-p_1}\fr{f_1(\pi(X))}{f_0(\pi(X))}\sigma_0^2(\pi(X))}{D=0}(1-p_1)&=\int_{\underline p}^{\ba p}p_1\fr{f_1(p)}{f_0(p)}\sigma_0^2(p)f_0(p)\deriv p \\
	&= \E \pi(X)\sigma_0^2(\pi(X)),
\end{align*}
where we used that $f_1(p)=\fr{p}{p_1}f_{\pi(X)}(p)$, with $f_{\pi(X)}$ being the density of $\pi(X)$. Noting that $f_0(p)=\fr{1-p}{1-p_1}f_{\pi(X)}(p)$, \eqref{eq:problim_exp_2} is $\E(1-\pi(X))\sigma_1^2(\pi(X))$. For $r=2$, the same arguments yield $\E \fr{\pi(X)^2}{1-\pi(X)}\sigma_0^2(\pi(X))$ for \eqref{eq:problim_exp_1}, and $\E \fr{(1-\pi(X))^2}{\pi(X)}\sigma_1^2(\pi(X))$ for \eqref{eq:problim_exp_2}. Note that $\E \sigma_D^2(\pi(X))=\E(1-\pi(X))\sigma_0^2(\pi(X))+\E \pi(X)\sigma_1^2(\pi(X))$.
Collect the terms to get the assertion, which also holds for $\delta_n=\caliperDelta$ in view of \Cref{lem:convergence_of_ratios}.
%(Put $\pi(X)$ for what is $X$ in the notation of \cite{hahn_role_1998} and note that $\prob{D=1\mid \pi(X)}=\Ebc{\Ebc{D}{\pi(X),X}}{\pi(X)}=\Ebc{\Ebc{D}{X}}{\pi(X)}=\Ebc{\pi(X)}{\pi(X)}=\pi(X)$, where the second inequality follows from $D\indep \pi(X)\mid X$, to see that the bound given above is the same as that of \cite{hahn_role_1998}.)
\end{proof}

\end{appendix}

%%%%%%%%%%%%%%%%%%%%%%%%%%%%%%%%%%%%%%%%%%%%%%%%%%%%%%%%%
%%%%%%%%%%%%%%%%%%%%%%%%%%%%%%%%%%%%%%%%%%%%%%%%%%%%%%%%%

%%%%%%%%%%%%%%%%%%%% 		PROOFS		%%%%%%%%%%%%%%%%%%%%%%

%%%%%%%%%%%%%%%%%%%%%%%%%%%%%%%%%%%%%%%%%%%%%%%%%%%%%%%%%
%%%%%%%%%%%%%%%%%%%%%%%%%%%%%%%%%%%%%%%%%%%%%%%%%%%%%%%%%

\newpage
\begin{appendix}
\section*{Supplement}\label{app:online} %% if no title is needed, leave empty \section*{}.
% refer to this as  \hyperref[appn]{Online Appendix}

%%%%%%%%%%%%%%%%%%%% 		Semiparametric efficiency		%%%%%%%%%%%%%%%%%%%%%%
\begin{proof}[Proof of \Cref{prop:semipara_eff}]
The semiparametric lower bound is the variance of the efficient influence function. When the observed sample is $((Y_i,D_i,X_i))_{i\in[n]}$, the efficient influence function of ATE is
\begin{align*}
\chi_{\Xsupp}(Y,D,X)\ceq \fr{D(Y-\mu_\Xsupp^1(X))}{\pi(X)}-\fr{(1-D)(Y-\mu_\Xsupp^0(X))}{1-\pi(X)}+\mu_\Xsupp^1(X)-\mu_\Xsupp^0(X)-\tau,
\end{align*}
see \cite{hahn_role_1998}. Then $V=V_\mathrm{eff}$ follows under \eqref{cond:semipara_eff:c2} in \Cref{prop:semipara_eff}. One can verify that $V=\Vb{\chi(Y,D,\pi(X))}$, where
\begin{align*}
\chi(Y,D,\pi(X))\ceq&\, \fr{D(Y-\mu^1(\pi(X)))}{\pi(X)}-\fr{(1-D)(Y-\mu^0(\pi(X)))}{1-\pi(X)} \\
&+\mu^1(\pi(X))-\mu^0(\pi(X))-\tau.
\end{align*}
\Cref{ass:ucf} and $\Ebc{D}{X}=\pi(X)$ imply that 
$$\E \chi_\Xsupp(Y,D,X)(\chi(Y,D,\pi(X))-\chi_\Xsupp(Y,D,X))=0,$$
thus 
\begin{align*}
V=\Vb{\chi(Y,D,\pi(X))} &=\Vb{\chi(Y,D,\pi(X))-\chi_\Xsupp(Y,D,X)}+\Vb{\chi_\Xsupp(Y,D,X)} \\
	&\geq \Vb{\chi_\Xsupp(Y,D,X)}=V_\mathrm{eff}.
\end{align*}

For ATT, $V_{\mathrm{t},\mathrm{eff}}\leq V_{\mathrm{t}}$ follows similarly as the efficient influence function of ATT under unknown propensity score \citep{hahn_role_1998} is
\begin{align*}
\chi_{\mathrm{t},\Xsupp}(Y,D,X)\ceq &\, \fr{D}{p_1}(Y-\mu_\Xsupp^1(X))-\fr{1-D}{p_1}\fr{\pi(X)}{1-\pi(X)}(Y-\mu_\Xsupp^0(X)) \\ 
&+\fr{D}{p_1}(\mu_\Xsupp^1(X)-\mu_\Xsupp^0(X)-\taut).
\end{align*}
 Under \Cref{ass:ucf}, one verifies that $V_\mathrm{t}=\Vb{\chi_{\mathrm{t}}(Y,D,\pi(X))}$, where
 \begin{align*}
\chi_{\mathrm{t}}(Y,D,\pi(X))\ceq &\, \fr{D}{p_1}(Y-\mu^1(\pi(X)))-\fr{1-D}{p_1}\fr{\pi(X)}{1-\pi(X)}(Y-\mu^0(\pi(X))) \\ 
&+\fr{D}{p_1}(\mu^1(\pi(X))-\mu^0(\pi(X))-\taut),
\end{align*}
and that 
$$\E \chi_{\mathrm{t},\Xsupp}(Y,D,X)(\chi_{\mathrm{t}}(Y,D,\pi(X))-\chi_{\mathrm{t}, \Xsupp}(Y,D,X))=0,$$ thus proving the assertion.

\end{proof}

%%%%%%%%%%%%%%%%%%%% 		Moment bounds un uniform spacings		%%%%%%%%%%%%%%%%%%%%%%
\begin{lemma}[Moment Bounds of Ordered Uniform Spacings]
\label{lem:spacingsmoments}
Let the order statistics of $(U_1,U_2,\ldots,U_n)\overset{\text{i.i.d.}}{\sim}\text{Uniform}(0,1)$ be $U_{(1)}\leq U_{(2)}\leq \ldots\leq U_{(n)}$. Let $\tilde U_1\ceq U_{(1)}$, $\tilde U_i\ceq U_{(i)}-U_{(i-1)}$ for $i=2,\ldots,n$ and $\tilde U_{n+1}\ceq 1-U_{(n)}$ be the spacings generated by $(U_i)_{i\in[n]}$. Let $\tilde U_{(1)}\leq \tilde U_{(2)}\leq\ldots\leq \tilde U_{(n+1)}$ be the ordered spacings. Then for any finite fixed integer $1\leq a < \fr{n+1}{2}$ and $n\geq2$, 
\begin{align*}
\E \tld U_{(r)}^a &= \begin{cases}
\bigO{\left(\fr{1}{n}\right)^{2a}} \quad&\text{ for } r=1 \\
\bigO{\left(\fr{\log r}{n}\right)^{a}} \quad&\text{ for all } r=2,3,\ldots,n+1.
\end{cases}
\end{align*}
In particular, $\E\tld U_{(r)}^a=\smallO{1}$ for all $r\in[n+1]$ and finite fixed integer $1\leq a< \fr{n+1}{2}$.
\end{lemma}
\begin{proof}
By \citet[Chapter 21]{shorack_empirical_2009}, $\tld U_{(r)}$ is distributed as $\fr{Z_{r:n+1}}{\sum_{i\in[n+1]}Z_i}$, for $(Z_1,Z_2,\ldots,Z_{n+1})\overset{\text{i.i.d.}}{\sim}\text{Exponential}(1)$ with order statistics $Z_{1:n+1}\leq Z_{2:n+1}\leq\ldots\leq Z_{n+1:n+1}$. The Cauchy--Schwarz inequality gives 
$$\E \tld U_{(r)}^a\leq\sqrt{\E Z_{r:n+1}^{2a}\E \left(\sum_{i\in[n+1]}Z_i\right)^{-2a}}.$$
Here, $\E Z_{1:n+1}^{2a}=\bigO{n^{-2a}}$ and $\E Z_{r:n+1}^{2a}=\bigO{\left(\log r\right)^{2a}}$ for $r\geq 2$ and for finite fixed integer $a\geq1$ by \Cref{lem:exponential_orderstats_moments}. The sum of $n+1$ i.i.d. $\text{Exponential}(1)$ variates follows a $\text{Gamma}(n+1,1)$ distribution. Thus, $(\sum_{i\in[n+1]}Z_i)^{-1}$ follows an $\text{Inverse-Gamma}(n+1,1)$ distribution, whose $2a$-th moment is equal to 
$\fr{(n-2a)!}{n!}\lesssim n^{-2a}$
for $2a<n+1$, where the last inequality follows from $\sqrt{2\pi}n^{n+1/2}e^{-n}\leq n!\leq e n^{n+1/2}e^{-n}$ \citep{robbins_remark_1955}, because $n-2a+1/2\geq 0$ for a positive integer $a$.
\end{proof}

%%%%%%%%%%%%%%%%%%%% 		Exponential order staistics		%%%%%%%%%%%%%%%%%%%%%%
\begin{lemma}[Moments of Exponential$(1)$ Order Statistics]
\label{lem:exponential_orderstats_moments}
Let $(Z_1,Z_2,\ldots,Z_{n})\overset{\text{i.i.d.}}{\sim}$ Exponential(1) with order statistics $Z_{1:n}\leq Z_{2:n+1}\leq\ldots\leq Z_{n:n}$. Then for finite fixed integer $k\geq 1$, we have for all $n\geq2$,
\begin{align}
\E Z_{r:n}^k=k!\sum_{t_1=1}^{r}\fr{1}{n+1-t_1}\sum_{t_2=1}^{t_1}\fr{1}{n+1-t_2}\cdots \sum_{t_{k-1}=1}^{t_{k-2}}\fr{1}{n+1-t_{k-1}}
 \sum_{t_k=1}^{t_{k-1}}\fr{1}{n+1-t_k} \label{eq:expordermoment}
\end{align}
for all $r=1,2,\ldots, n$. The right side in \eqref{eq:expordermoment} is $\bigO{(\log r)^k}$ for $r\geq 2$, and $\E Z_{1:n}^k=\bigO{n^{-k}}$.
\end{lemma}
\begin{proof}
The right side of \eqref{eq:expordermoment} and $\E Z_{1:n}^k=\bigO{n^{-k}}$ are obtained by solving the recursion in \citet[Theorems 1 and 2]{balakrishnan_2_1998}. The innermost sum in \eqref{eq:expordermoment} satisfies 
$\sum_{t_k=1}^{t_{k-1}}\fr{1}{n+1-t_k}\leq\sum_{j=1}^{r}\fr{1}{n+1-j}$ as $t_k\leq r$. Because every fraction in \eqref{eq:expordermoment} is positive, we can upper bound the right side of \eqref{eq:expordermoment} by 
$$k!\left(\sum_{j=1}^{r}\fr{1}{n+1-j}\right)\left(\sum_{t_1=1}^{r}\fr{1}{n+1-t_1}\sum_{t_2=1}^{t_1}\fr{1}{n+1-t_2}\cdots \sum_{t_{k-1}=1}^{t_{k-2}}\fr{1}{n+1-t_{k-1}}\right).$$ Apply the same bound for the remaining $k-1$ sums noting that $1\leq t_1\leq t_1\leq\ldots\leq t_{k-1}\leq r$, to obtain the bound $k!\left(\sum_{j=1}^{r}\fr{1}{n+1-j}\right)^k$ on  \eqref{eq:expordermoment}. The proof is complete as $\sum_{j=1}^{r}\fr{1}{n+1-j}$ is $\bigO{\log r}$ for $r\in[n]$ as $n\to\infty$.
\end{proof}

%%%%%%%%%%%%%%%%%%%% 		Conditional Martingale CLT		%%%%%%%%%%%%%%%%%%%%%%
 \begin{lemma}[Conditional Martingale Central Limit Theorem]
\label{lem:conditional_mclt}
Let $(\Omega_n,\mathcal{F}_n, \bbP_n)$ be a sequence of probability spaces. Let $\xi_{n1},\xi_{n2},\ldots,\xi_{nn}:\Omega_n\to \real$ be martingale differences with respect to sub-$\sigma$-algebras $\filt_{n1}\subset \filt_{n2}\subset\ldots\subset \filt_{nn}\subset\filt_n$. Let $\filt_{n0}\subset\filt_{n1}$ be a sub-$\sigma$-algebra. For $k=1,2,\ldots,n$, let $\sigma_{nk}^2\ceq \Ebcs{n}{\xi_{nk}^2}{\filt_{n,k-1}}$. If there exists a finite constant $\sigma>0$ such that
\begin{align}
&\probcs{n}{\abs{\sum_{k=1}^n \sigma_{nk}^2-\sigma^2}>\epsilon}{\filt_{n0}}\os{\bbP_n}{\longrightarrow}0\quad\text{for all constants $\epsilon>0$ and}\label{eq:mclt_var} \\
&\sum_{k=1}^n \Ebcs{n}{\xi_{nk}^2\indic{|\xi_{nk}|\geq\eta}}{\filt_{n0}}\os{\bbP_n}{\longrightarrow}0 \quad\text{for all constants $\eta>0$,}\label{eq:mclt_lindeberg}
\end{align}
then $\probcs{n}{\sigma^{-1}\sum_{k=1}^n \xi_{nk}\leq x}{\filt_{n0}}\os{\bbP_n}{\longrightarrow}\Phi(x)$ as $n\to\infty$ for all $x\in\real$, where $\Phi$ is the standard normal distribution function.
\end{lemma}
\begin{proof}
Follows from \citet[Theorem 35.12]{billingsley_probability_1995} by conditioning on $\filt_{n0}$ throughout.
 \end{proof}

\begin{proof}[Proof of \Cref{prop:consistent_variance}]
\emph{Existence of $A_n, \hat A_n$}.\, We show that $A_n ,\hat A_n$ are well-defined with probability tending to one. For $A_n$, this happens if and only if 
$$\prob{\uphatth+a_n<\bphatth-a_n}=\prob{2a_n<\bphatth-\uphatth}\to 1.$$ 
If $\uphatth=\underline p +\smallOP{1}$ and $\bphatth=\ba p +\smallOP{1}$, the probability is $\prob{2a_n<\ba p - \underline p + \smallOP{1}}=\prob{2< a_n^{-1}(\ba p - \underline p)+\smallOP{a_n^{-1}}}$, which goes to one as $\ba p>\underline p$ and $a_n\downarrow 0$. To show that these conditions hold, define $T_n(x)\ceq \hatth^\intercal x$ and $T(x)\ceq \theta_0^\intercal x$, so by definition $\ba p = \sup_{x\in\mathcal{X}}g(T(x))=g(\sup_{x\in\mathcal{X}} T(x))$ and $\bphatth = \sup_{x\in\mathcal{X}}g(T_n(x))=g(\sup_{x\in\mathcal{X}} T_n(x))$ because $g$ is increasing by \Cref{ass:singleindex_propscore_smooth_outcome}. Since $g$ is continuous by \Cref{ass:singleindex_propscore_smooth_outcome}, it suffices by the continuous mapping theorem to show $\sup_{x\in\mathcal{X}} T_n(x)\convprob \sup_{x\in\mathcal{X}} T(x)$. Because $\mathcal{X}$ is bounded and $\hatth\convprob \theta_0$ by \Cref{ass:theta_estimator}, $\sup_{x\in\mathcal{X}}|T_n(x)-T(x)|\lesssim \normu{\hatth-\theta_0}=\smallOP{1}$, so that $\bphatth=\ba p+\smallOP{1}$. Similar arguments yield $\uphatth=\underline p +\smallOP{1}$. 

For $\hat A_n$, the desired result follows from that for $A_n$ above, and that $\min_{i\in[n]}g(\hatth^\intercal X_i)=\uphatth+\smallOP{1}$ and $\max_{i\in[n]}g(\hatth^\intercal X_i)=\bphatth+\smallOP{1}$. Because $F_{\hatth}^{-1}$, the inverse of $F_{\hatth}(p)=p_1F_{1,\hatth}(p)+(1-p_1)F_{0,\hatth}(p)$ which is the distribution of $(\pi(X,\hatth)\mid \hatth)$ under Assumption \ref{ass:theta_estimator}, is strictly increasing by Assumption \ref{ass:propscoredist_estimated}, $(\min_{i\in[n]} \pi(X_i,\hatth)\mid \hatth)$ is distributed as $F_{\hatth}^{-1}(U_{(1)})$, where $U_{(1)}$ is the sample minimum of $((U_1,U_2,\ldots, U_n)\mid \hatth)\overset{\text{i.i.d.}}{\sim}\text{Uniform}(0,1)$. Then $(\min_{i\in[n]} \pi(X_i,\hatth)- \uphatth\mid\hatth)$ is distributed as $F_{\hatth}^{-1}(U_{(1)})-F_{\hatth}^{-1}(F_{\hatth}(\uphatth))=F_{\hatth}^{-1}(U_{(1)})-F_{\hatth}^{-1}(0)$ given $\hatth.$ By Assumption \ref{ass:propscoredist_estimated}, $F_{\hatth}^{-1}$ is Lipschitz with constant $\supnorm{(F_{\hatth}^{-1})'}$ with $(F_{\hatth}^{-1})'(u)=\fr{1}{f_{\hatth}(F_{\hatth}^{-1}(u))}$ finite for $\inf_{p\in \uphatth,\bphatth}f_{\hatth}(p)>0$ by Assumption \ref{ass:propscoredist_estimated} for $\hatth\in\ntheta$. Thus, $F_{\hatth}^{-1}(U_{(1)})-F_{\hatth}^{-1}(0)\lesssim U_{(1)}$. Here, $\E U_{(1)}$ goes to zero by the proof of \Cref{prop:maxmindistorder}. Hence, by \Cref{ass:singleindex_propscore_smooth_outcome}, $\min_{i\in[n]}g(\hatth^\intercal X_i)=\uphatth+\smallOP{1}$ and $\max_{i\in[n]}g(\hatth^\intercal X_i)=\bphatth+\smallOP{1}$ similarly. In the following, we prove the consistency of the variance component estimators for $V_{\hat\pi}$; similar arguments give the result for $V_{\mathrm{t},\hat\pi}$. We show below that $\ntruncate/n\convprob1$ (since \eqref{eq:indicator_order} is $\smallOP{1}$). Therefore, in the following, we prove the consistency of the estimators $\hat V_\tau,\hat V_{\taut}, \hat V_{\sigma,\pi},\hat V_{\mathrm{t},\sigma,\pi}, \hat q_d$ and $\hat q_{\mathrm{t},d}$ normalised by $n$ rather than $\ntruncate$.

%%% V_\tau
\emph{Consistency of  $\hat V_\tau$.} By \Cref{thm:asymnorm_estimated_propscore_ate} and the continuous mapping theorem, $(\thatpihat)^2\convprob \tau^2$.  For short, put $\pi_i\ceq g(\theta_0^\intercal X_i)$, $\hat\pi_i\ceq g(\hatth^\intercal X_i)$ and $\indichat{i}\ceq \indic{g(\hatth^\intercal X_i)\in\hat A_n}$. The first term in \eqref{eq:hatVtau} is
\begin{align}
\nsumn{i}[\mu^1(\hatth,\hat\pi_i)-\mu^0(\hatth,\hat\pi_i)]^2\indichat{i}  \label{eq:hatVtau3} \\
+\nsumn{i} [\hat\mu^1(\hatth,\hat\pi_i)-\hat\mu^0(\hatth,\hat\pi_i)-(\mu^1(\hatth,\hat\pi_i)-\mu^0(\hatth,\hat\pi_i))]^2\indichat{i} \label{eq:hatVtau1}\\
+ \fr{2}{n}\sum_{i\in[n]}[\hat\mu^1(\hatth,\hat\pi_i)-\hat\mu^0(\hatth,\hat\pi_i)-(\mu^1(\hatth,\hat\pi_i)-\mu^0(\hatth,\hat\pi_i))][\mu^1(\hatth,\hat\pi_i)-\mu^0(\hatth,\hat\pi_i)]\indichat{i}. \label{eq:hatVtau2}
\end{align}
Here, \eqref{eq:hatVtau3} converges to $\E(\mu^1(\theta_0,\pi_i)-\mu^0(\theta_0,\pi_i))^2$. To see this, first note that under \Cref{ass:theta_estimator},
\begin{align*}
\abs{\nsumn{i}[\mu^1(\hatth,\hat\pi_i)-\mu^0(\hatth,\hat\pi_i)]^2\indichat{i}-\nsumn{i}[\mu^1(\theta_0,\pi_i)-\mu^0(\theta_0,\pi_i)]^2\indichat{i}}\convprob 0,
\end{align*}
which follows from a mean-value expansion of $[\mu^1(\hatth,\hat\pi_i)-\mu^0(\hatth,\hat\pi_i)]^2$ in $\hatth$, similarly to the treatment of \eqref{eq:tauhat_thetahat_decomp1_2} in the proof of  \Cref{thm:asymnorm_estimated_propscore_ate}. Second, as the $\mu^d(\theta,\cdot)$, $\theta\in\ntheta$, are bounded by \Cref{ass:lip_regression}, 
\begin{align}
\abs{\nsumn{i}[\mu^1(\theta_0,\pi_i)-\mu^0(\theta_0,\pi_i)]^2\indichat{i}-\nsumn{i}[\mu^1(\theta_0,\pi_i)-\mu^0(\theta_0,\pi_i)]^2} \label{eq:mu_indichat}
\end{align}
is of the order
\begin{align}
\nsumn{i}(1-\indichat{i})=&\,\nsumn{i}\left(\indic{g(\hatth^\intercal X_i)\notin\hat A_n}-\Ebc{\indic{g(\hatth^\intercal X_i)\notin\hat A_n}}{\hatth}\right) \nonumber \\
&+\nsumn{i}\Ebc{\indic{g(\hatth^\intercal X_i)\notin\hat A_n}}{\hatth} \label{eq:indicator_order}.
\end{align}
with probability tending to one. Here the first term has mean zero and variance bounded by $1/n$, so it converges to zero in the first mean, and then so in probability. By \Cref{ass:theta_estimator}, the second term of \eqref{eq:indicator_order} is 
\begin{align}
\probc{g(\hatth^\intercal X_i)\notin \hat A_n}{\hatth}=&\probc{g(\hatth^\intercal X_i)< \min_{i\in[n]}g(\hatth^\intercal X_i)+a_n}{\hatth} \label{eq:indicator_order_1}\\
&+\probc{\max_{i\in[n]}g(\hatth^\intercal X_i))-a_n<g(\hatth^\intercal X_i)}{\hatth} \label{eq:indicator_order_2}
\end{align}
for some $i\in[n]$. Let $\underline G_{\hatth}\ceq \min_{i\in[n]}g(\hatth^\intercal X_i)$. To bound \eqref{eq:indicator_order_1}, we have, by \cite{shanmugam_characterizations_1988},
\begin{align}
\probc{g(\hatth^\intercal X_i)\leq \underline G_{\hatth}+a_n}{\hatth, \underline G_{\hatth}} = \fr{1}{n}+\fr{n-1}{n}\fr{F_{\hatth}(\underline G_{\hatth}+a_n)-F_{\hatth}(\underline G_{\hatth})}{1-F_{\hatth}(\underline G_{\hatth})} \label{eq:indicator_order_1_bound}
\end{align}
under \Cref{ass:theta_estimator,ass:singleindex_propscore_smooth_outcome}, where $F_{\hatth}\ceq p_1F_{\hatth,1}+(1-p_1)F_{\hatth,0}$ is the distribution function of $g(\hatth^\intercal X)$ given $\hatth$. By \Cref{ass:propscoredist_estimated}, $F_{\hatth}$ is continuous, and by arguments on $\hat A_n$ above $\underline G_{\hatth}=\uphatth+\smallOP{1}$. Then $a_n\downarrow 0$ implies that \eqref{eq:indicator_order_1_bound} is $\smallOP{1}$, which in turn implies that \eqref{eq:indicator_order_1}, being bounded by one, is also $\smallOP{1}$. Term \eqref{eq:indicator_order_2} is $\smallOP{1}$ by similar arguments, noting that $\set{\max_{i\in[n]}g(\hatth^\intercal X_i))-a_n<g(\hatth^\intercal X_i)}=\set{\min_{i\in[n]}-g(\hatth^\intercal X_i))+a_n> -g(\hatth^\intercal X_i)}$. Thus, \eqref{eq:indicator_order} is $\smallOP{1}$. Conclude that \eqref{eq:hatVtau3} converges in probability to $V_\tau+\tau^2$. Write \eqref{eq:hatVtau1} as
\begin{align*}
\nsumn{i} [\hat\mu^1(\hatth,\hat\pi_i)-\mu^1(\hatth,\hat\pi_i)]^2\indichat{i} 
+\fr{2}{n}\sum_{i\in[n]}(\hat\mu^1(\hatth,\hat\pi_i)-\mu^1(\hatth,\hat\pi_i))(\hat\mu^0(\hatth,\hat\pi_i)-\mu^0(\hatth,\hat\pi_i))\indichat{i} \\
+\nsumn{i} [\hat\mu^0(\hatth,\hat\pi_i)-\mu^0(\hatth,\hat\pi_i)]^2\indichat{i}.
\end{align*}
As \Cref{ass:covar_outcome_distribution}\ref{ass:covar_outcome_distribution:outcome} bounds both $\hat\mu^d$ and $\mu^d$, $\sup_{p\in A_n}|\hat\mu^d(\hatth,p)-\mu^d(\hatth,p)|\convprob0$ implies that all three terms in the last display are $\smallOP{1}$ (note that $\hat A_n\subset A_n$). This convergence can be established along the same lines as that of other estimators, which are detailed below. Similar arguments show that \eqref{eq:hatVtau2} is $\smallOP{1}$. Conclude that $\hat V_\tau\convprob V_\tau$.

%%% V_{\sigma,\pi}
\emph{Consistency of $\hat V_{\sigma,\pi}$.} First, $0<g(\hatth^\intercal x)<1$ for all $x$ in compact $\mathcal{X}$ and for all $\hatth\in\ntheta$. Then under \Cref{ass:singleindex_propscore_smooth_outcome}, a mean-value expansion and $\hatth\convprob\theta_0$, implied by \Cref{ass:theta_estimator}, gives that $\sup_{x\in\mathcal{X}}|1/g(\hatth^\intercal x)-1/g(\theta_0^\intercal x)|=\smallOP{1}$ and similarly for $1/(1-g(\hatth^\intercal x))$. Second, the Lipschitz condition in \Cref{ass:lip_conditional_var}\ref{ass:lip_conditional_var:lip} implies 
\begin{align*}
\abs{\nsumn{i}\sigma_d^2(\hatth,\hat\pi_i)\indichat{i}-\nsumn{i}\sigma_d^2(\theta_0,\pi_i)\indichat{i}}\convprob 0,
\end{align*}
and, as in \eqref{eq:mu_indichat} above, we also have
\begin{align*}
\abs{\nsumn{i}\sigma_d^2(\theta_0,\pi_i)\indichat{i}-\nsumn{i}\sigma_d^2(\theta_0,\pi_i)}\convprob 0.
\end{align*}
Thus, if $\sup_{p\in A_n}|\hat\sigma_d^2(\hatth,p)-\sigma_d^2(\hatth,p)|\convprob0$, then the law of large numbers gives that $\hat V_{\sigma,\pi}\convprob V_{\sigma,\pi}$. This holds if both  
$\sup_{p\in A_n}|\hat\mu^d(\hatth,p)-\mu^d(\hatth,p)|\convprob0$ and $\sup_{p\in A_n}|\hat\mu_2^d(\hatth,p)-\mu_2^d(\hatth,p)|\convprob0$ since the former implies $\sup_{p\in A_n}|(\hat\mu^d(\hatth,p))^2-(\mu^d(\hatth,p))^2|\convprob0$ under \Cref{ass:covar_outcome_distribution}\ref{ass:covar_outcome_distribution:outcome}. See below for a proof of such convergence (e.g. establishing \eqref{eq:muhat}).

%%%  q
\emph{Consistency of $\hat q_d$.} By definition of $q_d$, the conditions
\begin{align}
\normu{\nsumn{i} \Lambda^d(\hatth, X_i)- \nsumn{i} \Lambda^d(\hatth, X_i)\indichat{i}} \convprob 0 \label{eq:qhatindicator} \\
\sup_{p\in A_n}\abs{\widehat{\left(\fr{\partial \mu^d}{\partial p}\right)}(\hatth,p) - \fr{\partial \mu^d}{\partial p}(\hatth,p)
}\convprob 0 \label{eq:derivhat_p} \\
\sup_{p\in A_n}\abs{\widehat{\left(\fr{\partial \mu^d}{\partial \theta_k}\right)}(\hatth,p) - \fr{\partial \mu^d}{\partial \theta_k}(\hatth,p)
}\convprob 0\quad \text{for all }k=1,2,\ldots,K, \label{eq:derivhat_theta}
\end{align}
together with \Cref{ass:covar_outcome_distribution,ass:singleindex_propscore_smooth_outcome}, implying  the boundedness of $\mathcal{X}$ and $g'$, ensure $\hat q_d\convprob q_d$. \Cref{ass:differentiability_regression}, \ref{ass:singleindex_propscore_smooth_outcome}\ref{ass:singleindex_propscore_smooth_outcome:propscore}, and \Cref{ass:covar_outcome_distribution} imply that $\Lambda^d$ is bounded uniformly. Therefore, \eqref{eq:qhatindicator} is satisfied, because the arguments treating $\indichat{i}$ in the case of $\hat V_\tau$ above (e.g. \eqref{eq:mu_indichat}) apply.

In the following, we show that \eqref{eq:derivhat_p} and \eqref{eq:derivhat_theta} hold, wherein we also give a detailed proof of the uniform consistency of the nonparametric estimators assumed above. For simplicity of exposition, we only give the proof for the (derivatives of) $\mu(\theta,p)\ceq \Ebc{Y}{g(\theta^\intercal X)=p}=\Ebc{Y}{\theta^\intercal X=\ginv(p)}$ and $\mu_2(\theta,p)\ceq \Ebc{Y^2}{g(\theta^\intercal X)=p}$. The proof when we also condition on $D=d$ follows along the same lines. To this end, let $h(\theta,p)\ceq p_1 h_{1}(\theta, p)+(1-p_1)h_0(\theta,p)$ be the density of $\pi(X,\theta)$, $q_\mu(\theta,p)\ceq \mu(\theta,p)h(\theta,p)$, $q_{\mu_2}(\theta,p)\ceq \mu_2(\theta,p)h(\theta,p)$ and their corresponding estimators be obtained by setting $\indic{D_j=d}\ceq 1$ for all $j\in[n]$ in the formulae for $\hat h_d$, $\hat q_{\mu,d}$, $\hat q_{\mu_2,d}$, respectively. For short, we also let $h'(\theta,p)\ceq (\partial/\partial p)h(\theta,p)$, $\hat h'(\theta,p)\ceq (\partial/\partial p)\hat h(\theta,p)$ and likewise for $q_\mu', \hat q_\mu'$. Furthermore, we let $\widehat{\left(\fr{\partial h}{\partial \theta_k}\right)}$ and $\widehat{\left(\fr{\partial q_\mu}{\partial \theta_k}\right)}$ be obtained by setting $\indic{D_j=d}\ceq 1$ in the formula for $\widehat{\left(\fr{\partial h_d}{\partial \theta_k}\right)}$ and $\widehat{\left(\fr{\partial q_{\mu,d}}{\partial \theta_k}\right)}$, respectively. For short, we put  $\hdotk(\theta,p)\ceq  (\partial/\partial \theta_k)h(\theta,p)$, $\hdotkhat(\theta,p)\ceq \widehat{\left(\fr{\partial h}{\partial \theta_k}\right)}(\theta,p)$ and likewise for $\qdotk(\theta,p), \qdotkhat(\theta,p)$. For a function $r:\Theta\times [0,1]\to\real$, we let $\norm{A_n}{r}\ceq\sup_{p\in A_n}|r(\hatth,p)|$. 

%% Derivative w.r.t p
\emph{Condition \eqref{eq:derivhat_p}.} First we show that the numerator of
\begin{align}
\widehat{\left(\fr{\partial \mu}{\partial p}\right)}(\hatth,p)\ceq \fr{\hat q_\mu'(\hatth,p) \hat h(\hatth,p)-\hat q_\mu(\hatth,p)\hat h'(\hatth,p)}{(\hat h(\hatth,p))^2} \label{eq:derivhat_nod_p}
\end{align}
converges to that of 
\begin{align}
\fr{\partial \mu}{\partial p}(\hatth,p)=\fr{q_\mu'(\hatth,p) h(\hatth,p)- q_\mu(\hatth,p)h'(\hatth,p)}{h(\hatth,p)^2}. \label{eq:deriv_nod_p}
\end{align}
Consider
\begin{align*}
\norm{A_n}{\hat q_\mu' \hat h - \hat q_\mu\hat h'-(q_\mu'h-q_\mu h')} \leq &\, \norm{A_n}{\hat q_\mu' \hat h - q_\mu' h} + \norm{A_n}{\hat q_\mu \hat h' - q_\mu h'} \\
\norm{A_n}{\hat q_\mu' \hat h - q_\mu' h}  = &\, \norm{A_n}{(\hat q_\mu' - q_\mu' + q_\mu')(\hat h - h + h) - q_\mu' h} \\
	\leq&\, \norm{A_n}{\hat q'_\mu - q_\mu'}\norm{A_n}{\hat h-h}+\norm{A_n}{\hat q_\mu' - q_\mu'}\norm{A_n}{h} \\
	&+ \norm{A_n}{q'_\mu}\norm{A_n}{\hat h - h} \\
\norm{A_n}{\hat q_\mu \hat h' - q_\mu h'} \leq &\, \norm{A_n}{\hat h' - h'}\norm{A_n}{\hat q_\mu - q_\mu}+\norm{A_n}{\hat h' - h'}\norm{A_n}{q_\mu} \\
	&+ \norm{A_n}{h'}\norm{A_n}{\hat q_\mu - q_\mu}.
\end{align*}
By \Cref{ass:propscoredist_estimated}\ref{ass:propscoredist_estimated:differentiability_p_and_theta}, $\norm{A_n}{h}$ is bounded with probability tending to one as $\hatth\convprob \theta_0$ and combining it with \Cref{ass:differentiability_regression,ass:singleindex_propscore_smooth_outcome}\ref{ass:singleindex_propscore_smooth_outcome:regression}, the same holds for  $\norm{A_n}{h'}$, $\norm{A_n}{q_\mu}=\norm{A_n}{\mu h}$ and $\norm{A_n}{q_\mu'}=\norm{A_n}{\mu' h + \mu h'}$. It follows that if all
\begin{align}
\norm{A_n}{\hat q'_\mu - q_\mu'}&\convprob 0, \label{eq:hatq_prime} \\
\norm{A_n}{\hat h' - h'}&\convprob 0,  \label{eq:hath_prime}\\ 
\norm{A_n}{\hat q_\mu - q_\mu}&\convprob 0, \label{eq:hatq} \\
\norm{A_n}{\hat h - h}&\convprob 0, \label{eq:hath} 
\end{align}
then the numerator of \eqref{eq:derivhat_nod_p} converges to that of \eqref{eq:deriv_nod_p} uniformly in $p\in A_n$. The denominator of \eqref{eq:derivhat_nod_p} satisfies
\begin{align*}
\norm{A_n}{(\hat h)^2-h^2} &= \norm{A_n}{(\hat h-h)(\hat h + h + h - h)} \leq \norm{A_n}{\hat h - h}\left\{\norm{A_n}{\hat h- h}+2\norm{A_n}{h}\right\}.
\end{align*}
By \Cref{ass:propscoredist_estimated}\ref{ass:propscoredist_estimated:differentiability_p_and_theta}, $\norm{A_n}{h}<\infty$ with probability tending to one as $\hatth\convprob\theta_0$, hence \eqref{eq:hath} implies that the denominator of \eqref{eq:derivhat_nod_p} converges to that of \eqref{eq:deriv_nod_p} uniformly in $p\in A_n$.
Thus, both the numerator and the denominator of \eqref{eq:derivhat_nod_p} converges to those of \eqref{eq:deriv_nod_p}. Because $\inf_{p\in[\uphatth,\bphatth]}h(\hatth,p)>0$ for all $\hatth\in\ntheta$, it follows that \eqref{eq:derivhat_nod_p} converges to \eqref{eq:deriv_nod_p} uniformly in $p\in A_n$. In the following, we show that \eqref{eq:hatq_prime}--\eqref{eq:hath} hold.

\emph{Condition \eqref{eq:derivhat_p}, part \eqref{eq:hatq_prime}.}  The proof consists in showing
\begin{align}
\E \Ebc{\sup_{p\in A_n}\abs{\hat q_\mu'(\hatth,p)-\Ebc{\hat q_\mu'(\hatth,p)}{\hatth}}}{\hatth} &\to 0 \quad\text{ and } \label{eq:hatq_prime_1} \\
\sup_{p\in A_n}\abs{\Ebc{\hat q_\mu'(\hatth,p)}{\hatth}-q_\mu'(\hatth,p)} &\convprob 0. \label{eq:hatq_prime_2}
\end{align}
We show \eqref{eq:hatq_prime_1} following \cite{bierens_topics_1994}. For the imaginary unit $i$, let $\psi(t)\ceq \int e^{itx}K(x)\deriv x$, $t\in\real$, be the characteristic function of $K$, which is $\psi(t)=e^{-t^2/2}$ for the Gaussian kernel $K$, so that $K(x)=(2\pi)^{-1}\int  e^{-itx}\psi(t)\deriv t$ by the inversion formula for characteristic functions. Then $K'(x)=(2\pi)^{-1}\int  (-it)e^{-itx}\psi(t)\deriv t$, hence
\begin{align*}
\hat q_\mu'(\hatth,p) &= -\fr{1}{n\gamma_n^2}\sum_{j\in[n]} Y_j (2\pi)^{-1} \int (-it) e^{-it(g(\hatth^\intercal X_j)-p)/\gamma_n}\psi(t)\deriv t \\
		&= (2\pi)^{-1}\int \left(\nsumn{j} Y_j e^{-itg(\hatth^\intercal X_j)} \right) e^{itp} it\psi(\gamma_n t)\deriv t,
\end{align*}
where we used a change of variables and Fubini's theorem (the integral is bounded as $|Y_j|<\ba y$ almost surely by \Cref{ass:covar_outcome_distribution}\ref{ass:covar_outcome_distribution:outcome}, $|it|\leq 1$, $\sup_{t\in\real}|\psi(t)|<\infty$ and the exponential is bounded too). As $|e^{itp}|\leq1$ and $|it|\leq |t|$, we have that
\begin{align*}
\Ebc{\sup_{p\in A_n}\abs{\hat q_\mu'(\hatth,p)-\Ebc{\hat q_\mu'(\hatth,p)}{\hatth}}}{\hatth}\\
 \leq (2\pi)^{-1} \int \Ebc{\abs{\left(\nsumn{j} Y_j e^{-itg(\hatth^\intercal X_j)} \right)-\Ebc{Y_j e^{-itg(\hatth^\intercal X_j)}}{\hatth}}}{\hatth} |t| |\psi(\gamma_n t)| \deriv t.
\end{align*}
As $e^{-ia}=\cos(a)-i\sin(a)$ for $a\in\real$, and $\E |W|\leq \sqrt{\E W^2}$ for any square-integrable random variable $W$, the expectation on the right of the last display is bounded by 
\begin{align*}
\Vbc{\nsumn{j} Y_j\cos(g(\hatth^\intercal X_j))}{\hatth}^{1/2}+\Vbc{\nsumn{j} Y_j\sin(g(\hatth^\intercal X_j))}{\hatth}^{1/2} \leq 2 \sqrt{\fr{\E Y_j^2}{n}},
\end{align*}
where we used that by \Cref{ass:theta_estimator}\ref{ass:theta_estimator:indep} the elements in the sum are i.i.d. given $\hatth$, so the covariances are zero, and that $\Vbc{Y_j\cos(g(\hatth^\intercal X_j))}{\hatth}\leq \Ebc{Y_j^2}{\hatth}=\E Y_j^2$. Thus,
\begin{align*}
\Ebc{\sup_{p\in A_n}\abs{\hat q_\mu'(\hatth,p)-\Ebc{\hat q_\mu'(\hatth,p)}{\hatth}}}{\hatth}
 &\leq \sqrt{\fr{\E Y_j^2}{\pi^2 n}}\int |t||\psi(\gamma_n t)|\deriv t \\
 &\leq \sqrt{\fr{\E Y_j^2}{\pi^2 n \gamma_n^4}}\int |t||\psi(t)|\deriv t. 
\end{align*}
As $\E Y_j^2<\infty$ by \Cref{ass:covar_outcome_distribution}\ref{ass:covar_outcome_distribution:outcome} and $\int |t||\psi(t)|\deriv t=\int |t|e^{-t^2/2}<\infty$ for the Gaussian $K$, the right side is of the order $1/(\gamma_n^2\sqrt{n})=(\kappa_0)^{-2}n^{2\beta-1/2}=\smallO{1}$ for $\beta<1/4$.

Next, we show \eqref{eq:hatq_prime_2}. As the summands are identically distributed given $\hatth$ by  \Cref{ass:theta_estimator}\ref{ass:theta_estimator:indep}, the tower property of expectations gives
\begin{align*}
\Ebc{\hat q_\mu'(\hatth,p)}{\hatth} & = -\fr{1}{\gamma_n^2}\Ebc{\Ebc{Y}{g(\hatth^\intercal X),\hatth}K'((g(\hatth X)-p)/\gamma_n)}{\hatth} \\
	&= -\fr{1}{\gamma_n^2}\Ebc{\mu(\hatth, g(\hatth^\intercal X))K'((g(\hatth X)-p)/\gamma_n)}{\hatth} \\
	&= -\fr{1}{\gamma_n^2} \int_{\uphatth}^{\bphatth} \mu(\hatth,\tld p)K'((\tld p-p)/\gamma_n)h(\hatth,\tld p)\deriv\tld p \\
	&= -\fr{1}{\gamma_n^2}\int_{\uphatth}^{\bphatth} q_\mu(\hatth,\tld p)K'((\tld p-p)/\gamma_n)\deriv\tld p \\
	&= -\fr{1}{\gamma_n} \int_{(\uphatth-p)/\gamma_n}^{(\bphatth-p)/\gamma_n} q_\mu(\hatth,\gamma_nv+p)K'(v)\deriv v 
\end{align*}
by definition of $h(\hatth,\cdot)$ as the density of $(g(\hatth^\intercal X)\mid \hatth)$ under Assumptions \ref{ass:theta_estimator}\ref{ass:theta_estimator:indep} and \ref{ass:singleindex_propscore_smooth_outcome}\ref{ass:singleindex_propscore_smooth_outcome:propscore}, and $q_\mu=\mu h$. Integration by parts gives
\begin{align}
\Ebc{\hat q_\mu'(\hatth,p)}{\hatth}  =&\, -\fr{1}{\gamma_n}\left\{q(\hatth, \bphatth)K((\bphatth-p)/\gamma_n)- q(\hatth, \uphatth)K((\uphatth-p)/\gamma_n)\right\} \label{eq:expected_qhat_1} \\
&+ \int_{(\uphatth-p)/\gamma_n}^{(\bphatth-p)/\gamma_n} q_\mu'(\hatth,\gamma_nv+p)K(v) \deriv v.  \label{eq:expected_qhat_2}
\end{align}
As $q_\mu(\hatth,\cdot)$ is bounded for $\hatth\in\ntheta$, with probability tending to one \eqref{eq:expected_qhat_1} is of the order 
$$\gamma_n^{-1}\left[\sup_{p\in A_n}K((\bphatth-p)/\gamma_n)+\sup_{p\in A_n}K((\uphatth-p)/\gamma_n)\right]=2\gamma_n^{-1}K(a_n/\gamma_n).$$
As $\gamma_n^{-1}K(a_n/\gamma_n)\to 0$, \eqref{eq:expected_qhat_1} is $\smallOP{1}$. Now we show that \eqref{eq:expected_qhat_2} converges uniformly to $q_\mu'(\hatth,p)$ adapting the proof of  \citet[Lemma 1]{schuster_contributions_1979}. The kernel $K$ being a density integrating to one implies
\begin{align*}
q_\mu'(\hatth,p) & =q_\mu'(\hatth,p)\gamma_n^{-1} \left\{ \int_{-\infty}^{p-\bphatth}  K(u/\gamma_n)\deriv u +  \int_{p-\bphatth}^{p-\uphatth}  K(u/\gamma_n)\deriv u + \int_{p-\uphatth}^\infty K(u/\gamma_n)\deriv u \right\}.
\end{align*}
Combine this with a change of variables in \eqref{eq:expected_qhat_2} (with $u\ceq -\gamma_n v$ noting that $K$ is symmetric about zero), to get 
\begin{align}
\sup_{p\in A_n} \abs{\int_{(\uphatth-p)/\gamma_n}^{(\bphatth-p)/\gamma_n} q_\mu'(\hatth,\gamma_nv+p)K(v) \deriv v - q_\mu'(\hatth,p)} \nonumber \\
\leq \sup_{p\in A_n} \abs{\int_{-\infty}^{p-\bphatth}q_\mu'(\hatth,p)\gamma_n^{-1}K(u/\gamma_n)\deriv u} 
+\sup_{p\in A_n} \abs{\int_{p-\uphatth}^\infty q_\mu'(\hatth,p)\gamma_n^{-1}K(u/\gamma_n)\deriv u} \label{eq:expected_qhat_1_1} \\
+\sup_{p\in A_n} \abs{\int_{p-\bphatth}^{p-\uphatth} [q_\mu'(\hatth,p-u)-q_\mu'(\hatth,p)]\gamma_n^{-1}K(u/\gamma_n)\deriv u} \label{eq:expected_qhat_1_2}.
\end{align}
The two terms in \eqref{eq:expected_qhat_1_1} are $\smallOP{1}$. The first one is bounded by 
$$\sup_{p\in A_n}|q_\mu'(\hatth,p)|\sup_{p\in A_n}\int_{-\infty}^{(p-\bphatth)/\gamma_n}K(v)\deriv v=\sup_{p\in A_n}|q_\mu'(\hatth,p)|\sup_{p\in A_n}\int_{-\infty}^{-a_n/\gamma_n}K(v)\deriv v,$$
which vanishes as, on one hand, $\sup_{p\in A_n}|q_\mu'(\hatth,p)|<\infty$ with probability tending to one as $\hatth\convprob \theta_0$ by Assumptions \ref{ass:propscoredist_estimated}\ref{ass:propscoredist_estimated:differentiability_p_and_theta} and \ref{ass:differentiability_regression}, and, on the other hand, $a_n/\gamma_n\to \infty$ implies that the integral of the Gaussian $K$ goes to zero. The second term in \eqref{eq:expected_qhat_1_1} is bounded by
$$\sup_{p\in A_n}|q_\mu'(\hatth,p)|\sup_{p\in A_n}\int_{(p-\uphatth)/\gamma_n}^\infty K(v)\deriv v=\sup_{p\in A_n}|q_\mu'(\hatth,p)|\sup_{p\in A_n}\int_{a_n/\gamma_n}^\infty K(v)\deriv v,$$
so it is also $\smallOP{1}$ by the same argument. To show that \eqref{eq:expected_qhat_1_2} also vanishes, let  $\rho_n>0$ be a sequence satisfying $\rho_n/\gamma_n\to\infty$ and $\rho_n< a_n$, i.e. $\gamma_n\ll \rho_n < a_n$ (as $\gamma_n=\kappa_0n^{-\beta}$ and $a_n=\kappa_1n^{-\alpha}$, $\beta>\alpha$, we can take $\rho_n=\kappa_2 n^{-(\beta+\alpha)/2}$, $0<\kappa_2<\kappa_1$). Then \eqref{eq:expected_qhat_1_2} is bounded by
\begin{align}
\sup_{p\in A_n}\abs{\int_{[p-\bphatth, p-\uphatth]\cap \set{|u|\leq\rho_n}} [q_\mu'(\hatth,p-u)-q_\mu'(\hatth,p)]\gamma_n^{-1}K(u/\gamma_n)\deriv u} \label{eq:expected_qhat_1_2_1} \\
+ \sup_{p\in A_n}\abs{\int_{[p-\bphatth, p-\uphatth]\cap \set{|u|>\rho_n}} [q_\mu'(\hatth,p-u)-q_\mu'(\hatth,p)]u^{-1} u\gamma_n^{-1}K(u/\gamma_n)\deriv u}.\label{eq:expected_qhat_1_2_2} 
\end{align}
Here, \eqref{eq:expected_qhat_1_2_1} is bounded by 
\begin{align*}
\sup_{p\in A_n}\sup_{|u|\leq \rho_n}|q_\mu'(\hatth,p-u)-q_\mu'(\hatth,p)|\sup_{p\in A_n}\int_{-\infty}^{\infty} \gamma_n^{-1}K(u/\gamma_n) \deriv u \\
\leq \sup_{p\in A_n}\sup_{|u|\leq \rho_n}|q_\mu'(\hatth,p-u)-q_\mu'(\hatth,p)|.
\end{align*}
By Assumptions \ref{ass:propscoredist_estimated}\ref{ass:propscoredist_estimated:differentiability_p_and_theta} and \ref{ass:differentiability_regression}, $q_\mu'(\hatth,\cdot)$ is continuous on the compact set $[\uphatth,\bphatth]\supset A_n$ and is therefore uniformly continuous. Thus, $|u|\leq\rho_n$ for a small enough $\rho_n<a_n$ implies that $p-u,p\in [\uphatth,\bphatth]$ for all $p\in A_n$. Hence, $\sup_{p\in A_n}\sup_{|u|\leq \rho_n}|q_\mu'(\hatth,p-u)-q_\mu'(\hatth,p)|=\smallOP{1}$. In the integral of \eqref{eq:expected_qhat_1_2_2}, $u\in[p-\bphatth, p-\uphatth]$, so $p-u\in [\uphatth,\bphatth]$, and thus \eqref{eq:expected_qhat_1_2_2} is bounded by $2\sup_{p\in[\uphatth,\bphatth]}|q_\mu'(\hatth, p)|$ times 
\begin{align*}
\sup_{p\in A_n}\abs{\int_{[p-\bphatth, p-\uphatth]\cap \set{|u|>\rho_n}}\fr{K(u/\gamma_n)}{\gamma_n}\deriv u}  &\leq \sup_{p\in A_n}\int_{\set{|u|>\rho_n}}\fr{K(u/\gamma_n)}{\gamma_n}\deriv u \\
& \leq \int_{\set{|v|>\rho_n/\gamma_n}} K(v) \deriv v,
\end{align*}
where we used that $K>0$. As $\gamma_n\ll \rho_n$, and $K(v)\downarrow 0$ as $|v|\to \infty$, the right integral tends to zero. As Assumptions \ref{ass:propscoredist_estimated}\ref{ass:propscoredist_estimated:differentiability_p_and_theta} and \ref{ass:differentiability_regression} control $\sup_{p\in[\uphatth,\bphatth]}|q_\mu'(\hatth, p)|$, \eqref{eq:expected_qhat_1_2_2} is $\smallOP{1}$. Thus, \eqref{eq:hatq_prime_2} holds. Conclude that \eqref{eq:hatq_prime} holds.

\emph{Condition \eqref{eq:derivhat_p}, part \eqref{eq:hath_prime}.} Follows directly along the lines of \eqref{eq:hatq_prime}, setting $Y_j\ceq 1$ for all $j\in[n]$, in the formulae of \eqref{eq:hatq_prime}.

\emph{Condition \eqref{eq:derivhat_p}, part \eqref{eq:hatq}.} Analogously to $\hat q_\mu'$ above, we can write 
\begin{align*}
\hat q_\mu(\hatth,p) &= \fr{1}{n\gamma_n}\sum_{j\in[n]} Y_j (2\pi)^{-1}\int e^{-it(g(\hatth^\intercal X_j)-p)/\gamma_n}\psi(t)\deriv t \\
& = (2\pi)^{-1}\int \left(\nsumn{j} Y_j e^{-itg(\hatth^\intercal X_j)} \right) e^{itp} \psi(\gamma_n t)\deriv t,
\end{align*}
and then 
\begin{align*}
\Ebc{\sup_{p\in A_n}\abs{\hat q_\mu(\hatth,p)-\Ebc{\hat q_\mu(\hatth,p)}{\hatth}}}{\hatth}
 \leq \sqrt{\fr{\E Y_j^2}{\pi^2 n}}\int |\psi(\gamma_n t)|\deriv t \leq \sqrt{\fr{\E Y_j^2}{\pi^2 n \gamma_n^2}}\int |\psi(t)|\deriv t. 
\end{align*}
\Cref{ass:covar_outcome_distribution}\ref{ass:covar_outcome_distribution:outcome} and $\int |\psi(t)|\deriv t=\int e^{-t^2/2}=2\pi$ for the Gaussian kernel $K$ mean that the right side is of the order $1/(\gamma_n\sqrt{n})=(\kappa_0)^{-1}n^{\beta-1/2}=\smallO{1}$ for $\beta<1/4$. Next, as for $\hat q_\mu'$, 
\begin{align*}
\Ebc{\hat q_\mu(\hatth,p)}{\hatth} & = \fr{1}{\gamma_n} \int_{\uphatth}^{\bphatth} \mu(\hatth,\tld p)K((\tld p-p)/\gamma_n)h(\hatth,\tld p)\deriv\tld p \\
	&= \fr{1}{\gamma_n}\int_{\uphatth}^{\bphatth} q_\mu(\hatth,\tld p)K((\tld p-p)/\gamma_n)\deriv\tld p 
	= \int_{(\uphatth-p)/\gamma_n}^{(\bphatth-p)/\gamma_n} q_\mu(\hatth,\gamma_nv+p)K(v)\deriv v.
\end{align*}
This converges uniformly to $q_\mu(\hatth,p)$ in $p\in A_n$ by the same arguments as \eqref{eq:expected_qhat_2} does to $q_\mu'(\hatth,p)$, given \Cref{ass:propscoredist_estimated,ass:differentiability_regression} ensuring the boundedness and continuity of $q_\mu(\hatth,\cdot)$ with probability tending to one as $\hatth\convprob \theta_0$.

\textit{Condition \eqref{eq:derivhat_p}, part \eqref{eq:hath}.} Follows from \eqref{eq:hatq} by setting $Y_j\ceq 1$ for all $j\in[n]$.

%% Derivative w.r.t theta
\emph{Condition \eqref{eq:derivhat_theta}.} To establish the uniform convergence of 
\begin{align}
\widehat{\left(\fr{\partial \mu}{\partial \theta_k}\right)}(\hatth,p)\ceq \fr{\qdotkhat(\hatth,p)\hat h(\hatth,p)-\hat q_{\mu}(\hatth,p)\hdotkhat(\hatth,p)}{(\hat h(\hat\theta,p))^2}
\end{align}
to 
\begin{align}
\fr{\partial \mu}{\partial \theta_k}(\hatth,p)\ceq \fr{\qdotk(\hatth,p)h(\hatth,p)-\hat q_{\mu}(\hatth,p)\hdotk(\hatth,p)}{(\hat h(\hat\theta,p))^2}
\end{align}
in $p\in A_n$, we can follow the same steps which led to \eqref{eq:hatq_prime}--\eqref{eq:hath} of Condition \eqref{eq:derivhat_p}, because Assumptions \ref{ass:propscoredist_estimated}\ref{ass:propscoredist_estimated:differentiability_p_and_theta} and \ref{ass:differentiability_regression} ensure that with probability tending to one, $\norm{A_n}{h}$, $\norm{A_n}{\hdotk}$, $\norm{A_n}{q_\mu}$, $\norm{A_n}{\qdotk}$ are all bounded. As \eqref{eq:hatq} and \eqref{eq:hath} were proved above, it is sufficient to show both
\begin{align}
\norm{A_n}{\qdotkhat-\qdotk} &\convprob 0 \label{eq:hatq_dot}\quad\text{ and }\\
\norm{A_n}{\hdotkhat-\hdotk} &\convprob 0 \label{eq:hath_dot}.
\end{align}

\emph{Condition \eqref{eq:derivhat_theta}, part \eqref{eq:hatq_dot}.} We proceed by showing
\begin{align}
\E \Ebc{\sup_{p\in A_n}\abs{\qdotkhat(\hatth,p)-\Ebc{\qdotkhat(\hatth,p)}{\hatth}}}{\hatth} &\to 0 \quad\text{ and } \label{eq:hatq_dot_1} \\
\sup_{p\in A_n}\abs{\Ebc{\qdotkhat(\hatth,p)}{\hatth}-\qdotk(\hatth,p)} &\convprob 0. \label{eq:hatq_dot_2}
\end{align}
As for $\hat q_\mu'$ above, 
\begin{align}
\qdotkhat(\hatth,p) &= \fr{(\ginv)'(p)}{n\gamma_n^2}\sum_{j\in[n]}Y_jX_{j,k}(2\pi)^{-1}\int (-it)e^{-it (\hatth^\intercal X_j-\ginv(p))/\gamma_n}\psi(t)\deriv t \\
	&=(\ginv)'(p)(2\pi)^{-1}\int \left(\nsumn{j} Y_jX_{j,k}e^{-it\hatth^\intercal X_j}\right)e^{it\ginv(p)}(-it)\psi(\gamma_n t)\deriv t.
\end{align}
Since $|-it|\leq|t|$ and $\mathcal{X}$ is bounded by \Cref{ass:covar_outcome_distribution}, Euler's formula and \Cref{ass:theta_estimator}\ref{ass:theta_estimator:indep} imply that the bound of \eqref{eq:hatq_prime_1} also apply here up to a constant, which bounds $\mathcal{X}$, times $\sup_{p\in A_n}|(\ginv)'(p)|$. As $\sup_{p\in A_n}|(\ginv)'(p)|\leq 1/\supnorm{g'}<\infty$ by \Cref{ass:singleindex_propscore_smooth_outcome}, \eqref{eq:hatq_dot_1} holds.

\emph{Condition \eqref{eq:derivhat_theta}, part \eqref{eq:hatq_dot}, \eqref{eq:hatq_dot_2}.} We begin by deriving $\qdotk(\hatth,p)=((\partial /\partial \theta_k)(\mu h))(\hatth,p)$. For simplicity, assume that we only have two covariates, both continuous, and we are interested in the derivative with respect to the first coordinate of $\theta$ ($k\ceq 1$). It is straightforward to generalise the arguments below for the general case. By the tower property of expectation,
$$\mu(\theta,p)=\Ebc{Y}{\theta^\intercal X=\ginv(p)}=\Ebc{m(X)}{\theta^\intercal X=\ginv(p)}.$$
In view of \eqref{eq:exph_givent} of \Cref{prop:admissible_models_estimated_propscore}, we can write
\begin{align}
\Ebc{m(X)}{\theta^\intercal X=t} &= \fr{1}{v(\theta,t)} \int_{\mathcal{X}_1} m\left(x_1,\fr{t-\theta_1x_1}{\theta_2}\right)\Xdensity \left(x_1,\fr{t-\theta_1x_1}{\theta_2}\right)\deriv x_1 \label{eq:expectedm_condt} \\
v(\theta, t) &\ceq \int_{\mathcal{X}_1}\Xdensity \left(x_1,\fr{t-\theta_1x_1}{\theta_2}\right)\deriv x_1, \label{eq:v_density}
\end{align}
where $v(\theta,\cdot)$ is the density of $\theta^\intercal X$ satisfying $v(\theta,\cdot)>0$ by \Cref{ass:covar_outcome_distribution}. Thus, $h(\theta,\cdot)$, being the density of $g(\theta^\intercal X)$, is equal to
\begin{align}
h(\theta,p) = (\ginv)'(p)\int_{\mathcal{X}_1}\Xdensity \left(x_1,\fr{\ginv(p)-\theta_1x_1}{\theta_2}\right)\deriv x_1 \label{eq:h_formula}.
\end{align}
By \Cref{ass:theta_estimator}\ref{ass:theta_estimator:indep}, \eqref{eq:expectedm_condt} and \eqref{eq:h_formula} remain valid once we replace $\theta$ with $\hatth$. It then follows for continuously differentiable $\Xdensity$ and $m$ (\Cref{ass:singleindex_propscore_smooth_outcome,ass:covar_outcome_distribution}), that for $k=1$,
\begin{align}
q_\mu(\theta,p) &= \mu(\theta,p)h(\theta,p)  \nonumber \\
&=(\ginv)'(p) \int_{\mathcal{X}_1} m\left(x_1,\fr{\ginv(p)-\theta_1x_1}{\theta_2}\right)\Xdensity \left(x_1,\fr{\ginv(p)-\theta_1x_1}{\theta_2}\right)\deriv x_1 \nonumber \\
\qdotk(\hatth,p) &= (\ginv)'(p) \int_{\mathcal{X}_1}  \fr{\deriv }{\deriv \theta_1}\left[ m\left(x_1,\fr{\ginv(p)-\hatth_1x_1}{\hatth_2}\right)\Xdensity \left(x_1,\fr{\ginv(p)-\hatth_1x_1}{\hatth_2}\right)\right]\deriv x_1\nonumber \\
	&= \fr{\deriv }{\deriv p}\int_{\mathcal{X}_1} -x_1 m\left(x_1,\fr{\ginv(p)-\hatth_1x_1}{\hatth_2}\right)\Xdensity \left(x_1,\fr{\ginv(p)-\hatth_1x_1}{\hatth_2}\right)\deriv x_1, \label{eq:qdot_formula}
\end{align}
where in the last step we used that $(\ginv)'>0$. We proceed by showing the desired convergence. Let $[\uthatth, \bthatth]\ceq [\ginv(\uphatth), \ginv(\bphatth)]$. We have for $k=1$ by the tower property
\begin{align*}
\Ebc{\qdotkhat(\hatth,p)}{\hatth} = \fr{\ginv(p)}{\gamma_n^2} \Ebc{Y_j X_{j,k}K'((\hatth^\intercal X_j-\ginv(p)/\gamma_n)}{\hatth} \\
	= \fr{(\ginv)'(p)}{\gamma_n^2} \Ebc{\Ebc{m(X_j) X_{j,1}}{\hatth^\intercal X_j, \hatth}K'((\hatth^\intercal X_j-\ginv(p))/\gamma_n)}{\hatth} \\
	= \fr{(\ginv)'(p)}{\gamma_n^2} \int_{\uthatth}^{\bthatth} K'\left(\fr{t-\ginv(p)}{\gamma_n}\right) \int_{\mathcal{X}_1} x_1 m\left(x_1,\fr{t-\hatth_1x_1}{\hatth_2}\right)\Xdensity \left(x_1,\fr{t-\hatth_1x_1}{\hatth_2}\right)\deriv x_1\, \deriv t \\
	=(\ginv)'(p)\\
	\times \int_{\mathcal{X}_1} x_1 \left[\fr{1}{\gamma_n^{2}}\int_{\uthatth}^{\bthatth} K'\left(\fr{t-\ginv(p)}{\gamma_n}\right) m\left(x_1,\fr{t-\hatth_1x_1}{\hatth_2}\right)\Xdensity \left(x_1,\fr{t-\hatth_1x_1}{\hatth_2}\right)\deriv t \right] \deriv x_1,
\end{align*}
where we used that \eqref{eq:expectedm_condt} holds not only for $m$, but for any generic function of $X$ by \Cref{prop:admissible_models_estimated_propscore}. Let $\lambda_n(x_1, t)\ceq m\left(x_1,\fr{t-\hatth_1x_1}{\hatth_2}\right)\Xdensity \left(x_1,\fr{t-\hatth_1x_1}{\hatth_2}\right)$. We show that the term in the square brackets converges to 
$$-\fr{\partial }{\partial t} \lambda_n(x_1,\ginv(p))=-\left(\fr{\deriv }{\deriv p} \lambda_n(x_1,\ginv(p))\right)/[(\ginv)'(p)]$$ uniformly in $x_1\in\mathcal{X}_1$ and $p\in A_n$, which completes the proof for \eqref{eq:hatq_dot_2} in light of \eqref{eq:qdot_formula}. Put $\lambda_n'(x_1,t)\ceq \fr{\partial }{\partial t} \lambda_n(x_1,t)$. Integration by parts gives
\begin{align}
\fr{1}{\gamma_n^{2}}\int_{\uthatth}^{\bthatth} K'\left(\fr{t-\ginv(p)}{\gamma_n}\right)\lambda_n(x_1,t)\deriv t = \fr{1}{\gamma_n}\int_{(\uthatth-\ginv(p))/\gamma_n}^{(\bthatth-\ginv(p))/\gamma_n} K'(v)\lambda_n(x_1,\gamma_n v+ \ginv(p))\deriv t \nonumber \\
= \fr{1}{\gamma_n}\left\{\lambda_n(x_1, \bthatth)K((\bthatth-\ginv(p))/\gamma_n)- \lambda_n(x_1, \uthatth)K((\uthatth-\ginv(p))/\gamma_n)\right\} \label{eq:lambda_1} \\
- \int_{(\uthatth-\ginv(p))/\gamma_n}^{(\bthatth-\ginv(p))/\gamma_n} \lambda_n'(x_1,\gamma_n v+\ginv(p))K(v) \deriv v.  \label{eq:lambda_2}
\end{align}
Now $\lambda_n$ is bounded with probability tending to one by \Cref{ass:covar_outcome_distribution}. Recall that $\uthatth=\ginv(\uphatth), \bthatth=\ginv(\bphatth)$. As $\ginv$ is increasing, and $K$ reaches its maximum at zero and satisfies $\lim_{|u|\to\infty}K(u)\to0$, \eqref{eq:lambda_1} is of the order
\begin{align}
\gamma_n^{-1}\left[\sup_{p\in A_n}K\left(\fr{\ginv(\bphatth)-\ginv(p)}{\gamma_n}\right)+ \sup_{p\in A_n}K\left(\fr{\ginv(\uphatth)-\ginv(p)}{\gamma_n}\right) \right] \nonumber \\
= \gamma_n^{-1}\left[K\left(\fr{\ginv(\bphatth)-\ginv(\bphatth-a_n)}{\gamma_n}\right)+ K\left(\fr{\ginv(\uphatth)-\ginv(\uphatth+a_n)}{\gamma_n}\right) \right]. \label{eq:lamdba_1_orderbound}
\end{align}
By a mean-value expansion, the first term in \eqref{eq:lamdba_1_orderbound} is of the order
$$\gamma_n^{-1}K\left(\left(\inf_{p\in[\uphatth,\bphatth]}(\ginv)'(p)\right)a_n/\gamma_n\right).$$ 
As $\inf_{p\in[\uphatth,\bphatth]}(\ginv)'(p)>0$, this is $\smallOP{1}$ for $a_n/\gamma_n\to\infty$ and Gaussian $K$. The same applies to the second term of \eqref{eq:lamdba_1_orderbound}, so \eqref{eq:lambda_1} is $\smallOP{1}$. Last, we show that \eqref{eq:lambda_2} converges to $-\lambda_n'(x_1,\ginv(p))$. Again, we can write
\begin{align*}
\lambda_n'(x_1,\ginv(p)) =  \fr{\lambda_n'(x_1,\ginv(p))}{\gamma_n} \\
\times  \left\{ \int_{-\infty}^{\ginv(p)-\bthatth}  K(u/\gamma_n)\deriv u +  \int_{\ginv(p)-\bthatth}^{p-\uthatth}  K(u/\gamma_n)\deriv u + \int_{\ginv(p)-\uthatth}^\infty K(u/\gamma_n)\deriv u \right\}.
\end{align*}
A change of variables in \eqref{eq:lambda_2} ($u\ceq-\gamma_n v$) and $K$ being symmetric about zero give
\begin{align}
\sup_{x_1\in\mathcal{X}_1, p\in A_n} \abs{\int_{(\uthatth-\ginv(p))/\gamma_n}^{(\bthatth-\ginv(p))/\gamma_n} \lambda_n'(x_1,\gamma_n v+\ginv(p))K(v) \deriv v - \lambda_n'(x_1,\ginv(p))} \nonumber \\
\leq \sup_{x_1\in\mathcal{X}_1, p\in A_n} \abs{\int_{-\infty}^{\ginv(p)-\bthatth}\lambda_n'(x_1,\ginv(p))\gamma_n^{-1}K(u/\gamma_n)\deriv u} \label{eq:lambda_2_1_1} \\
+\sup_{x_1\in\mathcal{X}_1, p\in A_n} \abs{\int_{\ginv(p)-\uthatth}^\infty \lambda_n'(x_1,\ginv(p))\gamma_n^{-1}K(u/\gamma_n)\deriv u} \label{eq:lambda_2_1_2} \\
+\sup_{x_1\in\mathcal{X}_1, p\in A_n} \abs{\int_{\ginv(p)-\bthatth}^{\ginv(p)-\uthatth} [\lambda_n'(x_1,\ginv(p)-u)-\lambda_n'(x_1,\ginv(p))]\gamma_n^{-1}K(u/\gamma_n)\deriv u} \label{eq:lambda_2_2}.
\end{align}
Since $\mathcal{X}_1$ is bounded by \Cref{ass:covar_outcome_distribution},  $\sup_{x_1\in\mathcal{X}_1, p\in A_n}|\lambda_n'(x_1,\ginv(p))|$ is bounded with probability tending to one. Hence, as $\ginv$ is increasing, $K\geq0$ and $\bthatth=\ginv(\bphatth)$,  \eqref{eq:lambda_2_1_1} is of the order $\int_{-\infty}^{(\ginv(\bphatth-a_n)-\ginv(p))/\gamma_n}K(v)\deriv v$. Here, 
$$(\ginv(\bphatth-a_n)-\ginv(p))/\gamma_n\leq (\inf_{p\in[0,1]}(\ginv)'(p))(-a_n/\gamma_n)\leq (1/\supnorm{g'})(-a_n/\gamma_n)$$
 by a mean-value expansion. As the infimum is positive, \eqref{eq:lambda_2_1_1} is $\smallOP{1}$ for $a_n/\gamma_n\to\infty$, and so is \eqref{eq:lambda_2_1_2} by similar arguments. The term \eqref{eq:lambda_2_2} can be treated analogously to \eqref{eq:expected_qhat_1_2}. First note that taking the supremum over $p\in A_n$ is the same as taking the supremum over $t\in[\ginv(\uphatth+a_n), \ginv(\bphatth-a_n)]\eqc T_n$ as $\ginv$ is increasing. Take a $\gamma_n\ll \rho_n<(1/\supnorm{g'})a_n$ (e.g. $\rho_n\ceq (1/\supnorm{g'})\kappa_2 n^{-(\alpha+\beta)/2}$ for some $0<\kappa_2<\kappa_1$). Then \eqref{eq:lambda_2_2} is bounded by
\begin{align}
\sup_{x_1\in\mathcal{X}_1, t\in T_n} \abs{\int_{[t-\bthatth, t-\uthatth]\cap\set{|u|\leq\rho_n}} [\lambda_n'(x_1,t-u)-\lambda_n'(x_1,t)]\gamma_n^{-1}K(u/\gamma_n)\deriv u} \label{eq:lambda_2_2_1} \\
\sup_{x_1\in\mathcal{X}_1, t\in T_n} \abs{\int_{[t-\bthatth, t-\uthatth]\cap\set{|u|>\rho_n}} [\lambda_n'(x_1,t-u)-\lambda_n'(x_1,t))]\gamma_n^{-1}K(u/\gamma_n)\deriv u}.  \label{eq:lambda_2_2_2}
\end{align}
Here, \eqref{eq:lambda_2_2_1} is bounded by 
\begin{align*}
\sup_{x_1\in\mathcal{X}_1,t\in T_n}\sup_{|u|\leq\rho_n}|\lambda_n'(x_1, t-u)-\lambda_n'(x_1,t)|\sup_{x_1\in\mathcal{X}_1,t\in T_n}\int_{-\infty}^\infty \gamma_n^{-1}K(u/\gamma_n)\deriv u\\
\leq \sup_{x_1\in\mathcal{X}_1,t\in T_n}\sup_{|u|\leq\rho_n}|\lambda_n'(x_1, t-u)-\lambda_n'(x_1,t)|.
\end{align*}
By \Cref{ass:covar_outcome_distribution,ass:singleindex_propscore_smooth_outcome}, $\lambda_n'(x_1,\cdot)$ is continuous uniformly in $x_1$ and is therefore uniformly continuous on the compact set $[\ginv(\uphatth), \ginv(\bphatth)]\supset [\ginv(\uphatth+a_n), \ginv(\bphatth-a_n)]$. Thus, $|u|\leq\rho_n$ for a small enough $\rho_n$ implies that $t-u,t\in [\ginv(\uphatth), \ginv(\bphatth)]$ for all $t\in T_n$. (Note that $\rho_n$ is small enough if and only if $|u|\leq\rho_n$ implies 
$$|u|\leq \min\set{\ginv(\uphatth+a_n)-\ginv(\uphatth), \ginv(\bphatth)-\ginv(\bphatth-a_n)}.$$ 
By a mean-value expansion, the right side is smaller than or equal to $\inf_{p\in[0,1]}(\ginv)'(p)a_n\leq (1/\supnorm{g'})a_n$. But then $\rho_n<(1/\supnorm{g'})a_n$ is small enough.) As a consequence, 
$$\sup_{x\in\mathcal{X}_1,t\in T_n}\sup_{|u|\leq\rho_n}|\lambda_n'(x_1, t-u)-\lambda_n'(x_1,t)|=\smallOP{1},$$ hence \eqref{eq:lambda_2_2_1} is $\smallOP{1}$. As $2\sup_{x_\in\mathcal{X}_1, t\in[\uthatth,\bthatth]}|\lambda_n(x_1,t)|$ is bounded by probability tending to one as $\hatth_2\convprob \theta_{0,2}\neq 0$ by \Cref{ass:covar_outcome_distribution,ass:singleindex_propscore_smooth_outcome}, \eqref{eq:lambda_2_2_2} can be shown to be $\smallOP{1}$ similarly to \eqref{eq:expected_qhat_1_2_2}. Thus, \eqref{eq:lambda_2_2} is $\smallOP{1}$ which shows the desired convergence of \eqref{eq:lambda_2}. Hence, \eqref{eq:hatq_dot_2} holds for $k=1$. The case for $k=2$ follows by writing 
\begin{align*}
\Ebc{m(X)}{\theta^\intercal X=t} &= (1/v(\theta,t)) \int_{\mathcal{X}_2} m\left(\fr{t-\theta_2x_2}{\theta_1},x_2\right)\Xdensity \left(\fr{t-\theta_2x_2}{\theta_1},x_2\right)\deriv x_2 \\
v(\theta, t) &\ceq \int_{\mathcal{X}_2}\Xdensity \left(\fr{t-\theta_2x_2}{\theta_1},x_2\right)\deriv x_2
\end{align*}
instead of \eqref{eq:expectedm_condt} and \eqref{eq:v_density}. By \Cref{ass:covar_outcome_distribution}, these behave equally well. Conclude that the desired convergence of $\qdotkhat$ (Condition \eqref{eq:derivhat_theta}, part \eqref{eq:hatq_dot}) holds.

\emph{Condition \eqref{eq:derivhat_theta}, part \eqref{eq:hath_dot}.}\, Follows along the same lines as \eqref{eq:hatq_dot} by setting $Y_j\ceq 1$ for all $j\in[n]$. Conclude that Condition \eqref{eq:derivhat_theta} holds.

\emph{Remaining Uniform Consistency Results.} Showing
\begin{align}
\sup_{p\in A_n}|\hat\mu(\hatth,p)-\mu(\hatth,p)|\convprob0, \label{eq:muhat} \\
\sup_{p\in A_n}|\hat\mu_2(\hatth,p)-\mu_2(\hatth,p)|\convprob0 \label{eq:mu2hat}
\end{align}
completes the proof of \Cref{prop:consistent_variance}. As $h(\hat\theta,\cdot)$ is bounded away from zero with probability tending to one by \Cref{ass:propscoredist_estimated}, \eqref{eq:muhat} is implied by \eqref{eq:hatq} and \eqref{eq:hath}. Likewise, \eqref{eq:mu2hat} is implied by $\norm{A_n}{\hat q_{\mu_2}-q_{\mu_2}}\convprob 0$, which can be shown as \eqref{eq:hatq}.
\end{proof}

\end{appendix}

\newpage
{\scriptsize%
\bibliography{../idea1.bib}
\bibliographystyle{plainnat}
}

\end{document}